\documentclass[1 leqno,11pt,psfig]{amsart}
\usepackage{amssymb, amsmath,amsmath,latexsym,amssymb,amsfonts,amsbsy, amsthm,mathtools,graphicx,CJKutf8,CJKnumb,CJKulem,color}
\usepackage{graphicx,epsfig}
\usepackage{graphicx}
\usepackage{amssymb}
\usepackage{graphicx}
\usepackage{graphicx,epsfig}
\usepackage{amsmath}
\usepackage{mathrsfs}
\usepackage{amssymb}

\usepackage{amssymb,mathrsfs,amsfonts,amsmath,amsbsy,tikz}\usepackage{fontenc}\usepackage{textcomp}

\usepackage{amssymb,  amsmath, amsmath, latexsym, amssymb, amsfonts, amsbsy, amsthm,mathtools, graphicx, CJKutf8, CJKnumb, CJKulem, extarrows, color, dsfont,verbatim}

\numberwithin{equation}{section}

\allowdisplaybreaks

\let\pa=\partial

\def\curl{\mathop{\rm curl}\nolimits}

\newcommand{\beq}{\begin{equation}}
\newcommand{\eeq}{\end{equation}}
\newcommand{\ben}{\begin{eqnarray}}
\newcommand{\een}{\end{eqnarray}}
\newcommand{\beno}{\begin{eqnarray*}}
\newcommand{\eeno}{\end{eqnarray*}}

\newtheorem{Theorem}{Theorem}[section]

\newtheorem{lemma}[Theorem]{Lemma}
\newtheorem{proposition}[Theorem]{Proposition}

\newtheorem{remark}[Theorem]{Remark}

\newtheorem{Proposition}[Theorem]{Proposition}
\newtheorem{Corollary}[Theorem]{Corollary}

\begin{document}
\begin{CJK*}{GBK}{song}
\title[Dynamics near Couette flow for the $\beta$-plane equation]{\textbf {Dynamics near Couette flow for the $\beta$-plane equation }}

\author{Luqi Wang}
\address{School of Mathematical Science, Peking University, 100871, Beijing, P. R. China}
\email{wangluqi@pku.edu.cn}

\author{Zhifei Zhang}
\address{School of Mathematical Science, Peking University, 100871, Beijing, P. R. China}
\email{zfzhang@math.pku.edu.cn}

\author{Hao Zhu}
\address{Department of Mathematics, Nanjing University,  210093, Nanjing, Jiangsu, P. R. China}
\email{haozhu@nju.edu.cn}

\date{\today}

\maketitle
\begin{abstract}
 In this paper, we study stationary structures near the planar Couette flow in Sobolev spaces on a channel $\mathbb{T}\times[-1,1]$, and asymptotic behavior of Couette flow in Gevrey spaces on $\mathbb{T}\times\mathbb{R}$ for the $\beta$-plane equation. Let $T>0$ be the horizontal period of the channel and $\alpha={2\pi\over T}$ be the wave number. We obtain a sharp region $O$ in the whole $(\alpha,\beta)$ half-plane such that non-parallel steadily traveling waves do not exist for $(\alpha,\beta)\in O$ and such traveling waves exist for $(\alpha,\beta)$ in the remaining regions, near Couette flow for $H^{\geq5}$ velocity perturbation. The borderlines between the region $O$ and its remaining are determined by two curves of the principal eigenvalues of singular Rayleigh-Kuo operators. Our results reveal that there exists $\beta_*>0$ such that if $|\beta|\leq \beta_*$, then non-parallel traveling waves do not exist for any $T>0$, while if $|\beta|>\beta_*$, then there exists a critical period $T_\beta>0$ so that such traveling waves exist for $T\in \left[T_\beta,\infty\right)$ and do not exist for $T\in \left(0,T_\beta\right)$, near Couette flow for $H^{\geq5}$ velocity perturbation. This contrasting dynamics plays an important role in studying the long time dynamics near Couette flow with Coriolis effects. Moreover, for any $\beta\neq0$ and $T>0$,  there exist no non-parallel traveling waves with speeds converging in $(-1,1)$ near Couette flow for $H^{\geq5}$ velocity perturbation, in contrast to this, we construct non-shear stationary solutions near Couette flow for $H^{<{5\over2}}$ velocity perturbation, which is a generalization of Theorem 1 in [22] but the construction is more difficult due to the $\beta$'s term. Finally, we prove nonlinear inviscid damping for Couette flow in some Gevrey spaces by extending the method of [4] to the $\beta$-plane equation on $\mathbb{T}\times\mathbb{R}$.
\end{abstract}
\section{Introduction}

Dynamics of oceans and planetary atmospheres is one of the central topics in geophysical fluid dynamics.
In the study of such large-scale motion in a rotating frame, it is reasonable to include  Coriolis force  to be geophysically relevant.
A common model for large-scale motion is described by
the $\beta$-plane equation
\ben\label{eq}
\pa_t \vec{v}+(\vec{v}\cdot \nabla)\vec{v}=-\nabla P-\beta y J \vec{v}, \quad  \nabla\cdot \vec{v}=0,
\een
where $\vec{v}=(u,v)$ is the fluid velocity, $P$ is the pressure,
\beno
J=\left(
\begin{aligned}
0&&-1\\
1&&0
\end{aligned}
\right)
\eeno
is the rotation matrix, and $\beta$ is the Coriolis parameter.
Then the vorticity  $\omega=\curl \vec{v}=\pa_x v-\pa_y u$ solves
\ben\label{vor-eq}
\pa_t \omega+(\vec{v}\cdot \nabla)\omega+\beta v=0.
\een
We will work on the domain
$
D_T=\mathbb{T}\times [-1,1]
$
with non-permeable boundary condition
\ben\label{bc}
v=0 \quad \textrm{on} \quad y=\pm 1,
\een
where $\mathbb{T}=\mathbb{R}/(T\mathbb{Z})$.

A shear flow is a steady solution of \eqref{vor-eq}. The Couette flow $(y,0)$ is one of the simplest laminar flows.  We are interested in the long time dynamics near Couette flow for the $\beta$-plane equation. For $\beta=0$, it is known that  nonlinear inviscid damping is true if the perturbation is taken in a suitable Gevrey space \cite{BM,Ionescu-Jia20}.
In the Sobolev space $H^{<{5\over2}}$ for velocity perturbation,
richer dynamics around Couette flow was found by constructing non-shear steady states (and traveling waves) near Couette, and on the other hand, relatively simpler dynamics for $H^{>{5\over2}}$ velocity perturbation was obtained by proving non-existence of non-parallel steadily traveling waves \cite{LZ, CL}.

In this paper, we study whether similar results are true in Gevrey and Sobolev spaces respectively for $\beta\neq0$.
Our main results roughly state that if the perturbation is taken in a suitable Gevrey space or in the (velocity) Sobolev space $H^{<{5\over2}}$, similar results are true for $\beta\neq0$, while if the velocity perturbation is considered in the Sobolev space $H^{\geq{5}}$,
it turns out that the situation is very different from the case $\beta=0$. In fact, there exists  $\beta_*>0$ such that similar results are still  true for $0<|\beta|\leq\beta_*$. The difference is for the case $|\beta|>\beta_*$, namely, there exists a critical period  $T_\beta>0$ such that   traveling waves always exist for $T\in[T_\beta,\infty)$ near Couette flow  no matter how much regularity is required  and traveling waves do not exist for $T\in(0,T_\beta)$, where $T_\beta={2\pi\over\alpha_\beta}$ is given in Theorem \ref{thm-non-shear-nodelta2}.

In the $\beta$-plane model, the $\beta$'s term  brings some fundamental changes to the internal structure of \eqref{vor-eq} near a shear flow on $
\mathbb{T}\times [-1,1]
$. This induces  new long time dynamical behavior around a shear flow, which is  useful in understanding  the various  large-scale physical  phenomena in atmospheres and oceans. Let us now explain how the $\beta$'s term in \eqref{vor-eq}  influences the spectrum of the linearized operator  and  dynamics around a shear flow.
By the incompressible condition, we can introduce
 the stream function $\psi$ such that $\vec{v}=(\pa_y\psi,-\pa_x\psi)$. The
linearized equation of \eqref{vor-eq} around a shear flow $(u(y),0)$ is
\begin{align*}
\partial_{t}\Delta\psi+u\partial_{x}\Delta\psi+(\beta-u^{\prime\prime})\partial_x\psi=0.
\end{align*}
Taking Fourier transform in $x$, we have
$
 (\partial^2_y-\alpha^2)\partial_t\widehat{\psi}=i\alpha((u''-\beta)-u(\partial^2_y-\alpha^2))\widehat{\psi},
$
 where $\alpha={2\pi\over T}$ is the wave number.
For  $\beta\in \mathbb{R}$,  the linearized  operator is given by
\begin{align}\label{linearized Euler operator}
\mathcal{R}_{\alpha,\beta}\widehat{\psi}:=-(\partial^2_y-\alpha^2)^{-1}((u''-\beta)-u(\partial^2_y-\alpha^2))\widehat{\psi}.
\end{align}
 The essential spectrum $\sigma_e(\mathcal{R}_{\alpha,\beta})=\text{Ran} (u)$, and
  $c\in \sigma_d(\mathcal{R}_{\alpha,\beta})$ (the discrete spectrum) if and only if its corresponding eigenfunction $\psi_c$ satisfies the
Rayleigh-Kuo boundary value problem (BVP):
\begin{align}\label{sturm-Liouville}
\mathcal{L}_{\alpha,\beta}\phi:=-\phi''+{u''-\beta\over u-c}\phi=\lambda\phi, \;\;\;\;\phi(\pm1)=0,
\end{align}
where $\phi\in H_0^1\cap H^2(-1,1)$ and $\lambda=-\alpha^2$.
  An important difference is that
\begin{align*}&\sigma_d(\mathcal{R}_{\alpha,\beta})\cap\mathbb{R}=\emptyset \text{ if }\beta=0,\\
 &\sigma_d(\mathcal{R}_{\alpha,\beta})\cap\mathbb{R}\neq\emptyset \text{ if }\beta\neq0
 \end{align*}
 in general \cite{LYZ}. For example, if $|\beta|$ is sufficiently large, then  $\sigma_d(\mathcal{R}_{\alpha,\beta})\cap\mathbb{R}\neq\emptyset$ for  some wave numbers. Reflected on the dynamical behavior near the shear flow,
 this  difference of the linearized operators' spectrum
  brings new non-parallel steady traveling wave families, with traveling
speeds converging outside the range of the flow, near the shear flow for $\beta\neq0$, while
 no such traveling wave families exist for $\beta=0$ \cite{LWZZ}. This implies that the  long time dynamics near a shear flow is richer for $\beta\neq0$.  In fact, the long time dynamics near a shear flow  might be much complicated due to the  $\beta$'s term. Taking Sinus flow for example,
there are
infinitely many such traveling wave families  for $\beta<-{1\over 2}\pi^2$ or $\beta>{9\over 16}\pi^2$ and any horizontal period  \cite{LWZZ}.


Lyapunov stability  is a classical issue    in the context of hydrodynamics for  general stationary flows.
 For shear flows,
 Rayleigh \cite{Rayleigh}   proved that   a necessary condition for linear instability is that $u$ has an inflection point for $\beta=0$. Kuo \cite{Kuo} extended  the necessary condition to $\beta\neq0$ that  $\beta-u^{\prime\prime}$ must change sign.
 Howard \cite{Howard1961} proved that the unstable eigenvalues must lie in a semicircle region, which is called the Howard semicircle theorem for $\beta=0$. Pedlosky \cite{Pedlosky1963} extended  the radius of the  semicircle  by
${|\beta|\over\alpha^2}$ for $\beta \neq 0$.
 By introducing the energy-Casimir functional,
Arnol'd proved nonlinear Lyapunov stability for a class of   stationary flow by  showing that it  is a minimizer or a maximizer of the functional \cite{Arnold1,Arnold2}.  The so-called  energy-Casimir method has been developed in \cite{Wolansky-Ghil,Lin2} and also extended to many other physical models,  such as
the quasi-geostrophic equations for planetary-scale rotating flows \cite{Blumen,Sakuma-Ghil}.
  The index formula developed by Lin and Zeng \cite{Lin-Zeng} provides a useful tool to study the sufficient conditions for linear instability of stationary flows, as well as to count the number of unstable modes, if the linearized equation has a Hamiltonian structure.

Our concern is the long time dynamics around Couette flow, and  this problem  has become  a topic of interest for $\beta=0$ and attracted the attention of many mathematicians.  In the earlier works,
Kelvin \cite{Kelvin1887} gave the  construction of exact solutions to the linearized problem near Couette flow.
Orr \cite{Orr} observed the decay of velocity for the linearized equation around Couette flow.
Lin and Zeng \cite{LZ} confirmed the linear inviscid damping around Couette flow for $L^2$ vorticity perturbation.
For the nonlinear equation, as mentioned  above, they found Cat's eyes flow  near Couette   for $H^{<{3\over2}}$ vorticity perturbation, and proved that  non-parallel steadily  traveling waves do not exist for $H^{>{3\over2}}$ perturbation. Similar results were obtained for the Vlasov-Poisson system \cite{LZ2}, and instability in high Sobolev spaces was established in \cite{Bedrossian1}.
Castro and Lear \cite{CL} showed the existence of nontrivial and smooth traveling waves close to Couette flow for  $H^{<\frac{3}{2}}$ vorticity perturbation  with speed of order 1.
As for the perturbation in Gevrey spaces,  Bedrossian and Masmoudi \cite{BM} proved nonlinear inviscid damping around the Couette flow  in Gevrey class $\mathcal{G}^{\lambda,1/2+}$ on $\mathbb{T}\times\mathbb{R}$. Deng and Masmoudi \cite{DM} showed that this is the critical regularity by proving the instability in Gevrey class $\mathcal{G}^{\lambda,1/2-}$.
Ionescu and  Jia \cite{Ionescu-Jia20} proved nonlinear inviscid damping  in a channel $\mathbb{T}\times[-1,1]$ under the compacted support's assumption on the initial vorticity perturbation. We also point out  some important progress on linear inviscid damping for general monotone and non-monotone  shear flows  in \cite{Zillinger2016,Zillinger2017,WZZ-mono, WZZ-shear, WZZ-Kol,Jia2020-1,Jia2020-2,GNRS}, on nonlinear inviscid damping for monotone shear flows in \cite{Ionescu-Jia2001,Masmoudi-Zhao2001}, and on nontrivial invariant structures near Kolmogorov flow and Poiseuille flow in various domains in \cite{ZEW}. It is still challenging to prove nonlinear damping for non-monotone flows.

We now turn back to   the case   $\beta\neq0$.
For  the linearized equation,
Lin, Yang and the third author gave a method to study the linear Lyapunov instability for a class of shear flows based on Hamiltonian systems and spectral analysis  of ODEs, see Subsection 3.3 in \cite{LYZ}.
Then Wei and the last two authors  gave the explicit decay rate of the velocity for a class of monotone shear flows based on  the space-time
estimate and the vector field method, as well as  proved the linear damping
for a class of general shear flows under some spectral conditions, see Theorems 1.1 and 1.2 in \cite{WZZhu}.
For  the nonlinear equation, Lin, Wei and the last two authors found some richer dynamics near a class of shear flows based on asymptotic behavior of spectrum of Rayleigh-Kuo BVP and bifurcation theory of nonlinear maps. More precisely, we proved that
 if the flow $u$ has a critical point at which $u$ attains
its minimal value, then there exists a unique  $\beta_+$  in the positive
 half-line such that the number of traveling wave families near the shear flow
changes suddenly from finite one to infinity when $\beta$ passes through it. On the other hand, if
$u$ has no such critical points, then the number is always at most  finite for positive $\beta$
values. A similar result holds true  in the negative half-line,  see Theorems 1.2 and 2.1 in \cite{LWZZ}. Here, the traveling speeds lie outside the range of the flow.
 Elgindi, Pusateri and Widmayer  took  the advantage of  dispersive operator induced by the Coriolis effect  and proved the
stability of the zero solution for  $\beta\neq0$ in \cite{EW, PW}.


In this paper, we are interested in  the  long time nonlinear  dynamics around Couette flow for the case   $\beta\neq0$.
At first, we consider the perturbation in the Sobolev spaces.
 By Corollaries 2.2 and 2.4 in \cite{LWZZ},
it is known that there are at most finitely many traveling
wave families  near Couette flow for $H^{\geq3}$ velocity perturbation, where the traveling
speeds  converge outside the range of the  flow.
It is necessary to clarify whether
``at most finitely" implies
existence or not.
Furthermore,  if we try to understand the nontrivial stationary  structures  in any reference frame near Couette flow deeply, then  two  questions naturally arise:\smallskip

{\bf Q1.} {\it Are there  traveling waves with the traveling speeds $c$ lying  inside the range of Couette flow $[-1,1]$ near  the flow for $H^{\geq 3}$ velocity perturbation ? This is more delicate  than the case that traveling speeds lie outside $[-1,1]$, since we have to deal with the singularity in the terms involved with the factor like ``${1\over y-c}"$.}

{\bf Q2}. {\it By the method of Theorem $2$ in  \cite{LZ}, a conclusion can be essentially obtained    as follows:
for any $\delta>0$, there exists $\beta_\delta>0$ small enough such that if $|\beta|<\beta_\delta$ and the horizontal period is arbitrary, then there exist no  traveling waves near Couette flow    for $H^{\geq {5\over2}+\delta}$ velocity perturbation.
By Corollary $2.4$ in \cite{LWZZ}, if $|\beta|\gg1$, there exist  traveling waves near Couette flow for $H^{\geq3}$ velocity perturbation and for some period in $x$. Consider   $H^{\geq s_0}$ velocity perturbation near Couette flow for some $s_0\geq3$.
For any fixed  $\beta\neq0$, can we determine
 for which periods there exist  traveling waves  near  Couette flow, and for other periods  there exist no traveling waves  near the flow ?
 For any fixed horizontal period $T>0$, can we determine for which $\beta$ values there exist traveling waves  near  Couette flow,
  and for other $\beta$ values there exist no traveling waves  near  the flow ?  Here, the traveling waves always mean the non-parallel steadily ones.}


\smallskip

To answer {\bf Q1}, we have the following result.

\begin{Theorem}\label{thm-non-shear}
Let $\beta\neq0$. For any $T>0$, $s\geq5$ and $0<\delta<1$,  there exists $\varepsilon_\delta>0$ such that any traveling wave solution $(u(x-ct,y),v(x-ct,y))$  to the $\beta$-plane equation \eqref{eq}-\eqref{bc} with $c\in[-1+\delta,1-\delta]$, $x$-period $T$ and satisfying that
\beno
\|(u,v)-(y,0)\|_{H^{s}{(D_T)}}<\varepsilon_\delta,
\eeno
must have $v(x,y)\equiv 0$, that is, $(u,v)$ is necessarily a shear flow.
\end{Theorem}

\begin{remark}
$[-1+\delta,1-\delta]$ in Theorem $\ref{thm-non-shear}$ can not be extended to the whole range $[-1,1]$, since there exist traveling waves with traveling speeds converging to $\pm1$,   see Theorem $\ref{thm-non-shear-nodelta2}\; (2\rm{i})$ for  $|\beta|>\beta_*$ and $T={T_\beta}$.
\end{remark}

 Let
\begin{align}\label{lambda-beta}
\lambda_1(\beta,-1)=\inf_{\phi\in H_0^1(-1,1),\|\phi\|_{L^2(-1,1)}=1}\int_{-1}^1\left(|\phi'|^2-{\beta\over y+1}|\phi|^2\right)dy
\end{align}
be the principal eigenvalue of the singular Rayleigh-Kuo BVP \eqref{sturm-Liouville} with $u(y)=y$, $c=-1$. The properties of $\lambda_1(\beta,-1)$ is given in Section 3. Next, we give a positive answer to {\bf Q2}.
\begin{Theorem}\label{thm-non-shear-nodelta2}
Let $s\geq5$ and $\beta_*>0$ be the unique point such that $\lambda_1(\beta,-1)=0$, where $\lambda_1(\beta,-1)$ is defined in \eqref{lambda-beta}. 

$(1)$ Let $0<|\beta|\leq\beta_*$. For any $T>0$, there exists $\varepsilon_0>0$ such that any traveling wave solution  $(u(x-ct,y),v(x-ct,y))$ to the $\beta$-plane equation $\eqref{eq}$-$\eqref{bc}$ with $c\in\mathbb{R}$, $x$-period $T$ and satisfying that
\beno
\|(u,v)-(y,0)\|_{H^{s}{(D_T)}}<\varepsilon_0,
\eeno
must have $v(x,y)\equiv 0$, that is, $(u,v)$ is necessarily a shear flow.

$(2)$
Let $|\beta|>\beta_*$, $\alpha_\beta=\sqrt{-\lambda_1(|\beta|,-1)}>0$ and $T_\beta={2\pi\over \alpha_\beta}$.
 \begin{itemize}
 \item[(2i)]
Fix $T\in\left[T_\beta,\infty\right)$. Then  for any $\varepsilon>0$, there exists a traveling wave solution  $(u_\varepsilon(x-c_\varepsilon t,y),v_\varepsilon(x-c_\varepsilon t,y))$ to the $\beta$-plane equation $\eqref{eq}$-$\eqref{bc}$ with $x$-period $T$ and satisfying that
\beno
\|(u_\varepsilon,v_\varepsilon)-(y,0)\|_{H^{s}{(D_T)}}<\varepsilon,
\eeno
but $v_\varepsilon(x,y)\not\equiv 0$. As $\varepsilon\to0,$  we have $c_\varepsilon\to-1$ for $\beta>\beta_*$ and $T=T_\beta$;
$c_\varepsilon\to c_0\in(-\infty,-1)$ for $\beta>\beta_*$ and $T>T_\beta$; $c_\varepsilon\to1$ for $\beta<-\beta_*$ and $T=T_\beta$;  $c_\varepsilon\to c_0\in(1,\infty)$ for $\beta<-\beta_*$ and $T>T_\beta$.
\item[(2ii)]
Fix $T\in\left(0,T_\beta\right)$. Then similar conclusion  in $(1)$ holds true.
\end{itemize}
Moreover, $\alpha_\beta$ is continuous and increasing on $\beta\in[\beta_*,\infty)$, and $\alpha_\beta\to\infty$ as $\beta\to\infty$.
\end{Theorem}

\begin{center}
 \begin{tikzpicture}[scale=0.58]
 \draw [->](-11, 0)--(11, 0)node[right]{$\beta$};
 \draw [->](0,0)--(0,9) node[above]{$\alpha$};
 \draw (4, 0).. controls (5, 2) and (8, 4)..(10.5,7.5);
\draw (-4, 0).. controls (-5, 2) and (-8, 4)..(-10.5,7.5);
\path (4.5, 0)  edge [-,dotted](5.65,2.15) [line width=0.8pt];
 \path (5, 0)  edge [-,dotted](6.25,2.75) [line width=0.8pt];
  \path (5.5, 0)  edge [-,dotted](6.75,3.25) [line width=0.8pt];
  \path (6, 0)  edge [-,dotted](7.25,3.75) [line width=0.8pt];
  \path (6.5, 0)  edge [-,dotted](7.75,4.25) [line width=0.8pt];
    \path (7, 0)  edge [-,dotted](8.25,4.75) [line width=0.8pt];
    \path (7.5, 0)  edge [-,dotted](8.4,3.852) [line width=0.8pt];
    \path (8.5,4.28)  edge [-,dotted](8.75,5.35) [line width=0.8pt];
     \path (8, 0)  edge [-,dotted](9.25,5.9) [line width=0.8pt];
     \path (8.5, 0)  edge [-,dotted](9.75,6.5) [line width=0.8pt];
      \path (9, 0)  edge [-,dotted](10.25,7.15) [line width=0.8pt];
      \path (9.5, 0)  edge [-,dotted](10.5,7.48) [line width=0.8pt];
      \path (-4.5, 0)  edge [-,dotted](-5.65,2.15) [line width=0.8pt];
      \path (-5, 0)  edge [-,dotted](-6.25,2.75) [line width=0.8pt];
       \path (-5.5, 0)  edge [-,dotted](-6.75,3.25) [line width=0.8pt];
  \path (-6, 0)  edge [-,dotted](-7.25,3.75) [line width=0.8pt];
  \path (-6.5, 0)  edge [-,dotted](-7.75,4.25) [line width=0.8pt];
    \path (-7, 0)  edge [-,dotted](-8.25,4.75) [line width=0.8pt];
    \path (-7.5, 0)  edge [-,dotted](-8.4,3.852) [line width=0.8pt];
    \path (-8.5,4.28)  edge [-,dotted](-8.75,5.35) [line width=0.8pt];
     \path (-8, 0)  edge [-,dotted](-9.25,5.9) [line width=0.8pt];
     \path (-8.5, 0)  edge [-,dotted](-9.75,6.5) [line width=0.8pt];
      \path (-9, 0)  edge [-,dotted](-10.25,7.15) [line width=0.8pt];
      \path (-9.5, 0)  edge [-,dotted](-10.5,7.48) [line width=0.8pt];
       \node (a) at (4,-0.5) {\tiny$\beta_*$};
       \node (a) at (-4,-0.5) {\tiny$-\beta_*$};
        \node (a) at (-8.5,4) {\tiny$I_-$};
        \node (a) at (1,4) {\tiny$O$};
         \node (a) at (8.45,4) {\tiny$I_+$};
         \node (a) at (-7,4) {\tiny$\Gamma_-$};
        \node (a) at (7,4) {\tiny$\Gamma_+$};
 \end{tikzpicture}
\end{center}\vspace{-0.2cm}
 \begin{center}\vspace{-0.2cm}
   {\small {\bf Figure 1.} }
  \end{center}

Theorem \ref{thm-non-shear-nodelta2} is illustrated in Figure 1. The whole half-plane is divided into three regions
\begin{align*}
I_-=&\{(\alpha,\beta)|\beta<-\beta_*,0<\alpha\leq \alpha_\beta\},\quad I_+=\{(\alpha,\beta)|\beta>\beta_*,0<\alpha\leq \alpha_\beta\},\\
O=&\{(\alpha,\beta)|\beta<-\beta_*,\alpha>\alpha_\beta\}\cup\{(\alpha,\beta)||\beta|\leq\beta_*,\alpha>0\}\cup\{(\alpha,\beta)|\beta>\beta_*,\alpha>\alpha_\beta\}.
\end{align*}
Theorem \ref{thm-non-shear-nodelta2} reveals contrasting dynamics near Couette flow between $(\alpha,\beta)\in O$ and $(\alpha,\beta)\in I_-\cup I_+$: non-parallel steadily traveling waves do not exist for $(\alpha,\beta)\in O$ and such traveling waves exist for $(\alpha,\beta)\in I_-\cup I_+$,
   near Couette flow for  $H^{\geq5}$ velocity perturbation. Here, the conclusion for $\beta=0$ is proved in Theorem 2 of \cite{LZ}.
 The borderlines between the region $O$ and its remaining regions are the symmetry curves
\begin{align*}
\Gamma_-=\{(\alpha,\beta)|\beta<-\beta_*,\alpha= \alpha_\beta\},\quad \Gamma_+=\{(\alpha,\beta)|\beta>\beta_*,\alpha= \alpha_\beta\},
\end{align*}
where $-\alpha_\beta^2=\lambda_1(\beta,1)$ for $ \beta<-\beta_*$ and  $-\alpha_\beta^2=\lambda_1(\beta,-1)$  for $\beta>\beta_*$ are exactly the two curves of the principal eigenvalues of singular Rayleigh-Kuo BVP in \eqref{sturm-Liouville} with $u(y)=y$, $c=\pm1$. The symmetry of  $\Gamma_{\pm}$ and $I_{\pm}$ with respect to the vertical axis is due to
symmetry of the principal eigenvalues  of  Rayleigh-Kuo operators  for Couette flow, see Lemma
\ref{symmetry}.

From the perspective of fixed horizontal period, we get the following restatement of Theorem \ref{thm-non-shear-nodelta2}.\\
{\bf Restatement of Theorem \ref{thm-non-shear-nodelta2}.} {\it
Let $s\geq5$, $T>0$ and $\beta_T>0$  be the unique point such that $-{4\pi^2\over T^2}={\lambda_1(\beta_T,-1)}$.

$(1)$ Fix $\beta\in (-\beta_T,\beta_T)$. Then there exists $\varepsilon_0>0$ such that any traveling wave solution  $(u(x-ct,y),v(x-ct,y))$ to the $\beta$-plane equation $\eqref{eq}$-$\eqref{bc}$ with $c\in\mathbb{R}$, $x$-period $T$ and satisfying that
\beno
\|(u,v)-(y,0)\|_{H^{s}{(D_T)}}<\varepsilon_0,
\eeno
must have $v(x,y)\equiv 0$, that is, $(u,v)$ is necessarily a shear flow.

$(2)$
Fix $\beta\in (-\infty,-\beta_T]\cup[\beta_T,\infty)$. Then for any $\varepsilon>0$, there exists a traveling wave solution  $(u_\varepsilon(x-c_\varepsilon t,y),v_\varepsilon(x-c_\varepsilon t,y))$ to the $\beta$-plane equation $\eqref{eq}$-$\eqref{bc}$ with $x$-period $T$ and satisfying that
\beno
\|(u_\varepsilon,v_\varepsilon)-(y,0)\|_{H^{s}{(D_T)}}<\varepsilon,
\eeno
but $v_\varepsilon(x,y)\not\equiv 0$.  As $\varepsilon\to0,$  we have $c_\varepsilon\to-1$ for $\beta=\beta_T$; $c_\varepsilon\to c_0\in(-\infty,-1)$ for $\beta>\beta_T$; $c_\varepsilon\to1$ for $\beta=-\beta_T$;  $c_\varepsilon\to c_0\in(1,\infty)$ for $\beta<-\beta_T$.

Moreover, $\beta_T$ is continuous and decreasing on $T\in(0,\infty)$,  $\beta_T\to \infty$ as $T\to0$, and $\beta_T\to \beta_*$ as $T\to\infty$.}
\medskip

Theorem \ref{thm-non-shear-nodelta2} and its restatement are briefly refined as follows.
Consider   $H^{\geq5}$ velocity perturbation near Couette flow. From the perspective of fixed $\beta\in\mathbb{R}$, we have the conclusions as follows.
 \begin{itemize}
 \item
Fix $\beta\in[-\beta_*,\beta_*]$. Then for any  horizontal period $T>0$,  non-parallel  traveling waves do not exist near Couette flow.
 \item Fix $\beta\in(-\infty,-\beta_*)\cup(\beta_*,\infty)$. Then
  \begin{enumerate}
\item  non-parallel traveling waves  exist for the  horizontal period  $T\in \left[T_\beta,\infty\right)$,
   \item non-parallel  traveling waves do not exist for the  horizontal period  $T\in \left(0,T_\beta\right)$.
   \end{enumerate}
\end{itemize}
From the perspective of fixed horizontal period $T>0$, we have
 \begin{itemize}
 \item
 non-parallel  traveling waves do not exist for $\beta\in (-\beta_T,\beta_T)$,
\item
  non-parallel traveling waves  exist for $\beta\in (-\infty,-\beta_T]\cup[\beta_T,\infty)$.
\end{itemize}

\begin{remark} $(1)$ By Theorem $\ref{thm-non-shear-nodelta2}$ $(2\rm{i})$, if $(\alpha,\beta)\in (I_+\cup I_-)\setminus(\Gamma_+\cup\Gamma_-)$ $\left(i.e.\; T={2\pi\over \alpha}\in\left(T_\beta,\infty\right)\right)$, then the traveling speeds of  constructed   traveling waves converge to an isolated real eigenvalue of $\mathcal{R}_{\alpha,\beta}$; while  if $(\alpha,\beta)\in \Gamma_+\cup\Gamma_-$ $\left(i.e. \; T=T_\beta\right)$, then the traveling speeds of  constructed traveling waves converge to the embedding eigenvalues $\pm1$ of $\mathcal{R}_{\alpha,\beta}$.

$(2)$ For $(\alpha,\beta)\in I_+\cup I_-$, the long time dynamics near Couette flow is richer due to the existence of non-parallel  traveling waves, as  the evolutionary  velocity might tend asymptotically to some nontrivial (relative) equilibrium if the initial data is
taken close to Couette flow for $H^{\geq5}$ velocity perturbation.  It is very challenging to give a complete description of asymptotic behavior for the solutions if the initial data is taken
 near  Couette flow. On the other hand, for $(\alpha,\beta)\in O$ (including the case of no Coriolis effects), the dynamics near Couette flow is relatively simpler on account of the absence of non-parallel  traveling waves.

$(3)$ From the perspective of spectrum of  $\mathcal{R}_{\alpha,\beta}$, we give the differences among $(\alpha,\beta)\in O, \Gamma_+, \Gamma_-, I_+\setminus \Gamma_+$ and  $I_-\setminus \Gamma_-$, where $\mathcal{R}_{\alpha,\beta}$  always denotes the  linearized operator in \eqref{linearized Euler operator} with $u(y)=y$.
\begin{itemize}
 \item
For $(\alpha,\beta)\in O$,
 $\mathcal{R}_{\alpha,\beta}$ has no embedding eigenvalues or isolated  real eigenvalues.
 \item
 For $(\alpha,\beta)\in \Gamma_+$, $\mathcal{R}_{\alpha,\beta}$ has  a unique  embedding eigenvalue
 $-1$ and no   isolated real eigenvalues.
 \item
 For $(\alpha,\beta)\in \Gamma_-$, $\mathcal{R}_{\alpha,\beta}$ has  a unique  embedding eigenvalue
 $1$ and no   isolated real eigenvalues.
 \item
 For $(\alpha,\beta)\in I_+\setminus \Gamma_+$, $\mathcal{R}_{\alpha,\beta}$ has a unique  isolated real eigenvalue $c_0\in(-\infty,-1)$ and no  embedding eigenvalues.
 \item
 For $(\alpha,\beta)\in I_-\setminus\Gamma_-$, $\mathcal{R}_{\alpha,\beta}$ has a  unique  isolated real eigenvalue $c_0\in(1,\infty)$ and no  embedding eigenvalues.
\end{itemize}
The definition of the embedding eigenvalue of $\mathcal{R}_{\alpha,\beta}$ is given in Definition $3.10$ of \cite{WZZhu}. By Remark $1.3 \; (3)$ in \cite{WZZhu}, the potential embedding eigenvalues of $\mathcal{R}_{\alpha,\beta}$ can only be $\pm1$ for Couette flow and any $(\alpha,\beta)\in\mathbb{R}^+\times \mathbb{R}$.
Note that in the case $\beta=0$, there exist no embedding eigenvalues or  isolated real eigenvalues for the linearized Euler operator around  Couette flow.

$(4)$ For Sinus flow, which is a non-monotone shear flow, even though we only count the families of traveling waves with traveling speeds converging outside the range of the flow, the number is infinite for $\beta<-{1\over 2}\pi^2$ or $\beta>{9\over 16}\pi^2$, and any horizontal period, see Figure $1$ in \cite{LWZZ}. For Couette flow, which is a monotone shear flow, if we count the  families of traveling waves with traveling speeds converging no matter inside or outside of the range of the flow, the number is zero for $(\alpha,\beta)\in O$ and finite for $(\alpha,\beta)\in I_\pm$. Thus, the long time dynamics near a non-monotone flow seems more complicated than a monotone flow for the $\beta$-plane equation.

$(5)$ $s\geq5$ might be improved, as $s\geq3$ is sufficient for $(2\rm{i})$, see its proof.
However, as is shown in the next theorem, the optimal  $s$ value can not be less than ${5\over2}$.
\end{remark}

The proof of Theorem \ref{thm-non-shear} is to rule out the traveling waves with traveling speeds converging to $c\in(-1,1)$, while
 the main task in the  proof of Theorem  \ref{thm-non-shear-nodelta2} (1) and (2ii) is to  rule out the traveling waves with traveling speeds converging to $c=\pm1$. For the proof of Theorem \ref{thm-non-shear}, our approach is to prove that the sequence of  $L^2$ normalized vertical velocity $\tilde v_n$ of the  traveling waves is  uniformly $H^4$ bounded. This allows us to take limits at the equation  for $\tilde v_n$. The limit equation is exactly the Rayleigh-Kuo equation with $c\in(-1,1)$, which contradicts that  $c$ is not an embedding eigenvalue of $\mathcal{R}_{\alpha,\beta}$ \cite{LYZ}. The difficulty is  to prove the uniform $H^4$ bound for $\tilde v_n$, $n\geq1$. Thanks to the $\beta$'s term, we  use the integral expression of the velocity to cancel the singularity induced by $c\in(-1,1)$, and apply the Gagliardo-Nirenberg interpolation inequality  to close the estimates. For the proof of Theorem \ref{thm-non-shear-nodelta2} (1) and (2ii), we have no singularity cancelation as above and could only prove the uniform $H^2$ bound for  $\tilde v_n$, $n\geq1$. This turns out to be enough after we consider $\tilde v_n$'s equation in the weak sense and fully use the non-permeable boundary condition for both $\tilde v_n$ and the test functions.

The proof of Theorem \ref{thm-non-shear-nodelta2} (2i) is to construct traveling waves by bifurcation at suitable shear flows near Couette flow. In the case $(\alpha,\beta)\in \Gamma_{\pm}$,  since the eigenvalue $\mp1$ for  $\mathcal{R}_{\alpha,\beta}$ is embedded in $[-1,1]$, it is difficult to get the $C^2$ regularity of the nonlinear bifurcated map if we consider the bifurcation as Couette flow itself. Our approach is to consider the bifurcation at the scaled nearby shear flow $(ay,0)$ with $|a|<1$, and the eigenvalue of the linearized operator around  $(ay,0)$ becomes an isolated one. Thus, we could use the bifurcation result in \cite{LWZZ} at $(ay,0)$. In the case  $(\alpha,\beta)\in (I_+\cup I_-)\setminus(\Gamma_+\cup\Gamma_-)$, the traveling waves  are constructed by bifurcation directly at Couette flow, since we prove the existence of isolated real eigenvalue of $\mathcal{R}_{\alpha,\beta}$ in Section 3.

\smallskip

Our next result is to consider the dynamics near Couette flow for the $\beta$-plane equation in the Sobolev spaces with low regularity. We construct non-shear stationary solutions near Couette flow for $H^{<{5\over2}}$ velocity perturbation. This result  is a generalization of Theorem 1 in \cite{LZ}, but the bifurcation lemma and  construction of the modified shear flow  turn out to be more delicate due to the $\beta$'s term.
\begin{Theorem}\label{thm1}
$(1)$ Let $\beta\neq0$, $T>0$ and $0\leq s< \frac{5}{2}$. Then for any $\varepsilon>0$, there exists a steady solution $(u_\varepsilon(x,y),v_\varepsilon(x,y))$ to the $\beta$-plane equation \eqref{eq}-\eqref{bc} with $x$-period $T$ satisfying that
\beno
\|(u_\varepsilon,v_\varepsilon)-(y,0)\|_{H^{s}{(D_T)}}<\varepsilon,
\eeno
but $v_\varepsilon(x,y)\not\equiv 0$.

$(2)$ Let $0<|\beta|<{4\sqrt{2}\over3\pi}$, $T>0$ and $0\leq s< \frac{5}{2}$. Then the conclusion in $(1)$ holds true, and moreover,
$T$ is  the minimal period in $x$ of the  steady solution $(u_\varepsilon(x,y),v_\varepsilon(x,y))$.
\end{Theorem}
The bifurcation lemma  in \cite{LZ} can not be applied to the case  $\beta\neq0$, since there is a singularity at the middle point $0$ of the linearized bifurcated map, which is difficult to deal with. We introduce a bifurcation lemma for $\beta\neq0$ such that the potential term of the Rayleigh-Kuo operator  is flat near $0$. To adjust the flatness condition,  we add a cut-off   $\beta$'s term in the constructed shear flow, and to produce negative eigenvalues of the Rayleigh-Kuo BVP, we try to add the Gauss error function introduced in \cite{LZ}. However, a direct addition of the error function induces new singularity at $0$. Our method is to translate the cut-off Gauss  error function such that its  support does not intersect that of the $\beta$'s term and  to make the translation sufficiently  close to $0$, besides the size of the cut-off function should be suitably small.

\begin{remark}
We summarize the modified shear flows at which the bifurcation could be used to construct traveling waves for the $\beta$-plane equation, which are technically important.
 \begin{itemize}
 \item
Scaled modified shear flow $(ay,0)$, which is used to construct traveling waves near Couette flow  if  $(\alpha,\beta)$ are chosen such that the linearized operator has an embedding eigenvalue $1$ or $-1$, see the proof of Theorem $\ref{thm-non-shear-nodelta2}\; (2\rm{i})$.
 \item
Couette flow $+$ cut-off   $\beta$'s term $+$  cut-off Gauss  error function $($see \eqref{modified shear flow}$)$, which is used to construct steady states near Couette flow  at low regularity.
 \item
$(u+\nu u_1+\tau u_2,0)$,  which is used to  guarantee that the bifurcated solutions near the flow $(u,0)$ is not a
shear one in Lemma $2.3$ of \cite{LWZZ} and Lemma $\ref{lem-bifurcation}$.
 Here, $u_1$ and $u_2$ are added to obtain monotonicity of the eigenvalues of corresponding Rayleigh-Kuo operators, and  $\nu, \tau\in \mathbb{R}$ are sufficiently small.
\end{itemize}
\end{remark}

Finally, we generalize the asymptotic stability of Couette flow  in   Gevrey spaces  \cite{BM} to the  $\beta$-plane equation on $\Omega=\mathbb{T}_{2\pi}\times\mathbb{R}$. 
Consider the $\beta$-plane equation (\ref{eq}) on $\Omega$. Take $\vec{v}=(y,0)+\vec{U}=(y,0)+(U^x,U^y)$, where $\vec{U}=(U^x,U^y)$ denotes the velocity perturbation. Then the equation reads
\begin{eqnarray*}
\left\{
\begin{aligned}
&\partial_t\vec{U}+y\pa_x\vec{U}+(U^y,0)+\vec{U}\cdot\nabla\vec{U}+\nabla P=-\beta y(- U^y,y+U^x),\\
&\nabla\cdot \vec{U}=0.\\
\end{aligned}
\right .
\end{eqnarray*}
Let $w=\curl \vec{U}=\pa_x U^y-\pa_y U^x$, then the total vorticity is $\omega=-1+w$ and the vorticity form (\ref{vor-eq}) becomes
\begin{eqnarray}\label{vor-equ}
\left\{
\begin{aligned}
&\pa_t w+y\pa_x w+\vec{U}\cdot\nabla w+\beta U^y=0,\\
&\vec{U}=\nabla^\perp(\Delta)^{-1}w,\\
&w|_{t=0}=w_{in}.
\end{aligned}
\right .
\end{eqnarray}
Here, $(x,y)\in\mathbb{T}_{2\pi}\times\mathbb{R}$, $\nabla^\perp=(-\pa_y,\pa_x)$ and $(\vec{U},w)$ are periodic in the $x$ variable with period normalized to $2\pi$. Denote
$\tilde{\psi}=\Delta^{-1}w$.
We take an interest in the long time behavior of (\ref{vor-equ}) for small initial perturbations $w_{in}$ and get the following result.
\begin{Theorem}\label{thm-damping}
For all $\frac{1}{2}<s\leq 1$, $\lambda_0>\lambda'>0$, there exists  $\epsilon_0=\epsilon_0(\lambda_0,\lambda',s)\leq {1\over 2}$ such that for all $\epsilon\leq \epsilon_0$, if $w_{in}$ satisfies $\int_\Omega w_{in} dxdy=0$, $\int_\Omega |y w_{in}| dxdy<\epsilon$ and
\beno
\|w_{in}\|_{\mathcal{G}^{\lambda_0}(\Omega)}^2:=\sum_{k\in \mathbb{Z}} \int_\mathbb{R} |\hat{w}_{in}(k,\eta)|^2 e^{2\lambda_0|k,\eta|^s} d\eta\leq \epsilon^2,
\eeno
then there exists $f_{\infty}$ with $\int_\Omega f_{\infty} dxdy=0$ and $\|f_{\infty}\|_{\mathcal{G}^{\lambda'}(\Omega)}\lesssim \epsilon$ such that
\beno
\left \|w \left(t,x+ty+\Phi(t,y),y \right)-f_{\infty}(x,y)\right\|_{\mathcal{G}^{\lambda'}(\Omega)}\lesssim \frac{\epsilon^2}{\langle t \rangle},
\eeno
where $\Phi(t,y)$ is given by
\beno
\Phi(t,y)=\frac{1}{2\pi}\int_0^t\int_{\mathbb {T}_{2\pi}} U^x(\tau,x,y) dxd\tau=u_{\infty}(y)t+O(\epsilon),
\eeno
with $u_{\infty}(y)=\pa_y\pa_{yy}^{-1}\frac{1}{2\pi}\int_{\mathbb {T}_{2\pi}} f_{\infty}(x,y) dx$ for $y\in\mathbb{R}$. Moreover, the velocity field $\vec{U}$ satisfies
\beno
\left\|\frac{1}{2\pi}\int_{\mathbb{T}_{2\pi}} U^x(t,x,\cdot)dx-u_{\infty}(y) \right\|_{\mathcal{G}^{\lambda'}(\Omega)}&\lesssim& \frac{\epsilon^2}{\langle t \rangle^2},\\
\left\|U^x-\frac{1}{2\pi}\int_{\mathbb{T}_{2\pi}} U^x(t,x,\cdot)dx\right\|_{L^2(\Omega)} &\lesssim& \frac{\epsilon}{\langle t \rangle},\\
\left\|U^y(t) \right\|_{L^2(\Omega)} &\lesssim& \frac{\epsilon}{\langle t \rangle^2}.
\eeno
\end{Theorem}
{We use the same time-dependent norm and main energy as in \cite{BM}, and prove the bootstrap proposition. The difference comes from the new term $\int_{\mathbb{T}_{2\pi}} AfA(\beta\pa_z\phi)dx$, where $A$ is the multiplier in Subsection 2.3 of \cite{BM}. To treat this term, we improve the elliptic control (Proposition 2.4) in \cite{BM}. Notice that $\int_{\mathbb{T}_{2\pi}} AfA(\beta\pa_z\phi)dx=\beta\int_{\mathbb{T}_{2\pi}} AfA(\pa_z\phi_1)dx$ for $\phi_1=P_{\neq 0}\phi-\Delta_L^{-1}P_{\neq 0}f$, where $f$ is the vorticity under {changed coordinates $(z,\vartheta)$, $\Delta_L=\pa_z^2+(\pa_\vartheta-t\pa_z)^2$} and $P_{\neq 0}f=f-\langle f\rangle=f-\int_{\mathbb{T}_{2\pi}}f dx$. The new term has similar bound with Reaction (Proposition 2.3 in \cite{BM}), so the method of \cite{BM} works here.
}

\smallskip

The rest of this paper is organized as follows. In Section 2, we prove the non-existence of non-parallel traveling waves with traveling speeds converging  inside $(-1,1)$ near Couette flow for $H^{\geq5}$ velocity perturbation and $\beta\neq0$. In Section 3,  we study properties of  the   principal eigenvalues  of  Rayleigh-Kuo operators  for Couette flow. In Section 4, we determine the sharp region  such that
  no traveling waves exist for $(\alpha,\beta)$ in this region and  traveling waves  exist for $(\alpha,\beta)$ in the remaining regions   near Couette flow    for $H^{\geq 5}$ velocity perturbation.  In Section 5, we construct non-shear steady states  near Couette flow in low Sobolev spaces for $\beta\neq0$. Note that the domain in Sections 2-5 is a finite channel $\mathbb{T}\times [-1,1]$ and the perturbation is considered in Sobolev spaces.  Finally, we prove nonlinear inviscid damping near Couette flow  in  suitable Gevrey spaces for $\beta\neq0$ on $\Omega=\mathbb{T}_{2\pi}\times\mathbb{R}$ in Section 6.

\section{Non-existence of {traveling waves} with traveling speeds  inside (-1,1)}

In this section, we prove that there are no non-parallel traveling waves with traveling speeds converging in $(-1,1)$ near Couette flow for $H^{\geq5}$ velocity perturbation, $\beta\neq0$ and any $x$-period $T$, which is stated in Theorem \ref{thm-non-shear}.

The following two lemmas will be used, which are Hardy type inequality \cite{Godfrey-Hardy-Littlewood,LWZZ} and Gagliardo-Nirenberg interpolation inequality \cite{Gagliardo,Nirenberg1959,Nirenberg1966}.
\begin{lemma}\label{Hardy type inequality2}
Let $\phi\in H^1(a,b)$ and $\phi(y_0)=0$ for some $y_0\in [a,b]$. Then
\begin{align*}
\left\|\phi\over y-y_0\right\|_{L^2(a,b)}^2\leq C\|\phi'\|_{L^2(a,b)}^2.
\end{align*}
\end{lemma}
\begin{lemma}\label{Gagliardo-Nirenberg interpolation inequality}
Let $G$ be a bounded domain in $\mathbb{R}^n$ having the cone property. For $1\leq q,r\leq \infty$, suppose $u\in L^q(G)$ and $D^m u\in L^r(G)$. Then for $0\leq j<m$, the following inequalities hold (with constant $C_1, C_2$ depending only on $G$, $m, j, q, r$)
\begin{align}\label{Gagliardo-Nirenberg inequality}
\|D^j u\|_{L^p(G)}\leq C_1\|D^m u\|_{L^r(G)}^\alpha\|u\|_{L^q(G)}^{1-\alpha}+C_2\|u\|_{L^q(G)},
\end{align}
where
\begin{align*}
{1\over p}={j\over n}+\alpha\left({1\over r} -{m\over n}\right)+(1-\alpha){1\over q},
\end{align*}
for all $\alpha\in[{j\over m},1]$,
unless $1<r<\infty$ and $m-j-n/r$ is a nonnegative integer, in which case \eqref{Gagliardo-Nirenberg inequality} holds only for $\alpha$ satisfying $\alpha\in[{j\over m},1)$.
\end{lemma}

\begin{proof}[Proof of Theorem \ref{thm-non-shear}.]
Suppose otherwise, there exist $\{\varepsilon_n\}_{n=1}^\infty$, {$c_n\in[-1+\delta,1-\delta]$ and $(u_n(x-c_n t,y), v_n(x-c_n t,y))$} to the $\beta$-plane equation (\ref{eq})-(\ref{bc})  such that $\varepsilon_n\to0$, $(u_n,v_n)$ is $T$-periodic in $x$,
$
\|(u_n,v_n)-(y,0)\|_{H^{s}{(D_T)}}\leq \varepsilon_n
$
and $\|v_n\|_{L^2(D_T)}\neq0$.  Then
\ben\label{un-vn-eq}
{(u_n-c_n)}\pa_x\omega_n+v_n(\pa_y\omega_n+\beta)=0,
\een
where $\omega_n=\pa_x v_n-\pa_y u_n$.
Up to a subsequence, $c_n\to c_0\in[-1+\delta,1-\delta]$.
$s\geq5$ implies
\begin{align}\label{un-y-C3-nd}
\|u_n-y\|_{C^3{(D_T)}}+\|v_n\|_{C^3{(D_T)}}\leq C \|(u_n,v_n)-(y,0)\|_{H^{s}{(D_T)}}\leq& C \varepsilon_n,\\
\label{omegan-C2-nd}
\|\omega_n+1\|_{C^2{(D_T)}}\leq C\|\omega_n+1\|_{H^4{(D_T)}}\leq C \|(u_n,v_n)-(y,0)\|_{H^{s}{(D_T)}}\leq& C \varepsilon_n,\\
\label{omega-n-Lp-nd}
\|\pa_{xxy}\omega_n\|_{L^4{(D_T)}}+\|\pa_{xyy}\omega_n\|_{L^4{(D_T)}}+\|\pa_{y}^3\omega_n\|_{L^4{(D_T)}}\leq C\|\omega_n+1\|_{H^4{(D_T)}}\leq& C \varepsilon_n.
\end{align}
Then
\ben \label{un-omegan-C bound-nd}
\|u_n\|_{C^3{(D_T)}}\leq C,\quad \|\omega_n\|_{C^2{(D_T)}}\leq C.
\een
Moreover,
\ben \label{un-xy-bound-nd}
{1}/{2}<\pa_y u_n(x,y)<{3}/{2},\quad(x,y)\in D_T
\een
for $n$ sufficiently large.
Then $u_n(x,\cdot)$ is  increasing on $[-1,1]$ for $x\in[0, T]$ and by taking $n$ larger, we have $\|u_n-y\|_{L^\infty(D_T)}\leq C\varepsilon_n<{\delta\over2}$. For fixed $n$,  we claim that $u_n(x,-1)<c_n<u_n(x,1)$  for  $ x\in [0,T]$. Suppose otherwise, there exists  $x_0\in[0,T]$ such that $u_n(x_0,1)\leq c_{n}$ or $c_n\leq u_n(x_0,-1)$.   For the first case, we have
\beno
{\delta\over2}>C\varepsilon_n>\|u_n-y\|_{C^0(D_T)}\geq |u_n(x_0,c_n)-c_n|\geq |u_n(x_0,c_n)-u_n(x_0,1)|>\frac{1}{2}|c_n-1|>\frac{\delta}{2},
\eeno
which is a contradiction. The latter case $c_n\leq u_n(x_0,-1)$ is similar since  ${\delta\over2}>C\varepsilon_n>\|u_n-y\|_{C^0(D_T)}\geq|u_n(x_0,c_n)-c_n|\geq |u_n(x_0,c_n)-u_n(x_0,-1)|>\frac{1}{2}|c_n+1|>\frac{\delta}{2}.$

 Thus, there exists a unique $y_n(x)\in (-1,1)$ such that $u_n(x,y_n(x))={c_n}$ for $x\in[0,T]$ and $n$ sufficiently large.
 Since $\|\pa_y\omega_n\|_{C^0(D_T)}\leq\|\omega_n+1\|_{C^2{(D_T)}}\leq C \varepsilon_n$ and
  $\beta\neq0$, we have $|\pa_y\omega_n(x,y)+\beta|\neq0$ for $(x,y)\in D_T$ and $n$ sufficiently large. This, along with $u_n(x,y_n(x))={c_n}$ and   \eqref{un-vn-eq}, implies that
  $v_n(x,y_n(x))=0$ for $x\in [0,T]$.
By the incompressible condition, we have $\partial_x\omega_n=\Delta v_n$.  Let $\tilde v_n=v_n/\|v_n\|_{L^2(D_T)}$. Then $\tilde v_n(x,y_n(x))=0$  for $x\in [0,T]$.  By
\eqref{un-vn-eq}, we have
\ben\label{tilde-vn-eq}
\Delta \tilde v_n+(\pa_y\omega_n+\beta){\tilde v_n\over u_n{-c_n}}=0.
\een

First, we prove the uniform $H^2$ bound for $\tilde v_n$, $n\geq 1$. It follows from   \eqref{un-xy-bound-nd} that  $\left|{y-y_n(x)\over u_n(x,y){-c_n}}\right|=\left|{y-y_n(x)\over u_n(x,y)-u_n(x,y_n(x))}\right|\leq 2$ for $(x,y)\in D_T$. Thus, $\left\|{y-y_n(x)\over u_n{-c_n}}\right\|_{L^\infty(D_T)}\leq C$. This, along with  \eqref{tilde-vn-eq}, \eqref{un-omegan-C bound-nd} and Lemma  \ref{Hardy type inequality2}, gives
\ben\nonumber
\|\Delta \tilde v_n\|_{L^2(D_T)}&\leq& \left\|(\pa_y\omega_n+\beta){\tilde v_n\over u_n{-c_n}}\right\|_{L^2(D_T)}\leq C \left\|{\tilde v_n\over u_n{-c_n}}\right\|_{L^2(D_T)}\\\nonumber
&\leq& C \left\|{\tilde v_n\over y-y_n(x)}\right\|_{L^2(D_T)}\left\|{y-y_n(x)\over u_n{-c_n}}\right\|_{L^\infty(D_T)}\\\nonumber
&\leq& C\left\|{\tilde v_n\over y-y_n(x)}\right\|_{L^2(D_T)}\leq C\left(\int_0^T\left\|{\tilde v_n(x,\cdot)}\right\|_{H^1(-1,1)}^2dx\right)^{1\over2}\\\label{tilde vn-over-un-L2}
&\leq&C\left\|{\tilde v_n}\right\|_{H^1(D_T)}\leq C\left\|{\tilde v_n}\right\|_{H^2(D_T)}^{1/2}\left\|{\tilde v_n}\right\|_{L^2(D_T)}^{1/2}=C\left\|{\tilde v_n}\right\|_{H^2(D_T)}^{1/2}.
\een
Thus, by  periodic
boundary condition in $x$ and Dirichlet boundary condition in $y$ of $\tilde v_n$, we have
\ben\label{tilde vn-H2 bound}
\| \tilde v_n\|_{H^2(D_T)}\leq C\|\Delta \tilde v_n\|_{L^2(D_T)}\leq  C\left\|{\tilde v_n}\right\|_{H^2(D_T)}^{1/2}\Longrightarrow\| \tilde v_n\|_{H^2(D_T)}\leq C.
\een

Next, we prove the uniform $H^3$ bound for $\tilde v_n$, $n\geq 1$.
Taking derivative of  \eqref{tilde-vn-eq} with respect to $y$, we have
\beno
\pa_y\Delta\tilde v_n+\pa_y^2\omega_n{ \tilde v_n\over u_n{-c_n}}+(\pa_y\omega_n+\beta)\pa_y\left({ \tilde v_n\over u_n{-c_n}}\right)=0.
\eeno
Thus, by \eqref{un-omegan-C bound-nd} and  \eqref{tilde vn-over-un-L2} we have
\begin{align}\nonumber
\|\pa_y\Delta\tilde v_n\|_{L^2(D_T)}
\leq& \|\pa_y^2\omega_n\|_{L^\infty(D_T)}\left\|{ \tilde v_n\over u_n{-c_n}}\right\|_{L^2(D_T)}
+\|\pa_y\omega_n+\beta\|_{L^\infty(D_T)}\left\|\pa_y\left({ \tilde v_n\over u_n{-c_n}}\right)\right\|_{L^2(D_T)}\\
\label{derivative-to-y-norm}
\leq&C\left\|{ \tilde v_n}\right\|_{H^1(D_T)}
+C\left\|\pa_y\left({ \tilde v_n\over u_n{-c_n}}\right)\right\|_{L^2(D_T)}.
\end{align}
 Since $\tilde v_n(x,y_n(x))=u_n(x,y_n(x)){-c_n}=0$ for $x\in[0,T]$ and $n$ sufficiently large, we have
\begin{align}\nonumber
{\tilde v_n(x,y)\over u_n(x,y){-c_n}}=&{\tilde v_n(x,y)-\tilde v_n(x,y_n(x))\over u_n(x,y)-u_n(x,y_n(x))}={\int_{y_n(x)}^y\pa_y\tilde v_n(x,s_n)ds_n\over
\int_{y_n(x)}^y\pa_yu_n(x,s_n)ds_n} \\\label{vn-un-int}
=&{\int_0^1\pa_y\tilde v_n(x,y_n(x)+t(y-y_n(x)))dt\over \int_0^1\pa_y u_n(x,y_n(x)+t(y-y_n(x)))dt}={\int_0^1\pa_y\tilde v_n(x,s_n)dt\over \int_0^1\pa_y u_n(x,s_n)dt},
\end{align}
where  $s_n=y_n(x)+t(y-y_n(x))$. Here, $s_n$ is dependent on $x$, and we always use $s_n$ to
avoid tedious notation. Then
\begin{align*}
\pa_y\left({\tilde v_n(x,y)\over u_n(x,y){-c_n}}\right)
=&{\int_0^1\pa_y^2\tilde v_n(x,s_n)tdt\int_0^1\pa_y u_n(x,s_n)dt-\int_0^1\pa_y\tilde v_n(x,s_n)dt\int_0^1\pa_y^2 u_n(x,s_n) tdt\over \left(\int_0^1\pa_y u_n(x,s_n)dt\right)^2}.
\end{align*}
By \eqref{un-omegan-C bound-nd} and \eqref{un-xy-bound-nd}, we have
\begin{align}\label{un-derivative-to-y}
1/2<\int_0^1\pa_y u_n(x,s_n)dt<3/2\;\;\;\forall(x,y)\in D_T, \text{ and }\left\|\int_0^1\pa_y^2 u_n(x,s_n) tdt\right\|_{L^\infty(D_T)}\leq C.
\end{align}
 Thus,
\begin{align}\nonumber
\left\|\pa_y\left({\tilde v_n\over u_n{-c_n}}\right)\right\|_{L^2(D_T)}
{\leq}&C\left\|\int_0^1\pa_y^2\tilde v_n(x,s_n)tdt\right\|_{L^2(D_T)}\left\|\int_0^1\pa_y u_n(x,s_n)dt\right\|_{L^\infty(D_T)}\\\nonumber
&+C\left\|\int_0^1\pa_y\tilde v_n(x,s_n)dt\right\|_{L^2(D_T)}\left\|\int_0^1\pa_y^2 u_n(x,s_n)tdt\right\|_{L^\infty(D_T)}\\\label{vn-overun-derivative-L2}
\leq &C\left\|\int_0^1\pa_y^2\tilde v_n(x,s_n)tdt\right\|_{L^2(D_T)}
+C\left\|\int_0^1\pa_y\tilde v_n(x,s_n)dt\right\|_{L^2(D_T)}.
\end{align}
A direct computation implies
\begin{align}\nonumber
&\left\|\int_0^1|\pa_y\tilde v_n(x,s_n)|dt\right\|_{L^2(D_T)}^2=\int_0^T\int_{-1}^1\left|\int_{y_n(x)}^y{|\pa_y\tilde v_n(x,s_n)|\over y-y_n(x)}ds_n\right|^2dydx\\\nonumber
\leq&\int_0^T\int_{-1}^1{1\over |y-y_n(x)|}\left|\int_{y_n(x)}^y{|\pa_y\tilde v_n(x,s_n)|^2}ds_n\right|dydx\\\nonumber
=&\int_0^T\int_{-1}^{y_n(x)}{1\over y_n(x)-y}\int_y^{y_n(x)}{|\pa_y\tilde v_n(x,s_n)|^2}ds_ndydx\\\nonumber
&+\int_0^T\int_{y_n(x)}^1{1\over y-y_n(x)}\int_{y_n(x)}^y{|\pa_y\tilde v_n(x,s_n)|^2}ds_ndydx\\\nonumber
=&-\int_0^T\left(\ln(y_n(x)-y)\int_y^{y_n(x)}|\pa_y\tilde v_n(x,s_n)|^2ds_n\right)\big|_{y=-1}^{y_n(x)}dx\\\nonumber
&-\int_0^T\int_{-1}^{y_n(x)}\ln(y_n(x)-y)|\pa_y\tilde v_n(x,y)|^2dydx\\\nonumber
&+\int_0^T\left(\ln(y-y_n(x))\int_{y_n(x)}^y|\pa_y\tilde v_n(x,s_n)|^2ds_n\right)\big|_{y=y_n(x)}^{1}dx\\\nonumber
&-\int_0^T\int_{y_n(x)}^1\ln(y-y_n(x))|\pa_y\tilde v_n(x,y)|^2dydx\\\label{I1-4}
=&I_1(\pa_y\tilde v_n)+I_2(\pa_y\tilde v_n)+I_3(\pa_y\tilde v_n)+I_4(\pa_y\tilde v_n).
\end{align}
Since $\tilde v_n\in C^1(D_T)$, for fixed $n\geq1$ and $x\in[0,T]$ we have
\begin{align*}
&\lim_{y\to y_n(x)}\ln|y_n(x)-y|\left|\int_y^{y_n(x)}|\pa_y\tilde v_n(x,s_n)|^2ds_n\right|\\
=&\lim_{y\to y_n(x)}|y_n(x)-y|\ln|y_n(x)-y|\cdot\left|\lim_{y\to y_n(x)}{\int_y^{y_n(x)}|\pa_y\tilde v_n(x,s_n)|^2ds_n\over y_n(x)-y}\right|\\
=&0\cdot|\pa_y\tilde v_n(x,y_n(x))|^2=0.
\end{align*}
By \eqref{un-xy-bound-nd} and \eqref{un-y-C3-nd}, we have $|y_n(x){-c_n}|=|y_n(x)-u_n(x,{y_n(x)})|\leq\|u_n-y\|_{C^0(D_T)}\leq C\varepsilon_n,$ and thus,
\begin{align}\label{yn(x)-0uniform}
|y_n(x)-c_0|\leq |y_n(x)-c_n|+|c_n-c_0| \to 0
\end{align}
as $n\to\infty$ uniformly for $x\in[0,T]$. Then {$|y_n(x)|<\frac{1+|c_0|}{2}\leq 1-\frac{\delta}{2}$} for $x\in[0,T]$ and $n$ sufficiently large. Thus,
\begin{align}\nonumber
|I_1(\pa_y\tilde v_n)|+|I_3(\pa_y\tilde v_n)|\leq &\int_0^T|\ln(y_n(x)+1)|\int_{-1}^{y_n(x)}|\pa_y\tilde v_n(x,s_n)|^2ds_ndx\\\nonumber
&+\int_0^T|\ln(1-y_n(x))|\int_{y_n(x)}^1|\pa_y\tilde v_n(x,s_n)|^2ds_ndx\\\label{I13}
\leq& {C_\delta}\|\pa_y\tilde v_n\|_{L^2(D_T)}^2\leq C\|\tilde v_n\|_{H^1(D_T)}^2.
\end{align}
By \eqref{tilde vn-H2 bound} we have
\begin{align*}
\|\partial_y\tilde v_n\|_{L^4(D_T)}\leq C\|\tilde v_n\|_{H^2(D_T)}\leq C,
\end{align*}
and thus,
\begin{align}\nonumber
|I_2(\pa_y\tilde v_n)|\leq &\left(\int_0^T\int_{-1}^{y_n(x)}|\ln(y_n(x)-y)|^2dydx\right)^{1\over2} \left(\int_0^T\int_{-1}^{y_n(x)}|\pa_y\tilde v_n(x,y)|^4dydx\right)^{1\over2}\\\nonumber
\leq &\left(\int_0^T\int_{0}^{y_n(x)+1}|\ln \tau_n|^2d\tau_ndx\right)^{1\over2}\|\pa_y\tilde v_n\|_{L^4(D_T)}^2\\\label{I2}
\leq& \left(\int_0^T\int_{0}^2|\ln \tau|^2d\tau dx\right)^{1\over2}\|\tilde v_n\|_{H^2(D_T)}^2\leq C\|\tilde v_n\|_{H^2(D_T)}^2\leq C.
\end{align}
Similarly, we have
\begin{align}\label{I4}
|I_4(\pa_y\tilde v_n)|
\leq& \left(\int_0^T\int_{0}^2|\ln \tau|^2d\tau dx\right)^{1\over2}\|\tilde v_n\|_{H^2(D_T)}^2\leq C.
\end{align}
Combining \eqref{I1-4}, \eqref{I13}, \eqref{I2} and \eqref{I4}, we have
\begin{align}\label{I1-4-sum}
\left\|\int_0^1|\pa_y\tilde v_n(x,s_n)|dt\right\|_{L^2(D_T)}\leq C\|\tilde v_n\|_{H^2(D_T)}\leq C.
\end{align}
Thanks to the factor $t$, we can also prove that
 \begin{align}\label{derivative-y2}
 \left\|\int_0^1\pa_y^2\tilde v_n(x,s_n)tdt\right\|_{L^2(D_T)}\leq C\|\tilde v_n\|_{H^2(D_T)} \leq C.
 \end{align}
In fact,
\begin{align}&\left\|\int_0^1\pa_y^2\tilde v_n(x,s_n)tdt\right\|_{L^2(D_T)}^2\nonumber\\
=&\int_0^T\int_{-1}^1\left|\int_{y_n(x)}^y\pa_y^2\tilde v_n(x,s_n){s_n-y_n(x)\over (y-y_n(x))^2}ds_n\right|^2dydx\nonumber\\
\leq&\int_0^T\int_{-1}^1{1\over |y-y_n(x)|^3}\left|\int_{y_n(x)}^y|\pa_y^2\tilde v_n(x,s_n)|^2{(s_n-y_n(x))^2}ds_n\right|dydx\nonumber\\
=&\int_0^T\int_{-1}^{y_n(x)}{1\over (y_n(x)-y)^3}\int_y^{y_n(x)}|\pa_y^2\tilde v_n(x,s_n)|^2{(y_n(x)-s_n)^2}ds_ndydx\nonumber\\
&+\int_0^T\int_{y_n(x)}^1{1\over (y-y_n(x))^3}\int_{y_n(x)}^y|\pa_y^2\tilde v_n(x,s_n)|^2{(s_n-y_n(x))^2}ds_ndydx\nonumber\\
=&{1\over2}\int_0^T\left({1\over (y_n(x)-y)^2}\int_y^{y_n(x)}|\pa_y^2\tilde v_n(x,s_n)|^2{(y_n(x)-s_n)^2}ds_n\right)\bigg|_{y=-1}^{y_n(x)}dx\nonumber\\
&+{1\over2}\int_0^T\int_{-1}^{y_n(x)}|\pa_y^2\tilde v_n(x,y)|^2dydx\nonumber\\
&-{1\over2}\int_0^T\left({1\over (y-y_n(x))^2}\int_{y_n(x)}^y|\pa_y^2\tilde v_n(x,s_n)|^2{(s_n-y_n(x))^2}ds_n\right)\bigg|_{y=y_n(x)}^{1}dx\nonumber\\
&+{1\over2}\int_0^T\int_{y_n(x)}^1|\pa_y^2\tilde v_n(x,y)|^2dydx\nonumber\\
\leq&{1\over2}\int_0^T{1\over (y_n(x)+1)^2}\int_{-1}^{y_n(x)}|\pa_y^2\tilde v_n(x,s_n)|^2{(y_n(x)-s_n)^2}ds_ndx+{1\over2}\|\pa_y^2\tilde v_n\|_{L^2(D_T)}^2\nonumber\\
&+{1\over2}\int_0^T{1\over (1-y_n(x))^2}\int_{y_n(x)}^1|\pa_y^2\tilde v_n(x,s_n)|^2{(s_n-y_n(x))^2}ds_ndx+{1\over2}\|\pa_y^2\tilde v_n\|_{L^2(D_T)}^2\nonumber\\
\leq&2 \|\pa_y^2\tilde v_n\|_{L^2(D_T)}^2\leq 2\|\tilde v_n\|_{H^2(D_T)}^2\leq C,\label{til-v-derivative-2}
\end{align}
where we used $\tilde v_n\in C^2(D_T)$ to deduce that
$\lim\limits_{y\to y_n(x)}{1\over (y-y_n(x))^2}\left|\int_{y_n(x)}^y|\pa_y^2\tilde v_n(x,s_n)|^2{(s_n-y_n(x))^2}ds_n\right|$ $=0$ for fixed $n$ and $x\in[0,T]$.

By  \eqref{vn-overun-derivative-L2}, \eqref{I1-4-sum} and \eqref{derivative-y2}, we have
\begin{align}\label{vn-overun-derivative-L2-sum}
\left\|\pa_y\left({\tilde v_n\over u_n{-c_n}}\right)\right\|_{L^2(D_T)}\leq C,
\end{align}
and thus, by \eqref{derivative-to-y-norm},    \eqref{vn-overun-derivative-L2-sum} and \eqref{tilde vn-H2 bound} we have
\begin{align}\label{Delta-tilde-vn-derivative-to-y}
\|\pa_y\Delta\tilde v_n\|_{L^2(D_T)}
\leq &C\left\|{ \tilde v_n}\right\|_{H^1(D_T)}+C \leq C.
\end{align}
Taking derivative of  \eqref{tilde-vn-eq} with respect to $x$, by \eqref{un-omegan-C bound-nd} and  \eqref{tilde vn-over-un-L2} we have
\begin{align}\nonumber
\|\pa_x\Delta\tilde v_n\|_{L^2(D_T)}\leq&\|\pa_{xy}\omega_n\|_{L^\infty(D_T)}\left\|{ \tilde v_n\over u_n{-c_n}}\right\|_{L^2(D_T)}+\|\pa_y\omega_n+\beta\|_{L^\infty(D_T)}\left\|\pa_x\left({ \tilde v_n\over u_n{-c_n}}\right)\right\|_{L^2(D_T)}\\
\label{derivative-to-x-norm}
\leq&C\left\|{ \tilde v_n}\right\|_{H^1(D_T)}
+C\left\|\pa_x\left({ \tilde v_n\over u_n{-c_n}}\right)\right\|_{L^2(D_T)}.
\end{align}
Taking derivative of  \eqref{vn-un-int} with respect to $x$, we have
\begin{align}\nonumber
\pa_x\left({\tilde v_n(x,y)\over u_n(x,y){-c_n}}\right)
=&{\int_0^1\left(\pa_{xy}\tilde v_n(x,s_n)+\pa_y^2\tilde v_n(x,s_n)(1-t)y_n'(x)\right)dt\over \int_0^1\pa_y u_n(x,s_n)dt}\\\label{vn-un-int-derivative-to-x}
&-{\int_0^1\pa_y\tilde v_n(x,s_n)dt\int_0^1\left(\pa_{xy}u_n(x,s_n)+\pa_y^2u_n(x,s_n)(1-t)y_n'(x)\right)dt\over \left(\int_0^1\pa_y u_n(x,s_n)dt\right)^2}.
\end{align}
Since $u_n(x,y_n(x))={c_n}$ for $x\in[0,T]$, we have $y_n'(x)=-\pa_xu_n(x,y_n(x))/\pa_yu_n(x,y_n(x))$. By
\eqref{un-omegan-C bound-nd} and \eqref{un-xy-bound-nd}, we have
\begin{align}\label{yn-derivative-un-t}
\|y_n'\|_{L^\infty(0,T)}\leq C\;\text{and}\;
\left\|\int_0^1\left(\pa_{xy}u_n(x,s_n)+\pa_y^2u_n(x,s_n)(1-t)y_n'(x)\right)dt\right\|_{L^\infty(D_T)}\leq C.
\end{align}
 This, along with \eqref{un-derivative-to-y} and \eqref{I1-4-sum}, gives
\begin{align}\nonumber
\left\|\pa_x\left({\tilde v_n\over u_n{-c_n}}\right)\right\|_{L^2(D_T)}
\leq&C\left\|\int_0^1\pa_{xy}\tilde v_n(x,s_n)dt\right\|_{L^2(D_T)}+C\left\|\int_0^1\pa_y^2\tilde v_n(x,s_n)(1-t)y_n'(x)dt\right\|_{L^2(D_T)}\\\nonumber
&+C\left\|\int_0^1\pa_y\tilde v_n(x,s_n)dt\right\|_{L^2(D_T)}\\\label{tilde-vn-over-un-derivative-to-x}
\leq&C\left\|\int_0^1|\pa_{xy}\tilde v_n(x,s_n)|dt\right\|_{L^2(D_T)}+C\left\|\int_0^1|\pa_y^2\tilde v_n(x,s_n)|dt\right\|_{L^2(D_T)}
+C.
\end{align}
Let $\Gamma_1\tilde v_n=\pa_{xy}\tilde v_n$ and $\Gamma_2\tilde v_n=\pa_{y}^2\tilde v_n$.
We prove that for $i=1,2$,
\begin{align}\label{Gamma12-L2}
\left\|\int_0^1|\Gamma_i\tilde v_n(x,s_n)|dt\right\|_{L^2(D_T)}
\leq& C \|\tilde v_n\|_{H^3(D_T)}^\alpha+
C,
\end{align}
where $\alpha\in(0,1)$. By the same computation as in \eqref{I1-4}, we have
\begin{align}\label{Gamma-I1-4}
\left\|\int_0^1|\Gamma_i\tilde v_n(x,s_n)|dt\right\|_{L^2(D_T)}^2\leq I_1(\Gamma_i\tilde v_n)+I_2(\Gamma_i\tilde v_n)+I_3(\Gamma_i\tilde v_n)+I_4(\Gamma_i\tilde v_n)
\end{align}
for $i=1,2$. Since $\tilde v_n\in C^2(D_T)$, for fixed $n\geq1$ and $x\in[0,T]$ we have
\begin{align*}
\lim_{y\to y_n(x)}\ln|y_n(x)-y|\left|\int_y^{y_n(x)}|\Gamma_i\tilde v_n(x,s_n)|^2ds_n\right|=0.
\end{align*}
Since $|y_n(x)|<{1-\frac{\delta}{2}}$ for $x\in[0,T]$ and $n$ sufficiently large, by \eqref{tilde vn-H2 bound} we have
\begin{align}\label{Gamma-I13}
|I_1(\Gamma_i\tilde v_n)|+|I_3(\Gamma_i\tilde v_n)|
\leq C\|\Gamma_i\tilde v_n\|_{L^2(D_T)}^2\leq C\|\tilde v_n\|_{H^2(D_T)}^2\leq C.
\end{align}
Letting $q=2, r=2, m=1, n=2, j=0, \alpha\in(0,1)$, $G=D_T$ and $u=\Gamma_i\tilde v_n$ in Lemma \ref{Gagliardo-Nirenberg interpolation inequality}, we have by \eqref{tilde vn-H2 bound} that
\begin{align*}
\|\Gamma_i\tilde v_n\|_{L^{p}(D_T)}\leq& C\|D^1(\Gamma_i\tilde v_n)\|_{L^2(D_T)}^\alpha\|\Gamma_i\tilde v_n\|_{L^2(D_T)}^{1-\alpha}+C\|\Gamma_i\tilde v_n\|_{L^2(D_T)}\\
\leq &C\|\tilde v_n\|_{H^3(D_T)}^\alpha\|\tilde v_n\|_{H^2(D_T)}^{1-\alpha}+C\|\tilde v_n\|_{H^2(D_T)}\leq C\|\tilde v_n\|_{H^3(D_T)}^\alpha+C,
\end{align*}
where $p={2\over1-\alpha}\in(2,\infty)$. Let $p_0=\left({p\over2}\right)'={p\over p-2}>1.$ Then
\begin{align}\nonumber
|I_2(\Gamma_i\tilde v_n)|
\leq &\left(\int_0^T\int_{-1}^{y_n(x)}|\ln(y_n(x)-y)|^{p_0}dydx\right)^{1\over p_0} \left(\int_0^T\int_{-1}^{y_n(x)}|\Gamma_i\tilde v_n(x,y)|^{p}dydx\right)^{2\over p}\\\label{Gamma-I2}
\leq& \left(\int_0^T\int_{0}^2|\ln \tau|^{p_0}d\tau dx\right)^{1\over p_0}\|\Gamma_i\tilde v_n\|_{L^p(D_T)}^2\leq C\|\tilde v_n\|_{H^3(D_T)}^{2\alpha}+C.
\end{align}
A similar argument gives
\begin{align}\label{Gamma-I4}
|I_4(\Gamma_i\tilde v_n)|
\leq& C\|\tilde v_n\|_{H^3(D_T)}^{2\alpha}+C.
\end{align}
By \eqref{Gamma-I1-4}, \eqref{Gamma-I13}, \eqref{Gamma-I2} and \eqref{Gamma-I4}, we have
\begin{align*}
\left\|\int_0^1|\Gamma_i\tilde v_n(x,s_n)|dt\right\|_{L^2(D_T)}^2\leq C\|\tilde v_n\|_{H^3(D_T)}^{2\alpha}+C,
\end{align*}
 and thus,  \eqref{Gamma12-L2} holds for $i=1,2$. Then we deduce from   \eqref{tilde-vn-over-un-derivative-to-x} and \eqref{Gamma12-L2}  that
 \begin{align}\label{tilde-vn-over-un-derivative-to-x-L2-sum}
\left\|\pa_x\left({\tilde v_n\over u_n{-c_n}}\right)\right\|_{L^2(D_T)}
{\leq}&C\|\tilde v_n\|_{H^3(D_T)}^{\alpha}+C.
\end{align}
By \eqref{derivative-to-x-norm}, \eqref{tilde-vn-over-un-derivative-to-x-L2-sum} and \eqref{tilde vn-H2 bound}, we have
\begin{align}
\|\pa_x\Delta\tilde v_n\|_{L^2(D_T)}\leq&C\left\|{ \tilde v_n}\right\|_{H^1(D_T)}+C \|\tilde v_n\|_{H^3(D_T)}^\alpha+
C\leq C \|\tilde v_n\|_{H^3(D_T)}^\alpha+
C.\label{Delta-tilde-vn-derivative-to-x}
\end{align}
Thus, by the periodic
boundary condition in $x$ and Dirichlet boundary condition in $y$ of $\tilde v_n$,  we infer from \eqref{Delta-tilde-vn-derivative-to-y} and \eqref{Delta-tilde-vn-derivative-to-x} that
\begin{align*}
\|\tilde v_n\|_{H^3(D_T)}\leq C\|\pa_y\Delta\tilde v_n\|_{L^2(D_T)}+C\|\pa_x\Delta\tilde v_n\|_{L^2(D_T)}\leq  C \|\tilde v_n\|_{H^3(D_T)}^\alpha+
C,
\end{align*}
where $\alpha\in(0,1)$.
Thus,
\begin{align}\label{vn-H3-bound}
\|\tilde v_n\|_{H^3(D_T)}\leq C, \;\left\|\pa_x\left({\tilde v_n\over u_n{-c_n}}\right)\right\|_{L^2(D_T)}\leq C,\;\left\|\int_0^1|\Gamma_i\tilde v_n(x,s_n)|dt\right\|_{L^2(D_T)}
\leq C
\end{align}
for $i=1,2$.

Then we  prove the uniform $H^4$ bound for $\tilde v_n$, $n\geq 1$. Taking  derivative of  \eqref{tilde-vn-eq} with respect to $x$ twice, we have
\begin{align*}
\pa_x^2\Delta\tilde  v_n+\pa_{xxy}\omega_n{\tilde v_n\over u_n{-c_n}}+2\pa_{xy}\omega_n\pa_x\left({\tilde v_n\over u_n{-c_n}}\right)+\left(\pa_y\omega_n+\beta\right)\pa_x^2\left({\tilde v_n\over u_n{-c_n}}\right)=0.
\end{align*}
By \eqref{tilde vn-over-un-L2}, \eqref{vn-overun-derivative-L2-sum} and \eqref{vn-H3-bound}, we have
\begin{align}\label{tilde-vn-over-un-H1-L4}
\left\|{\tilde v_n\over u_n{-c_n}}\right\|_{H^1(D_T)}\leq C\Longrightarrow\left\|{\tilde v_n\over u_n{-c_n}}\right\|_{L^4(D_T)}\leq C\left\|{\tilde v_n\over u_n{-c_n}}\right\|_{H^1(D_T)}\leq C.
\end{align}
Then by \eqref{omega-n-Lp-nd}, \eqref{un-omegan-C bound-nd} and \eqref{tilde-vn-over-un-H1-L4}, we have
\begin{align}\nonumber
\left\|\pa_x^2\Delta\tilde  v_n\right\|_{L^2(D_T)}
\leq&\|\pa_{xxy}\omega_n\|_{L^4(D_T)}\left\|{\tilde v_n\over u_n{-c_n}}\right\|_{L^4(D_T)}+2\|\pa_{xy}\omega_n\|_{C^0(D_T)}\left\|\pa_x\left({\tilde v_n\over u_n{-c_n}}\right)\right\|_{L^2(D_T)}\\\nonumber
&+\|\pa_y\omega_n+\beta\|_{C^0(D_T)}\left\|\pa_x^2\left({\tilde v_n\over u_n{-c_n}}\right)\right\|_{L^2(D_T)}\\\label{Delta-tilde-v-n-derivative-to-x2-norm}
\leq& C+C\left\|\pa_x^2\left({\tilde v_n\over u_n{-c_n}}\right)\right\|_{L^2(D_T)}.
\end{align}
By \eqref{vn-un-int-derivative-to-x}, we have
\begin{align}\nonumber
&\pa_x^2\left({\tilde v_n(x,y)\over u_n(x,y){-c_n}}\right)
\\\nonumber
=&{\int_0^1(\pa_{xxy}\tilde v_n+2\pa_{xyy}\tilde v_n(1-t)y_n'+\pa_{y}^3\tilde v_n(1-t)^2y_n'^2+\pa_y^2\tilde v_n(1-t)y_n'')dt\over\int_0^1\pa_yu_ndt}\\\nonumber
&-2{\int_0^1(\pa_{xy}\tilde v_n+\pa_y^2\tilde v_n(1-t)y_n')dt\int_0^1(\pa_{xy}u_n+\pa_y^2u_n(1-t)y_n')dt\over \left(\int_0^1\pa_yu_ndt\right)^2}\\\nonumber
&-{\int_0^1\pa_y\tilde v_ndt\int_0^1(\pa_{xxy} u_n+2\pa_{xyy}u_n(1-t)y_n'+\pa_{y}^3u_n(1-t)^2y_n'^2+\pa_y^2u_n(1-t)y_n'')dt\over \left(\int_0^1\pa_yu_ndt\right)^2 }\\\nonumber
&+2{\int_0^1\pa_y\tilde v_ndt\left(\int_0^1(\pa_{xy}u_n+\pa_y^2u_n(1-t)y_n')dt\right)^2\over \left(\int_0^1\pa_yu_ndt\right)^3}\\
=&II_1+II_2+II_3+II_4.\label{vn-un-int-derivative-to-x2}
\end{align}
 Since $u_n(x,y_n(x))={c_n}$ for $x\in[0,T]$, we have
\begin{align*}
y_n''(x)=-{\pa_x^2u_n(x,y_n(x))+2\pa_{xy}u_n(x,y_n(x))y_n'(x)+\pa_y^2u_n(x,y_n(x))(y_n'(x))^2\over\pa_y u_n(x,y_n(x))},
\end{align*}
and thus, we deduce from
\eqref{un-omegan-C bound-nd}, \eqref{un-xy-bound-nd} and
\eqref{yn-derivative-un-t} that
\begin{align}\label{yn-derivative2-un-t}&\|y_n''\|_{L^\infty(0,T)}\leq C,\\
\label{un-int3-t}
&\left\|\int_0^1(\pa_{xxy} u_n+2\pa_{xyy}u_n(1-t)y_n'+\pa_{y}^3u_n(1-t)^2y_n'^2+\pa_y^2u_n(1-t)y_n'')dt\right\|_{L^\infty(D_T)}\leq C.
\end{align}
By \eqref{un-derivative-to-y},
\eqref{yn-derivative-un-t}, \eqref{yn-derivative2-un-t} and \eqref{un-int3-t}, we have
\begin{align}\nonumber
\|II_1\|_{L^2(D_T)}\leq&C\left\|\int_0^1|\pa_{xxy}\tilde v_n|dt\right\|_{L^2(D_T)}+C\left\|\int_0^1|\pa_{xyy}\tilde v_n|dt\right\|_{L^2(D_T)}\\\nonumber
&+C\left\|\int_0^1|\pa_{y}^3\tilde v_n|dt\right\|_{L^2(D_T)}+C\left\|\int_0^1|\pa_y^2\tilde v_n|dt\right\|_{L^2(D_T)},\\\nonumber
\|II_2\|_{L^2(D_T)}\leq&C\left\|\int_0^1|\pa_{xy}\tilde v_n|dt\right\|_{L^2(D_T)}+C\left\|\int_0^1|\pa_{y}^2\tilde v_n|dt\right\|_{L^2(D_T)},\\
\|II_j\|_{L^2(D_T)}\leq&C\left\|\int_0^1|\pa_{y}\tilde v_n|dt\right\|_{L^2(D_T)}\label{II1-4}
\end{align}
for $j=3,4$.
Let $\tilde\Gamma_1\tilde v_n=\pa_{xxy}\tilde v_n$, $\tilde\Gamma_2\tilde v_n=\pa_{xyy}\tilde v_n$ and $\tilde\Gamma_3\tilde v_n=\pa_{y}^3\tilde v_n$.
We prove that for $i=1,2,3$,
\begin{equation}\label{tilde-Gamma-vn}
\left\|\int_0^1|\tilde\Gamma_i\tilde v_n(x,s_n)|dt\right\|_{L^2(D_T)}\leq C \|\tilde v_n\|_{H^4(D_T)}^\alpha+C,
\end{equation}
where $\alpha\in(0,1)$. Based on the uniform $H^3$ bound for $\tilde v_n(n\geq1)$, the proof is similar as \eqref{Gamma12-L2}, and we give it here for completeness.
In fact, we have
\begin{align}\label{tilde-Gamma-I1-4}
\left\|\int_0^1|\tilde\Gamma_i\tilde v_n(x,s_n)|dt\right\|_{L^2(D_T)}^2\leq I_1(\tilde\Gamma_i\tilde v_n)+I_2(\tilde\Gamma_i\tilde v_n)+I_3(\tilde\Gamma_i\tilde v_n)+I_4(\tilde\Gamma_i\tilde v_n),\\\label{tilde-Gamma-I13}
|I_1(\tilde\Gamma_i\tilde v_n)|+|I_3(\tilde\Gamma_i\tilde v_n)|
\leq C\|\tilde\Gamma_i\tilde v_n\|_{L^2(D_T)}^2\leq C\|\tilde v_n\|_{H^3(D_T)}^2\leq C,
\end{align}
where we used $\tilde v_n\in C^3(D_T)$.
By Lemma \ref{Gagliardo-Nirenberg interpolation inequality}, we deduce from  \eqref{vn-H3-bound} that
\begin{align*}
\|\tilde\Gamma_i\tilde v_n\|_{L^{p}(D_T)}\leq& C\|D^1(\tilde\Gamma_i\tilde v_n)\|_{L^2(D_T)}^\alpha\|\tilde\Gamma_i\tilde v_n\|_{L^2(D_T)}^{1-\alpha}+C\|\tilde\Gamma_i\tilde v_n\|_{L^2(D_T)}\\
\leq &C\|\tilde v_n\|_{H^4(D_T)}^\alpha\|\tilde v_n\|_{H^3(D_T)}^{1-\alpha}+C\|\tilde v_n\|_{H^3(D_T)}\leq C\|\tilde v_n\|_{H^4(D_T)}^\alpha+C,
\end{align*}
where $p={2\over1-\alpha}\in(2,\infty)$. Let $p_0=\left({p\over2}\right)'={p\over p-2}>1.$ Then
\begin{align}\label{tilde-Gamma-I24}
|I_{j}(\tilde \Gamma_i\tilde v_n)|
\leq& \left(\int_0^T\int_{0}^2|\ln \tau|^{p_0}d\tau dx\right)^{1\over p_0}\|\tilde\Gamma_i\tilde v_n\|_{L^p(D_T)}^2\leq C\|\tilde v_n\|_{H^4(D_T)}^{2\alpha}+C
\end{align}
for $j=2,4$.
By \eqref{tilde-Gamma-I1-4}, \eqref{tilde-Gamma-I13} and \eqref{tilde-Gamma-I24}, we have
\begin{align*}
\left\|\int_0^1|\tilde\Gamma_i\tilde v_n(x,s_n)|dt\right\|_{L^2(D_T)}^2\leq C\|\tilde v_n\|_{H^4(D_T)}^{2\alpha}+C,
\end{align*}
 and thus,  \eqref{tilde-Gamma-vn} holds for $i=1,2,3$. By \eqref{II1-4}, \eqref{tilde-Gamma-vn}, \eqref{vn-H3-bound} and \eqref{I1-4-sum},  we have
\begin{align*}
\|II_1\|_{L^2(D_T)}\leq C\|\tilde v_n\|_{H^4(D_T)}^{\alpha}+C,\;\|II_j\|_{L^2(D_T)}\leq C
\end{align*}
for $j=2,3,4$, and thus, it follows from \eqref{vn-un-int-derivative-to-x2} and \eqref{Delta-tilde-v-n-derivative-to-x2-norm}  that
\begin{align}\label{vn-un-int-derivative-to-x2-L2}
\left\|\pa_x^2\left({\tilde v_n\over u_n{-c_n}}\right)\right\|_{L^2(D_T)}\leq C\|\tilde v_n\|_{H^4(D_T)}^{\alpha}+C\Longrightarrow\left\|\pa_x^2\Delta\tilde  v_n\right\|_{L^2(D_T)}\leq C\|\tilde v_n\|_{H^4(D_T)}^{\alpha}+C.
\end{align}

Taking  derivative of  \eqref{tilde-vn-eq} with respect to $y$ twice, we have
\begin{align*}
\pa_y^2\Delta\tilde  v_n+\pa_{y}^3\omega_n{\tilde v_n\over u_n{-c_n}}+2\pa_{y}^2\omega_n\pa_y\left({\tilde v_n\over u_n{-c_n}}\right)+\left(\pa_y\omega_n+\beta\right)\pa_y^2\left({\tilde v_n\over u_n{-c_n}}\right)=0.
\end{align*}
Then we deduce from \eqref{omega-n-Lp-nd}, \eqref{un-omegan-C bound-nd} and \eqref{tilde-vn-over-un-H1-L4} that
\begin{align}\nonumber
\left\|\pa_y^2\Delta\tilde  v_n\right\|_{L^2(D_T)}
\leq&\|\pa_{y}^3\omega_n\|_{L^4(D_T)}\left\|{\tilde v_n\over u_n{-c_n}}\right\|_{L^4(D_T)}+2\|\pa_{y}^2\omega_n\|_{C^0(D_T)}\left\|\pa_y\left({\tilde v_n\over u_n{-c_n}}\right)\right\|_{L^2(D_T)}\\\nonumber
&+\|\pa_y\omega_n+\beta\|_{C^0(D_T)}\left\|\pa_y^2\left({\tilde v_n\over u_n{-c_n}}\right)\right\|_{L^2(D_T)}\\\label{Delta-tilde-v-n-derivative-to-y2-norm}
\leq& C+C\left\|\pa_y^2\left({\tilde v_n\over u_n{-c_n}}\right)\right\|_{L^2(D_T)}.
\end{align}
By \eqref{vn-un-int} and direct computation, we have
\begin{align}\nonumber
\pa_{y}^2\left({\tilde v_n(x,y)\over u_n(x,y){-c_n}}\right)
=&{\int_0^1\pa_{y}^3\tilde v_nt^2dt\over\int_0^1\pa_yu_ndt}-2{\int_0^1\pa_{y}^2\tilde v_ntdt\int_0^1\pa_y^2u_ntdt\over \left(\int_0^1\pa_yu_ndt\right)^2}\\\nonumber
&-{\int_0^1\pa_y\tilde v_ndt\int_0^1\pa_y^3 u_nt^2dt\over \left(\int_0^1\pa_yu_ndt\right)^2 }
+2{\int_0^1\pa_y\tilde v_ndt\left(\int_0^1\pa_{y}^2u_ntdt\right)^2\over \left(\int_0^1\pa_yu_ndt\right)^3}\\\label{vn-un-int-derivative-to-y2}
=&III_1+III_2+III_3+III_4.
\end{align}
By \eqref{un-omegan-C bound-nd}, we have $\left\|\int_0^1\pa_y^3 u_nt^2dt\right\|_{L^\infty(D_T)}\leq C$, and thus, we deduce from
 \eqref{un-derivative-to-y},  \eqref{tilde-Gamma-vn}, \eqref{vn-H3-bound}  and \eqref{I1-4-sum} that
 \begin{align}\nonumber
\|III_1\|_{L^2(D_T)}\leq& C\left\|\int_0^1|\pa_{y}^3\tilde v_n|dt\right\|_{L^2(D_T)}\leq C \|\tilde v_n\|_{H^4(D_T)}^\alpha+C,\\\label{IV1-4-estimate}
\|III_j\|_{L^2(D_T)}\leq&C\left\| \int_0^1|\pa_y^2\tilde v_n|dt\right\|_{L^2(D_T)}+C\left\| \int_0^1|\pa_y\tilde v_n|dt\right\|_{L^2(D_T)}\leq C
\end{align}
for $j=2,3,4$.
 Here we omit the variables $(x,s_n)$ for the derivatives  of $\tilde v_n$, $u_n$ and the variable $x$ for the derivative of $y_n$ in
 \eqref{vn-un-int-derivative-to-x2}, \eqref{un-int3-t}, \eqref{II1-4},  \eqref{vn-un-int-derivative-to-y2} and \eqref{IV1-4-estimate}.
By \eqref{vn-un-int-derivative-to-y2}, \eqref{IV1-4-estimate} and \eqref{Delta-tilde-v-n-derivative-to-y2-norm}, we have
\begin{align}\label{vn-un-int-derivative-to-y2-L2}
 \left\|\pa_{y}^2\left({\tilde v_n\over u_n{-c_n}}\right)\right\|_{L^2(D_T)}\leq  C \|\tilde v_n\|_{H^4(D_T)}^\alpha+C\Longrightarrow\left\|\pa_{y}^2\Delta\tilde  v_n\right\|_{L^2(D_T)}\leq C \|\tilde v_n\|_{H^4(D_T)}^\alpha+C.
\end{align}

In summary, by the  periodic
boundary condition in $x$ and Dirichlet boundary condition in $y$ of $\tilde v_n$, we deduce from \eqref{vn-un-int-derivative-to-x2-L2} and \eqref{vn-un-int-derivative-to-y2-L2} that
\begin{align}\nonumber
&\|\tilde v_n\|_{H^4(D_T)}\leq C\|\pa_{x}^2\Delta\tilde v_n\|_{L^2(D_T)}+C\|\pa_{y}^2\Delta\tilde v_n\|_{L^2(D_T)}
\leq  C \|\tilde v_n\|_{H^4(D_T)}^\alpha+
C,
\end{align}
where $\alpha\in(0,1)$.
Thus, we have
\begin{align}\label{vn-H4-bound}
\|\tilde v_n\|_{H^4(D_T)}\leq C.
\end{align}

 By \eqref{vn-H4-bound}, there exists $\tilde v_0\in H^4(D_T)$ such that up to a subsequence, $\tilde v_n\rightharpoonup\tilde v_0 $ in $H^4(D_T)$,   $\tilde v_n\rightarrow\tilde v_0 $ in $C^2(D_T)$ and $\|\tilde v_0\|_{L^2(D_T)}=1$. Since
 $|\tilde{v}_0(x,{c_0})|=|\tilde{v}_n(x,y_n(x))-\tilde{v}_0(x,{c_0})|\leq |\tilde{v}_n(x,y_n(x))-\tilde{v}_n(x,{c_0})|+|\tilde{v}_n(x,{c_0})-\tilde{v}_0(x,{c_0})|\leq \|\pa_y \tilde{v}_n\|_{L^\infty(D_T)}|y_n(x){-c_0}|+|\tilde{v}_n(x,{c_0})-\tilde{v}_0(x,{c_0})|\to 0,$ we have $\tilde{v}_0(x,{c_0})=0$ for $x\in[0,T]$.
We claim that ${\tilde v_n\over u_n{-c_n}}\rightarrow {\tilde v_0\over y{-c_0}} $ in $C^0(D_T)$.
In fact,  by \eqref{un-derivative-to-y}, \eqref{yn(x)-0uniform}, \eqref{vn-H4-bound}, \eqref{un-y-C3-nd} and the fact that $\tilde{v}_n\rightarrow\tilde{v}_0$ in $C^1(D_T)$, we have
\begin{align*}
&\left|\frac{\tilde{v}_n(x,y)}{u_n(x,y){-c_n}}-\frac{\tilde{v}_0(x,y)}{y{-c_0}}\right|
=\left|\frac{\int_0^1\pa_y \tilde{v}_n(x,s_n)dt}{\int_0^1\pa_y u_n(x,s_n)dt}-\int_0^1\pa_y \tilde{v}_0(x,ty{+(1-t)c_0})dt\right|\\
\leq&\left|\int_0^1\pa_y \tilde{v}_n(x,s_n)dt\right|\cdot \left|\frac{1}{\int_0^1\pa_y u_n(x,s_n)dt}-1\right|
+\int_0^1|\pa_y \tilde{v}_n(x,s_n)-\pa_y \tilde{v}_n(x,ty{+(1-t)c_0})|dt\\
&+\int_0^1|\pa_y\tilde{v}_n(x,ty{+(1-t)c_0})-\pa_y \tilde{v}_0(x,ty{+(1-t)c_0})|dt\\
\leq&C\|\pa_y \tilde{v}_n\|_{L^\infty(D_T)}\|\pa_y u_n-1\|_{C^0(D_T)}+\|\pa_{y}^2\tilde{v}_n\|_{L^\infty(D_T)}|y_n(x){-c_0}|+\|\pa_y\tilde{v}_n-\pa_y\tilde{v}_0\|_{C^0(D_T)}\\
\leq&C\| \tilde{v}_n\|_{H^3(D_T)}\| u_n-y\|_{C^1(D_T)}+C\| \tilde{v}_n\|_{H^4(D_T)}|y_n(x){-c_0}|+\|\tilde{v}_n-\tilde{v}_0\|_{C^1(D_T)}\\
\leq&C\| u_n-y\|_{C^1(D_T)}+C|y_n(x){-c_0}|+\|\tilde{v}_n-\tilde{v}_0\|_{C^1(D_T)}\\
\to&0,\text{ as } n \to \infty
\end{align*}
uniformly for $(x,y)\in D_T$. By \eqref{omegan-C2-nd}, $\pa_y\omega_n\to0$ in $C^0(D_T)$. Since $\tilde v_n\rightarrow\tilde v_0 $ in $C^2(D_T)$ and ${\tilde v_n\over u_n{-c_n}}\rightarrow {\tilde v_0\over y{-c_0}} $ in $C^0(D_T)$, we now send $n\to\infty$ in \eqref{tilde-vn-eq} to get
\begin{align*}
\Delta \tilde v_0+{\beta\over y{-c_0}}\tilde v_0=0.
\end{align*}
Since $\tilde v_0$ is $T$-periodic in $x$, we have
 $\tilde v_0(x,y)=\sum_{k\in\mathbf{Z}}\widehat{\tilde v}_{{0},k}(y)e^{ik\alpha x}\neq0$,
 where $\alpha={2\pi\over T}$. Since $\tilde v_0\in H^4(D_T)$ and $\|\tilde v_0\|_{L^2(D_T)}=1$, there exists $k_0\in\mathbf{Z}$ such that
$0\neq\widehat{\tilde v}_{{0},k_0}\in H^4(-1,1)$ solves
\begin{align*}
-\widehat{\tilde v}_{{0},k_0}''-{\beta\over y}\widehat{\tilde v}_{{0},k_0}=-(k_0\alpha)^2\widehat{\tilde v}_{{0},k_0}, \quad \widehat{\tilde v}_{{0},k_0}(\pm1)=0.
\end{align*}
 If $k_0\neq 0$, this is a contradiction to Theorem 1 in \cite{LYZ}. If $k_0= 0$,
since $\tilde v_n=v_n/\|v_n\|_{L^2(D_T)}=-\partial_x\psi_n/\|v_n\|_{L^2(D_T)}$, we have
 ${1\over T}\int_0^T\tilde v_ndx=0$. Taking limit as $n\to\infty$, we have ${1\over T}\int_0^T\tilde v_0dx=0$ and thus, $\widehat{\tilde v}_{{0},0}(y)={1\over T}\int_0^T\tilde v_n(x,y)dx\equiv0$, which is also a contradiction.
 \end{proof}
 \begin{remark}
  $(1)$ ${\tilde v_n\over u_n-c_n}, n\geq1,$ have a uniform $H^2$ bound due to \eqref{tilde-vn-over-un-H1-L4}, \eqref{vn-un-int-derivative-to-x2-L2}, \eqref{vn-un-int-derivative-to-y2-L2} and \eqref{vn-H4-bound}.

 $(2)$ The uniform bound of  the left hand side of \eqref{derivative-y2} can   be also obtained from \eqref{vn-H3-bound}. Here, we use the factor $t$ in \eqref{til-v-derivative-2}  to give a direct proof of \eqref{derivative-y2} since the left hand side of \eqref{derivative-y2} can be  bounded  by the $H^2$ bound of $\tilde v_n$, $n\geq1$.
 \end{remark}
\section{The   principal eigenvalues  of  the singular  Rayleigh-Kuo BVP}
In this section, we study the   principal eigenvalues  of  the singular  Rayleigh-Kuo BVP  for Couette flow, which determine the borderlines $\Gamma_-\cup\Gamma_+$ in Figure 1. Dynamics  near  Couette flow is different if  $(\alpha,\beta)$ passes through the borderlines (see Theorem \ref{thm-non-shear-nodelta2}).

First, we  prove the existence, monotonicity and continuity  of the principal eigenvalues $\lambda_1(\beta,c)$ of  the singular  Rayleigh-Kuo BVP \eqref{sturm-Liouville} with $u(y)=y$, $c=\pm1$.
We mainly discuss  $\beta\geq0$ and $c=-1$, and give the analog result by symmetry  for $\beta\leq0$ and $c=1$.
Recall that $\lambda_1(\beta,-1)$ is precisely defined in  \eqref{lambda-beta}.
Then we have the following basic property of $\lambda_1(\beta,-1)$.


\begin{lemma}\label{eigenfunction} Let $\beta\geq0$. Then
the infimum in \eqref{lambda-beta} is attained by some function $\phi_\beta\in H_0^1(-1,1),\|\phi_\beta\|_{L^2(-1,1)}=1$. Moreover,
 there exists $m_\beta>-\infty$ such that
$m_\beta\leq\lambda_1(\beta,-1)\leq {\pi^2\over4}$.
\end{lemma}

\begin{proof}
Let $\{\phi_n\}$ be a minimizing sequence with $\|\phi_n\|_{L^2(-1,1)}=1$ and
\begin{align}\label{minimizing sequence}
\int_{-1}^1\left(|\phi_n'|^2-{\beta\over y+1}|\phi_n|^2\right)dy\to \lambda_1(\beta,-1) \quad\text{as}\quad n\to \infty.
\end{align}
Let $q=2$, $m=1$, $r=2$, $j=0$, $p=4$, $\alpha={1\over4}$, $G=(-1,1)$ and $u=\phi_n$  in Lemma \ref{Gagliardo-Nirenberg interpolation inequality}, we have
\begin{align}\label{GN-APP}
\|\phi_n\|_{L^4(-1,1)}\leq C_1\|\phi_n\|_{L^2(-1,1)}^{3\over4}\|\phi_n'\|_{L^2(-1,1)}^{1\over4}+C_2\|\phi_n\|_{L^2(-1,1)}.
\end{align}
Since $|\ln(y+1)||\phi_n(y)|^2\leq(y+1) |\ln(y+1)|\|\phi_n'\|_{L^2(-1,1)}^2\to0$ as $y\to-1^+$ for fixed $n\geq1$,  we have by \eqref{GN-APP} that
\begin{align}\nonumber
\int_{-1}^1{\beta\over y+1}|\phi_n|^2dy&=\beta\left(\ln(y+1)|\phi_n(y)|^2\big|_{y=-1}^1-\int_{-1}^1\ln(y+1)2\phi_n(y)\phi_n(y)'dy \right)\\\nonumber
&\leq 2\beta\|\ln(y+1)\|_{L^4(-1,1)}\|\phi_n\|_{L^4(-1,1)}\|\phi_n'\|_{L^2(-1,1)}
\\\nonumber
&\leq C(C_1\|\phi_n\|_{L^2(-1,1)}^{3\over4}\|\phi_n'\|_{L^2(-1,1)}^{1\over4}+C_2\|\phi_n\|_{L^2(-1,1)})\|\phi_n'\|_{L^2(-1,1)}\\\label{sec-term-estimates}
&\leq C\|\phi_n'\|_{L^2(-1,1)}^{5\over4}+C\|\phi_n\|_{L^2(-1,1)}^{1\over4}\|\phi_n'\|_{L^2(-1,1)}\leq C\|\phi_n'\|_{L^2(-1,1)}^{5\over4},
\end{align}
where we used $1=\|\phi_n\|_{L^2(-1,1)}\leq C\|\phi_n'\|_{L^2(-1,1)}$. By \eqref{minimizing sequence}, \eqref{sec-term-estimates} and the fact that  $\lambda_1(\beta,-1)\leq {\pi^2\over4}$, there exists $M>0$ such that
\begin{align}\label{u-l-bound}
M
\geq \|\phi_n'\|^2_{L^2(-1,1)}-\int_{-1}^1{\beta\over y+1}|\phi_n|^2dy\geq \|\phi_n'\|^2_{L^2(-1,1)}-C\|\phi_n'\|_{L^2(-1,1)}^{5\over4},
\end{align}
which implies
\begin{align*}
 \|\phi_n'\|^2_{L^2(-1,1)}\leq C\Longrightarrow\|\phi_n\|^2_{H^1(-1,1)}\leq C.
\end{align*}
Thus, there exists $\phi_{\beta}\in H_0^1(-1,1)$ such that up to a subsequence, $\phi_n\rightharpoonup\phi_{\beta}$ in $H_0^1(-1,1)$ and $\phi_n\to\phi_\beta$ in $L^4(-1,1)$, and
\begin{align}\label{first-term-liminf}
\|\phi_{\beta}\|_{H^1(-1,1)}^2\leq \liminf_{n\to\infty}\|\phi_n\|_{H^1(-1,1)}^2\Longrightarrow\|\phi_{\beta}'\|_{L^2(-1,1)}^2\leq \liminf_{n\to\infty}\|\phi_n'\|_{L^2(-1,1)}^2.
\end{align}
Note that
\begin{align}\label{sec-term-lim}
\lim_{n\to\infty}\int_{-1}^1{1\over y+1}|\phi_n|^2dy=\int_{-1}^1{1\over y+1}|\phi_\beta|^2dy.
\end{align}
In fact, since $|\ln(y+1)(|\phi_n(y)|^2-|\phi_\beta(y)|^2)|\leq(y+1)|\ln(y+1)|(\|\phi_n'\|_{L^2(-1,1)}^2+|\phi_\beta'\|_{L^2(-1,1)}^2)\to0$ as $y\to-1^+$, we have
\begin{align*}
\int_{-1}^1{1\over y+1}(|\phi_n|^2-|\phi_\beta|^2)dy=&\ln(y+1)(|\phi_n|^2-|\phi_\beta|^2)\big|_{y=-1}^1-\int_{-1}^12\ln(y+1)(\phi_n\phi_n'-\phi_\beta\phi_\beta')dy\\
=&-\int_{-1}^12\ln(y+1)(\phi_n-\phi_\beta)\phi_n'dy-\int_{-1}^12\ln(y+1)\phi_\beta(\phi_n'-\phi_\beta') dy\\
=&I_1+I_2.
\end{align*}
Note that
\begin{align*}
|I_1|\leq \|\ln(y+1)\|_{L^4(-1,1)}\|\phi_n-\phi_\beta\|_{L^4(-1,1)}\|\phi_n'\|_{L^2(-1,1)}\leq C\|\phi_n-\phi_\beta\|_{L^4(-1,1)}\to0.
\end{align*}
Define a functional $f$ on $H_0^1(-1,1)$ by
\begin{align*}
\langle f,\varphi\rangle=\int_{-1}^1\ln(y+1)\phi_\beta\varphi' dy,\quad\varphi\in H_0^1(-1,1).
\end{align*}
Then
\begin{align*}
&\langle f,\varphi\rangle\leq \|\ln(y+1)\|_{L^4(-1,1)}\|\phi_\beta\|_{L^4(-1,1)}\|\varphi' \|_{L^2(-1,1)}\\
\leq& C\|\phi_\beta\|_{H^1(-1,1)}\|\varphi\|_{H^1(-1,1)}\leq C\|\varphi\|_{H^1(-1,1)}.
\end{align*}
Thus, $f$ is a bounded functional on $H_0^1(-1,1)$ and $f\in H^{-1}(-1,1)$. Since $\phi_n\rightharpoonup\phi_{\beta}$ in $H_0^1(-1,1)$, we have
\begin{align*}
I_2=-2\langle f,\phi_n-\phi_\beta\rangle\to0\quad \text{as}\quad n\to\infty.
\end{align*}
This proves \eqref{sec-term-lim}.   By \eqref{first-term-liminf}, \eqref{sec-term-lim} and \eqref{minimizing sequence}, we have
\begin{align*}
\int_{-1}^1\left(|\phi_\beta'|^2-{\beta\over y+1}|\phi_\beta|^2\right)dy\leq\liminf_{n\to\infty}\int_{-1}^1\left( |\phi_n'|^2-{\beta\over y+1}|\phi_n|^2\right)dy=\lambda_1(\beta,-1).
\end{align*}
Thus,
\begin{align*}
\lambda_1(\beta,-1)=\int_{-1}^1\left(|\phi_\beta'|^2-{\beta\over y+1}|\phi_\beta|^2\right)dy.
\end{align*}
$\phi_\beta$ is called an eigenfunction of $\lambda_1(\beta,-1)$ in the sequel.

For any $\phi\in H_0^1(-1,1)$ with $\|\phi\|_{L^2(-1,1)}=1$, similar to \eqref{u-l-bound}, there exists $m_\beta>-\infty$ such that
\begin{align*}
\int_{-1}^1\left(|\phi'|^2-{\beta\over y+1}|\phi|^2\right)dy\geq \|\phi'\|^2_{L^2(-1,1)}-C_{\beta}\|\phi'\|_{L^2(-1,1)}^{5\over4}\geq m_\beta.
\end{align*}
For $\beta>0$, we have  $\lambda_1(\beta,-1)\leq \lambda_1(0,-1)={\pi^2\over4}$. Thus,
\begin{align}\label{lambda-bound}m_\beta\leq\lambda_1(\beta,-1)\leq {\pi^2\over4},\end{align}
which implies that $\lambda_1(\beta,-1)$ is  finite  for $\beta>0$.
\end{proof}
Then we get the monotonicity of $\lambda_1(\cdot,-1)$.
\begin{Corollary}\label{lambda-decreasing}
$\lambda_1(\cdot,-1)$ is decreasing on $[0,\infty)$.
\end{Corollary}
\begin{proof}
Let $0\leq \beta_1<\beta_2<\infty$.
By Lemma \ref{eigenfunction}, $\lambda_1(\beta_1,-1)$ is attained  by  $\phi_{\beta_1}\in H_0^1(-1,1)$ with $\|\phi_{\beta_1}\|_{L^2}=1$.
Then
\begin{align*}
\lambda_1(\beta_1,-1)=\int_{-1}^1\left(|\phi_{\beta_1}'|^2-{\beta_1\over y+1}|\phi_{\beta_1}|^2\right)dy>\int_{-1}^1\left(|\phi_{\beta_1}'|^2-{\beta_2\over y+1}|\phi_{\beta_1}|^2\right)dy\geq \lambda_1(\beta_2,-1).
\end{align*}
\end{proof}
Next, we consider the continuity of $\lambda_1(\cdot,-1)$.
\begin{lemma}\label{lambda-continuous}
$\lambda_1(\cdot,-1)$ is continuous on $[0,\infty)$.
\end{lemma}
\begin{proof}
First, we consider the left continuity of $\lambda_1(\cdot,-1)$ at $\beta_0\in(0,\infty)$. By Lemma \ref{eigenfunction}, $\lambda_1(\beta_0,-1)$ is attained  by  $\phi_{\beta_0}\in H_0^1(-1,1)$. Since $|\phi_{\beta_0}(y)|^2\leq \|\phi_{\beta_0}'\|_{L^2(-1,1)}^2(y+1)$, we have $\int_{-1}^1{1\over y+1}|\phi_{\beta_0}|^2dy\leq 2\|\phi_{\beta_0}'\|_{L^2(-1,1)}^2<\infty$. For $\varepsilon>0$, there exists $\delta>0$ such that $\delta\int_{-1}^1{1\over y+1}|\phi_{\beta_0}|^2dy<\varepsilon$. By the monotonicity of $\lambda_1(\cdot,-1)$, we have
 \begin{align*}0\leq&\lambda_1(\beta,-1)- \lambda_1(\beta_0,-1)\\
 \leq&\int_{-1}^1\left(|\phi_{\beta_0}'|^2-{\beta\over y+1}|\phi_{\beta_0}|^2\right)dy-\int_{-1}^1\left(|\phi_{\beta_0}'|^2-{\beta_0\over y+1}|\phi_{\beta_0}|^2\right)dy\\
 =&(\beta_0-\beta)\int_{-1}^1{1\over y+1}|\phi_{\beta_0}|^2dy<\delta\int_{-1}^1{1\over y+1}|\phi_{\beta_0}|^2dy<\varepsilon\end{align*}
  for $0<\beta_0-\beta<\delta$.

 Next, we consider the right continuity of $\lambda_1(\cdot,-1)$ at $\beta_0\in[0,\infty)$. We claim that
 \begin{align}\label{phi-beta-bound}
 \|\phi_{\beta}'\|_{L^2(-1,1)}\leq C \quad\text{for}\quad \beta\in[\beta_0,\beta_0+1].
 \end{align}
 In fact, since $\lim_{y\to-1^+}\ln(y+1)|\phi_\beta(y)|^2=0$, by Lemma \ref{Gagliardo-Nirenberg interpolation inequality} and $1=\|\phi_\beta\|_{L^2(-1,1)}\leq C\|\phi_\beta'\|_{L^2(-1,1)}$ we have
 \begin{align*}
 &\int_{-1}^1{1\over y+1}|\phi_\beta|^2dy=-\int_{-1}^12\ln(y+1)\phi_\beta\phi_\beta'dy\leq \|\ln(y+1)\|_{L^4(-1,1)}\|\phi_\beta\|_{L^4(-1,1)}\|\phi_\beta'\|_{L^2(-1,1)}\\
 \leq &C(C_1\|\phi_\beta\|_{L^2(-1,1)}^{3\over4}\|\phi_\beta'\|_{L^2(-1,1)}^{1\over4}+C_2\|\phi_\beta\|_{L^2(-1,1)})\|\phi_\beta'\|_{L^2(-1,1)}\leq C\|\phi_\beta'\|_{L^2(-1,1)}^{5\over4}.
 \end{align*}
 This, along with \eqref{lambda-bound}, gives
 \begin{align*}
{\pi^2\over4}\geq\|\phi_\beta'\|_{L^2(-1,1)}^2-\beta C\|\phi_\beta'\|_{L^2(-1,1)}^{5\over4}\geq \|\phi_\beta'\|_{L^2(-1,1)}^2-(\beta_0+1) C\|\phi_\beta'\|_{L^2(-1,1)}^{5\over4}
 \end{align*}
 for $\beta\in[\beta_0,\beta_0+1]$. This proves \eqref{phi-beta-bound}. By \eqref{phi-beta-bound},
we have $\int_{-1}^1{1\over y+1}|\phi_{\beta}|^2dy\leq 2\|\phi_{\beta}'\|_{L^2(-1,1)}^2<2C$ uniformly for $\beta\in[\beta_0,\beta_0+1]$. For $\varepsilon>0$, there exists $\delta\in(0,1)$ such that $\delta\int_{-1}^1{1\over y+1}|\phi_{\beta}|^2dy\leq 2\delta C<\varepsilon$ for $\beta\in[\beta_0,\beta_0+1]$. By the monotonicity of $\lambda_1(\cdot,-1)$, we have
 \begin{align*}0\leq&\lambda_1(\beta_0,-1)- \lambda_1(\beta,-1)\\
 \leq&\int_{-1}^1\left(|\phi_{\beta}'|^2-{\beta_0\over y+1}|\phi_{\beta}|^2\right)dy-\int_{-1}^1\left(|\phi_{\beta}'|^2-{\beta\over y+1}|\phi_{\beta}|^2\right)dy\\
 =&(\beta-\beta_0)\int_{-1}^1{1\over y+1}|\phi_{\beta}|^2dy<\delta\int_{-1}^1{1\over y+1}|\phi_{\beta}|^2dy<\varepsilon\end{align*}
  for $0<\beta-\beta_0<\delta$.
\end{proof}

For $\beta\geq0$ and $c<-1$,
\begin{align}\label{lambda1betac-1}
\lambda_1(\beta,c)=\inf_{\phi\in H_0^1(-1,1),\|\phi\|_{L^2(-1,1)}=1}\int_{-1}^1\left(|\phi'|^2-{\beta\over y-c}|\phi|^2\right)dy
\end{align}
is the principal eigenvalue of the regular Rayleigh-Kuo BVP
\begin{align}\label{lambda1betac-1-R-K}
-\phi''-{\beta\over y-c}\phi=\lambda\phi,\quad\phi(\pm1)=0,
\end{align}
where $\phi\in H^2\cap H_0^1(-1,1)$.
Then we have the following result.
\begin{lemma}\label{mono-beta}
$(1)$ $\lambda_1(\beta,\cdot)$ is decreasing on $(-\infty,-1]$ for fixed $\beta>0$.

$(2)$ $\lim_{\beta\to\infty}\lambda_1(\beta,c)=-\infty$ for fixed $c\leq -1$.
\end{lemma}
\begin{proof}(1)
For   $c_1<c_2\leq -1$, we have
\begin{align*}
\lambda_1(\beta,c_1)=\int_{-1}^{1} \left(|\phi_{\beta,c_1}'|^2-\frac{\beta}{y-c_1}|\phi_{\beta,c_1}|^2\right)dy
>\int_{-1}^{1}\left( |\phi_{\beta,c_1}'|^2-\frac{\beta}{y-c_2}|\phi_{\beta,c_1}|^2\right)dy
\geq \lambda_1(\beta,c_2),
\end{align*}
where $\phi_{\beta,c_1}\in H_0^1(-1,1)$ is  an eigenfunction of  $  \lambda_1(\beta,c_1)$ with $\|\phi_{\beta,c_1}\|_{L^2(-1,1)}=1$.

 (2)
Let $\varphi$ be a smooth function such that $ \text{supp}\,(\varphi)\subset({1\over4},{3\over4})$ and
 $\|\varphi\|_{L^2(-1,1)}=1$. Then (2) can be obtained by
 \begin{align*}
\lambda_{1}(\beta,c)
\leq\int_{1\over4}^{3\over4} \left(|\varphi'|^2-\frac{\beta}{y-c}|\varphi|^2\right)dy\to -\infty \;\;\text{as}\;\;\beta\to\infty.
\end{align*}
\end{proof}

We list some properties of $\lambda_1(\beta,c)$ for $\beta\geq0$ and $c\leq -1$.

 \begin{itemize}
 \item
For fixed $c<-1$,   $\lambda_1(\cdot,c)$ is smooth and  decreasing on $[0,\infty)$, see Lemma 11 in \cite{LYZ}.
\item
For $c=-1$, $\lambda_1(\cdot,c)$ is continuous and  decreasing on $[0,\infty)$, see Corollary
\ref{lambda-decreasing} and Lemma \ref{lambda-continuous}.
\item
For fixed $c\leq -1$,
$\lambda_1(0,c)={\pi^2\over4}$ and $\lim_{\beta\to\infty}\lambda_1(\beta,c)=-\infty$, see Lemma \ref{mono-beta} (2).
\item
For fixed $\beta>0$,
$\lambda_1(\beta,\cdot)$ is continuous and decreasing on $(-\infty,-1]$, see Lemma \ref{mono-beta} (1) and Lemma \ref{c less than -1 lim}.
\end{itemize}
Thus,
there exists a unique $\beta(c)>0$ such that $\lambda_1(\beta(c),c)=0$ for $c\leq-1$. Now, we  define the transitional value as
\begin{align}\label{def-beta*+}
\beta_*=\inf_{c\in(-\infty,-1]}\beta(c).
\end{align}
Then we claim that $\beta=\beta_*$ is  the unique point  such that $\lambda_1(\beta,-1)=0$.
\begin{lemma}\label{lem beta(-1)}
\beno
\beta_*=\beta(-1)\in(0,\infty).
\eeno
\end{lemma}
\begin{proof}
By Lemma \ref{mono-beta} (1), we have
\beno
\lambda_1(\beta(-1),c)>\lambda_1(\beta(-1),-1)=0=\lambda_1(\beta(c),c)
\eeno
for $c<-1$, which implies that
 $\beta(-1)<\beta(c)$ by Lemma 11 in \cite{LYZ}.
\end{proof}

Then we study the asymptotic  behavior of the principal eigenvalues for  regular approximations of
singular  Rayleigh-Kuo BVP with $c\to-1^-$.

\begin{lemma}\label{c less than -1 lim} $\lim_{c\to-1^-}\lambda_1(\beta,c)=\lambda_1(\beta,-1)$ for fixed $\beta>0$.
\end{lemma}
\begin{proof}
By  Lemma \ref{mono-beta},   $\lambda_1(\beta,-1)\leq \lim_{c\to-1^-}\lambda_1(\beta,c)$.
Conversely, let  $\phi_{\beta}$ be  an eigenfunction of $\lambda_1(\beta,-1)$ with $\|\phi_{\beta}\|_{L^2(-1,1)}=1$. For $c<-1$, we have
\begin{align}\label{potential-lim}
\lim_{c\to-1^{-}}\int_{-1}^1{1\over y-c}|\phi_\beta|^2dy =\int_{-1}^1{1\over y+1}|\phi_\beta|^2dy.
\end{align}
In fact,
\begin{align*}
&\left|\int_{-1}^1\left({1\over y-c}-{1\over y+1}\right)|\phi_{\beta}|^2dy\right|=\left|\int_{-1}^1\left({1+c\over( y-c)(y+1)}\right)|\phi_{\beta}|^2dy\right|\\
\leq&|1+c|\|\phi_{\beta}'\|_{L^2(-1,1)}^2 \int_{-1}^1{1\over y-c}dy=|1+c|\|\phi_{\beta}'\|_{L^2(-1,1)}^2 (\ln(1-c)-\ln(-1-c))\to0
\end{align*}
as $c\to-1^{-}$. Taking limit in
\begin{align*}
\lambda_1(\beta,c)\leq \int_{-1}^1\left(|\phi_{\beta}'|^2-{\beta\over y-c}|\phi_{\beta}|^2\right)dy,
\end{align*}
by \eqref{potential-lim} we have
\begin{align*}\lim_{c\to-1^-}\lambda_1(\beta,c)\leq \int_{-1}^1\left(|\phi_{\beta}'|^2-{\beta\over y+1}|\phi_{\beta}|^2\right)dy=\lambda_1(\beta,-1).\end{align*}
\end{proof}
Then we point out the difference of the principal eigenvalues of the  Rayleigh-Kuo BVP before and after $\beta_*$.

\begin{Proposition}\label{larger than less than beta*+}
$(1)$ Let $\beta>\beta_*$. Then there exists $c_\beta<-1$ such that $\lambda_1(\beta,c_\beta)=0$ and $\{\lambda_1(\beta,c)|c\in(c_\beta,-1)\}=(\lambda_1(\beta,-1),0)$.

$(2)$ Let $0<\beta\leq\beta_*$. Then  $\lambda_1(\beta,c)\geq0$ for $c\leq -1$.
\end{Proposition}

\begin{remark} We get more precise conclusions for the range of the principal eigenvalue.

$(1)$ Fix $\beta>0$. Since  $\lim_{c\to-\infty}\lambda_1(\beta,c)={\pi^2\over4}$ and $\lim_{c\to-1^-}\lambda_1(\beta,c)=\lambda_1(\beta,-1)$ by Lemma \ref{c less than -1 lim}, we infer from Lemma $\ref{mono-beta}$ $(1)$ that $\{\lambda_1(\beta,c)|c\in(-\infty,-1]\}=\left[\lambda_1(\beta,-1),{\pi^2\over4}\right)$.

$(2)$ For any $(\lambda_0,\beta_0)\in I:=\{(\lambda,\beta)| \beta\in(0,\infty),\lambda\in\left[\lambda_1(\beta,-1),{\pi^2\over4}\right)\}$, there exists a unique $c_0\in(-\infty, -1]$ such that
$\lambda_1(\beta_0,c_0)=\lambda_0$.  In particular, for any $(\lambda_0,\beta_0)\in \tilde I:=\left\{(\lambda,\beta)| \beta\in(\beta_*,\infty),\lambda\in\left(\lambda_1(\beta,-1),0\right)\right\}$, there exists a unique $c_0\in(c_{\beta_0}, -1)$ such that
$\lambda_1(\beta_0,c_0)$ $=\lambda_0$.
See Figure $2$.
\end{remark}

\begin{center}
 \begin{tikzpicture}[scale=0.58]
 \draw [->](-1, 0)--(10, 0)node[right]{$\beta$};
 \draw [->](0,-8)--(0,4) node[above]{$\lambda$};
 \draw  (0, 2).. controls (3, 0) and (5, -4)..(9.3,-8.75);
 \draw  (0, 2).. controls (3, 2) and (5, 2)..(8.7,2);
\path   (0.5, 2)  edge [-,dotted](0.6,1.57) [line width=0.8pt];
 \path   (1, 2)  edge [-,dotted](1.2,1.1) [line width=0.8pt];
  \path   (1.5, 2)  edge [-,dotted](1.8,0.45) [line width=0.8pt];
  \path (2, 2)  edge [-,dotted](2.3,-0.1) [line width=0.8pt];
    \path (2.5, 2)  edge [-,dotted](2.9,-0.75) [line width=0.8pt];
     \path (3, 2)  edge [-,dotted](3.5,-1.5) [line width=0.8pt];
      \path (3.5, 2)  edge [-,dotted](4,-2.1) [line width=0.8pt];
       \path (4, 2)  edge [-,dotted](4.5,-2.75) [line width=0.8pt];
       \path (4.5, 2)  edge [-,dotted](5,-3.5) [line width=0.8pt];
       \path (5, 2)  edge [-,dotted](5.6,-4.23) [line width=0.8pt];
       \path (5.5, 2)  edge [-,dotted](6.2,-5) [line width=0.8pt];
       \path (6, 2)  edge [-,dotted](6.7,-5.65) [line width=0.8pt];
       \path (6.5, 2)  edge [-,dotted](7.2,-6.4) [line width=0.8pt];
       \path (7, 2)  edge [-,dotted](7.7,-6.95) [line width=0.8pt];
       \path (7.5, 2)  edge [-,dotted](8.2,-7.56) [line width=0.8pt];
       \path (8, 2)  edge [-,dotted](8.8,-8.25) [line width=0.8pt];
       \path (8.5, 2)  edge [-,dotted](9.3,-8.75) [line width=0.8pt];
       \node (a) at (2.3,-0.5) {\tiny$\beta_*$};
      \node (a) at (-0.5,2) {\tiny${\pi^2\over4}$};
        \node (a) at (4.3,-4) {\tiny$\lambda_1(\cdot,-1)$};
 \end{tikzpicture}
\end{center}\vspace{-0.2cm}
 \begin{center}\vspace{-0.2cm}
   {\small {\bf Figure 2.} }
  \end{center}

\begin{proof} (1)
Since $\beta>\beta_*$, by Corollary \ref{lambda-decreasing} we have $\lambda_1(\beta,-1)<0$.  The conclusion then follows from Lemma \ref{mono-beta} (1), Lemma \ref{c less than -1 lim} and the fact that $\lim_{c\to-\infty}\lambda_1(\beta,c)={\pi^2\over4}$.

(2) Since $0<\beta\leq\beta_*$, by Corollary \ref{lambda-decreasing} we have $\lambda_1(\beta,-1)\geq0$. By Lemma \ref{mono-beta} (1),
 we have  $\lambda_1(\beta,c)\geq\lambda(\beta,-1)\geq0$ for $c\leq -1$.
\end{proof}

Similarly,  for $\beta\leq0$, we can define
\begin{align}\label{lambda-beta-neg}
\lambda_1(\beta,1)=\inf_{\phi\in H_0^1(-1,1),\|\phi\|_{L^2(-1,1)}=1}\int_{-1}^1\left(|\phi'|^2-{\beta\over y-1}|\phi|^2\right)dy.
\end{align}
For $\beta\leq0$ and $c>1$,
$
\lambda_1(\beta,c)$ defined in \eqref{lambda1betac-1}
is the principal eigenvalue of the regular Rayleigh-Kuo BVP \eqref{lambda1betac-1-R-K}.
We list some properties of $\lambda_1(\beta,c)$ for $\beta\leq 0$ and $c\geq 1$, whose proof is similar as above, and thus, omitted here.
\begin{lemma}\label{eigenfunction-neg}
$(1)$ Let $\beta\leq 0$. Then the infimum in \eqref{lambda-beta-neg} is attained by some function $\phi_\beta\in H_0^1(-1,1),$ $\|\phi_\beta\|_{L^2(-1,1)}=1$.
Moreover,
 there exists $m_\beta>-\infty$ such that
$m_\beta\leq\lambda_1(\beta,1)\leq {\pi^2\over4}$.

$(2)$ For fixed $c>1$, $\lambda_1(\cdot,c)$ is smooth and increasing on $\beta\in(-\infty,0]$.

$(3)$ For $c=1$, $\lambda_1(\cdot,c)$ is continuous and increasing on $\beta\in(-\infty,0]$.

$(4)$ For fixed $c\geq1$, $\lambda_1(0,c)={\pi^2\over4}$ and $\lim_{\beta\to-\infty}\lambda_1(\beta,c)=-\infty.$

$(5)$ For fixed $\beta<0$, $\lim_{c\to1^+}\lambda_1(\beta,c)=\lambda_1(\beta,1)$, and $\lambda_1(\beta,\cdot)$ is continuous and increasing on $c\in[1,\infty)$.
\end{lemma}

By Lemma \ref{eigenfunction-neg} (2)-(4),  there exists a unique $\beta(c)<0$ such that $\lambda_1(\beta(c),c)=0$ for $c\geq 1$.  Define
\begin{align*}
\beta_{*}^-=\sup_{c\in[1,+\infty)}\beta(c).
\end{align*}
Then $\beta_{*}^-$ has the following analog  properties as $\beta_*$, whose proof is similar to Lemma \ref{lem beta(-1)} and Proposition \ref{larger than less than beta*+}.
\begin{Proposition}\label{larger than less than beta*-}
$(1)$
$\beta_{*}^-=\beta(1)\in (-\infty,0)$.

$(2)$ Let $\beta<\beta_{*}^-$. Then there exists $c_\beta>1$ such that $\lambda_1(\beta,c_\beta)=0$ and $\{\lambda_1(\beta,c)|c\in(1,c_\beta)\}=(\lambda_1(\beta,1),0)$.

$(3)$ Let $\beta_*^-\leq\beta<0$. Then  $\lambda_1(\beta,c)\geq0$ for $c\geq 1$.
\end{Proposition}
%

Symmetry of the principal eigenvalue  $\lambda_1(\beta,c)$ could be  deduced by that  of Couette flow.

\begin{lemma}\label{symmetry}
 For $\beta\geq 0$ and $c\leq -1$, we have
\beno
\lambda_1(\beta,c)=\lambda_1(-\beta,-c).
\eeno
\end{lemma}
\begin{proof}
Let $\phi_0\in H_0^1(-1,1)$ be an eigenfunction of $\lambda_1(\beta,c)$ with $\|\phi_0\|_{L^2(-1,1)}=1$, and define $\tilde{\phi}_0(y)=\phi_0(-y)$ for $y\in [-1,1]$. Then
\beno
\lambda_1(\beta,c)
&=&\int_{-1}^1\left(|\phi_0'(y)|^2-{\beta\over y-c}|\phi_0(y)|^2\right)dy\\
&=&\int_{-1}^1\left(|\phi_0'(-z)|^2-{\beta\over -z-c}|\phi_0(-z)|^2\right)dz\\
&=&\int_{-1}^1\left(|\tilde\phi_0'(z)|^2-{-\beta\over z-(-c)}|\tilde\phi_0(z)|^2\right)dz\\
&\geq&\inf_{\phi\in H_0^1(-1,1),\|\phi\|_{L^2(-1,1)}=1}\int_{-1}^1\left(|\phi'|^2-{-\beta\over y-(-c)}|\phi|^2\right)dy\\
&=&\lambda_1(-\beta,-c).
\eeno
Similarly, we have $\lambda_1(-\beta,-c)\geq \lambda_1(\beta,c)$.
\end{proof}

\begin{lemma}
$\beta_*^-=-\beta_{*}$.
\end{lemma}
\begin{proof}
Note that $\lambda_1(\beta(1),1)=0=\lambda_1(\beta(-1),-1)$. By  Corollary \ref{lambda-decreasing} and Lemma \ref{eigenfunction-neg} (3), we infer from Lemma \ref{symmetry} that $-\beta(1)=\beta(-1)$. Thus, $\beta_*^-=\beta(1)=-\beta(-1)=-\beta_*$ by Proposition \ref{larger than less than beta*-} (1) and Lemma \ref{lem beta(-1)}.
\end{proof}

Combining Propositions \ref{larger than less than beta*+} and \ref{larger than less than beta*-} (2)-(3), we get the following result.
\begin{Proposition}\label{larger than less than beta*}
$(1)$ Let $|\beta|>\beta_*$. Then there exists $|c_\beta|>1$ such that $\beta c_\beta<0$, $\lambda_1(\beta,c_\beta)=0$ and $\{\lambda_1(\beta,c)||c|\in(1,|c_\beta|),\beta c<0\}=(\lambda_1(|\beta|,-1),0)$.

$(2)$ Let $0<|\beta|\leq\beta_*$. Then  $\lambda_1(\beta,c)\geq0$ for $|c|\geq 1$.
\end{Proposition}
\begin{proof}
 It suffices to prove for the case $\beta c > 0$ in (2).   Since $-{\beta \over y- c} > 0$ for $y \in(-1, 1)$,
the conclusion then follows from \eqref{lambda1betac-1}.
\end{proof}

\section{Contrasting dynamics near Couette flow  for $0<|\beta|\leq\beta_*$ and $|\beta|>\beta_*$ }

In this section, we prove  Theorem \ref{thm-non-shear-nodelta2}, which gives contrasting dynamics near Couette flow between
$(\alpha,\beta)\in O$ and $(\alpha,\beta)\in I_{\pm}$, see Figure 1.

\begin{proof}[Proof of Theorem \ref{thm-non-shear-nodelta2}.] First, we give the proof of  (1) and (2ii) for $\beta>0$, and the proof for  $\beta<0$ is similar.
Suppose otherwise, there exist $\{\varepsilon_n\}_{n=1}^\infty$,  {$c_n\in\mathbb{R}$ and traveling wave solutions $(u_n(x-c_n t,y), v_n(x-c_n t,y))$} to the $\beta$-plane equation (\ref{eq})-(\ref{bc})  such that $\varepsilon_n\to0$, $(u_n,v_n)$ is $T$-periodic in $x$,
$
\|(u_n,v_n)-(y,0)\|_{H^{s}{(D_T)}}\leq \varepsilon_n
$
and $\|v_n\|_{L^2(D_T)}\neq0$. Then $(u_n,v_n)$ solves \eqref{un-vn-eq}.
$s\geq5$ implies
(\ref{un-y-C3-nd})-(\ref{un-xy-bound-nd}) holds
for $n$ sufficiently large.  Let $\tilde v_n=v_n/\|v_n\|_{L^2(D_T)}$.   By
\eqref{un-vn-eq} and the fact that $\partial_x\omega_n=\Delta v_n$, we have
\ben\label{tilde-vn-eq-nd}
\Delta \tilde v_n+(\pa_y\omega_n+\beta){\tilde v_n\over u_n{-c_n}}=0.
\een
Then up to a subsequence, $c_n\to c_0\in\mathbb{R}\cup\{\pm\infty\}$, and we divide the proof into four cases in terms of $c_0$.\\
{\bf Case 1.} $|c_0|=\infty$.

Up to a subsequence, we can assume that $|c_n|>M_n+2$ with $0<M_n\to \infty$. By \eqref{un-y-C3-nd}-\eqref{omegan-C2-nd}, we have $\|u_n\|_{C^0(D_T)}<\frac{3}{2}$ and $\|\pa_y\omega_n\|_{C^0(D_T)}<1$ for $n$ large enough, and thus,
\beno
\|\Delta \tilde{v}_n\|_{L^2(D_T)}\leq \frac{|\beta|+1}{M_n}\|\tilde{v}_n\|_{L^2(D_T)}\to 0.
\eeno
However, by  periodic
boundary condition in $x$ and Dirichlet boundary condition in $y$ of $\tilde v_n$, we have
\beno
1\leq \| \tilde v_n\|_{H^2(D_T)}\leq C\|\Delta \tilde v_n\|_{L^2(D_T)},
\eeno
which is a contradiction.\\
{\bf Case 2.}  $|c_0|>1$.

By Lemma 2.5 in \cite{LWZZ} and $|c_0|>1$, $c_0 \in \cup_{k\geq 1}(\sigma_d(\mathcal{R}_{k\alpha,\beta}) \cap \mathbb{R})$. Thus, $\lambda_1(\beta,c_0)=-(k_0\alpha)^2<0$ for some $k_0\geq 1$, where $\lambda_1(\beta,c_0)$ is the  principal eigenvalue of the regular Rayleigh-Kuo BVP \eqref{lambda1betac-1-R-K} with $c=c_0$.

For (1), by Proposition \ref{larger than less than beta*} (2), we have $\lambda_1(\beta,c_0)\geq0$ for $0<\beta\leq\beta_*$, which is a contradiction.

For (2ii), since $k_0\alpha={2k_0\pi\over T}>\alpha_\beta=\sqrt{-\lambda_1(\beta,-1)}$, we have $\lambda_1(\beta,c_0)=-(k_0\alpha)^2<\lambda_1(\beta,-1)$.
If $c_0>1$, then $\lambda_1(\beta,c_0)\geq0$ by \eqref{lambda1betac-1}, which contradicts $\lambda_1(\beta,c_0)<0$. If
$c_0<-1$, then $\lambda_1(\beta,c_0)>\lambda_1(\beta,-1)$ by Lemma \ref{mono-beta} (1), which contradicts $\lambda_1(\beta,c_0)<\lambda_1(\beta,-1)$.
\\
{\bf Case 3.}  $c_0=\pm 1$.

We only give the proof  for $c_0=-1$ and the other is similar.
By (\ref{un-xy-bound-nd}), $u_n(x,\cdot)$ is increasing on $[-1,1]$  for $n$ large enough and $x\in[0,T]$. Divide $[0,T]$ into three subsets
\beno
P_n=\{x| c_n\leq u_n(x,-1)\},\quad Q_n=\{x| c_n\geq u_n(x,1)\},\quad
S_n=\{x|u_n(x,-1)<c_n<u_n(x,1)\}.
\eeno
For $c_0=-1$, we have $Q_n=\emptyset$ for $n$ sufficiently large. For fixed $n$ and $x\in[0,T]$,
let $d_n(x)$ be the point closest to $c_n$ in $\text{Ran}\,(u_n(x,\cdot))$,   i.e.
\beno
d_n(x)=\left\{\begin{array}{lr}
u_n(x,-1), & x\in P_n,\\
u_n(x,1),& x\in Q_n,\\
c_n, &x\in S_n.
\end{array}
 \right.
\eeno
Then there exists a unique $z_n(x)\in [-1,1]$ such that $u_n(x,z_n(x))=d_n(x)$ for $x\in [0,T]$. For $x\in S_n$, we have $u_n(x,z_n(x))=d_n(x)=c_n$, and   it follows from \eqref{un-vn-eq} and \eqref{omegan-C2-nd} that $\tilde{v}_n(x,z_n(x))=0$. For $x\in P_n$, $z_n(x)=-1$ and by the non-permeable boundary condition, we also have $\tilde{v}_n(x,z_n(x))=0$. Thus,  $\tilde{v}_n(x,z_n(x))=0$ for $x\in [0,T]$. Moreover,  $|u_n(x,y)-c_n|\geq |u_n(x,y)-d_n(x)|$ for $(x,y)\in D_T$.

Now we can prove the uniform $H^2$ bound for $\tilde v_n$, $n\geq 1$. By  \eqref{un-xy-bound-nd}, $\left|{y-z_n(x)\over u_n(x,y){-d_n(x)}}\right|=\left|{y-z_n(x)\over u_n(x,y)-u_n(x,z_n(x))}\right|\leq 2$ for $(x,y)\in D_T$ and $n$ large enough. Thus, $\left\|{y-z_n(x)\over u_n{-d_n(x)}}\right\|_{L^\infty(D_T)}\leq C$. This, along with  \eqref{tilde-vn-eq-nd}, \eqref{un-omegan-C bound-nd} and Lemma  \ref{Hardy type inequality2}, gives
\ben\nonumber
\|\Delta \tilde v_n\|_{L^2(D_T)}&\leq& \left\|(\pa_y\omega_n+\beta){\tilde v_n\over u_n{-c_n}}\right\|_{L^2(D_T)}\leq C \left\|{\tilde v_n\over u_n{-c_n}}\right\|_{L^2(D_T)}\leq C \left\|{\tilde v_n\over u_n{-d_n(x)}}\right\|_{L^2(D_T)}
\\\nonumber
&\leq& C \left\|{\tilde v_n\over y-z_n(x)}\right\|_{L^2(D_T)}\left\|{y-z_n(x)\over u_n{-d_n(x)}}\right\|_{L^\infty(D_T)}\\\nonumber
&\leq& C\left\|{\tilde v_n\over y-z_n(x)}\right\|_{L^2(D_T)}\leq C\left(\int_0^T\left\|{\tilde v_n(x,\cdot)}\right\|_{H^1(-1,1)}^2dx\right)^{1\over2}\\\label{tilde vn-over-un-L2-nd}
&\leq&C\left\|{\tilde v_n}\right\|_{H^1(D_T)}\leq C\left\|{\tilde v_n}\right\|_{H^2(D_T)}^{1/2}\left\|{\tilde v_n}\right\|_{L^2(D_T)}^{1/2}=C\left\|{\tilde v_n}\right\|_{H^2(D_T)}^{1/2}.
\een
Thus, by  periodic
boundary condition in $x$ and Dirichlet boundary condition in $y$ of $\tilde v_n$, we have
\ben\label{tilde vn-H2 bound-nd}
\| \tilde v_n\|_{H^2(D_T)}\leq C\|\Delta \tilde v_n\|_{L^2(D_T)}\leq  C\left\|{\tilde v_n}\right\|_{H^2(D_T)}^{1/2}\Longrightarrow\| \tilde v_n\|_{H^2(D_T)}\leq C.
\een
Then there exists $\tilde{v}_0\in H^2(D_T)$ such that up to a subsequence, $\tilde{v}_n\rightharpoonup \tilde{v}_0$ in $H^2(D_T)$, $\tilde{v}_n\to \tilde{v}_0$ in $H^1(D_T)$ and $C^0(D_T)$.
By (\ref{tilde-vn-eq-nd}),  we have
\ben\label{virtical-vel-weak-n}
\int_{D_T} \left(-\nabla \tilde{v}_n\cdot \nabla\phi+(\pa_y \omega_n+\beta)\frac{\tilde{v}_n\phi}{u_n-c_n}\right)dxdy=0
\een
for any $ \phi\in H^1(D_T)$ satisfying periodic
boundary condition in $x$ and Dirichlet boundary condition in $y$.
We prove that
\ben\label{tildev_0-eq}
\int_{D_T} \left(-\nabla \tilde{v}_0\cdot \nabla\phi+\beta\frac{\tilde{v}_0\phi}{y+1}\right)dxdy=0.
\een
By (\ref{omegan-C2-nd}), \eqref{tilde vn-over-un-L2-nd} and \eqref{tilde vn-H2 bound-nd}, we have
\begin{align}\nonumber
\bigg|\int_{D_T} \pa_y \omega_n \frac{\tilde{v}_n\phi}{u_n-c_n}dxdy \bigg|&\leq \|\pa_y \omega_n\|_{C^0(D_T)} \bigg\|\frac{\tilde{v}_n}{u_n-c_n}\bigg\| _{L^2(D_T)}\|\phi\|_{L^2(D_T)}\\\label{sec-term-n-inf}
&\leq C\varepsilon_n\|\tilde v_n\|_{H^1(D_T)}\|\phi\|_{L^2(D_T)}\to 0.
\end{align}
Moreover,
\begin{align}\nonumber
&\bigg|\beta\int_{D_T} \bigg(\frac{\tilde{v}_n}{u_n-c_n}-\frac{\tilde{v}_0}{y+1}\bigg)\phi dxdy \bigg|\\ \label{I+II-nd}
\leq &\bigg|\beta\int_{D_T} \bigg(\frac{\tilde{v}_n}{u_n-c_n}-\frac{\tilde{v}_n}{y+1}\bigg)\phi dxdy \bigg|+\bigg|\beta\int_{D_T} \bigg(\frac{\tilde{v}_n}{y+1}-\frac{\tilde{v}_0}{y+1}\bigg)\phi dxdy \bigg|=I+II.
\end{align}
For $II$, since $\phi(x,\pm1)=0$ for $x\in[0,T]$, by Lemma \ref{Hardy type inequality2} we have
\ben
II\leq |\beta| \bigg|\int_{D_T}(\tilde{v}_n-\tilde{v}_0) \frac{\phi}{y+1}dxdy\bigg|\leq C\|\tilde{v}_n-\tilde{v}_0\|_{C^0(D_T)} \|\phi\|_{H^1(D_T)}\to 0.
\een
For $I$, we decompose it into two parts:
\ben\label{I bound}
I&\leq&|\beta|\bigg|\int_{S_n}\int_{-1}^1 \bigg(\frac{\tilde{v}_n}{u_n-c_n}(y+1)-\tilde{v}_n\bigg)\frac{\phi}{y+1} dydx \bigg|\\  \nonumber
&&+|\beta|\bigg|\int_{P_n}\int_{-1}^1 (y+1)\frac{y+1-(u_n-c_n)}{u_n-c_n}\frac{\tilde{v}_n}{y+1}\frac{\phi}{y+1} dydx \bigg|=I_1+I_2.
\een
For $x\in S_n$, we have $u_n(x,z_n(x))=c_n$, and thus,
\beno
\frac{\tilde{v}_n}{u_n-c_n}(y+1)
=\frac{\tilde{v}_n}{u_n-c_n}(y-z_n(x)+z_n(x)+1)=\frac{\tilde{v}_n}{\int_0^1 \pa_yu_n(x,\tau_n)dt}+\frac{\tilde{v}_n}{u_n-c_n}(z_n(x)+1),
\eeno
where $\tau_n=ty+(1-t)z_n(x)$.
By \eqref{un-y-C3-nd}, we have
\begin{align*}
\left\|\frac{1}{\int_0^1 \pa_y u_n(x,\tau_n)dt}-1\right\|_{L^\infty(D_T)}\leq C\|\partial_yu_n-1\|_{C^0(D_T)}\leq C\varepsilon_n\to0,
\end{align*}
and
\begin{align*}
|z_n(x)+1|\leq&|z_n(x)-c_n|+|c_n+1|\leq |z_n(x)-u_n(x,{z_n(x)})|+|c_n+1| \\ \nonumber
\leq&\|u_n-y\|_{C^0(D_T)}+|c_n+1|\leq C\varepsilon_n+|c_n+1|\to 0
\end{align*}
for $x\in[0,T]$.
So by \eqref{tilde vn-over-un-L2-nd}-\eqref{tilde vn-H2 bound-nd}, Lemma \ref{Hardy type inequality2} and the fact that $\phi(x,\pm1)=0$ for $x\in[0,T]$,
\begin{align}
I_1\leq&|\beta|\bigg|\int_{S_n}\int_{-1}^1 \bigg(\frac{\tilde{v}_n}{\int_0^1 \pa_yu_n(x,\tau_n)dt}-\tilde{v}_n\bigg)\frac{\phi}{y+1} dydx\bigg|  \nonumber
\\&+|\beta|\bigg|\int_{S_n}\int_{-1}^1 \frac{\tilde{v}_n}{u_n-c_n}(z_n(x)+1)\frac{\phi}{y+1} dydx \bigg|  \nonumber\\
\leq&C\varepsilon_n\bigg|\int_{D_T} |\tilde{v}_n|\frac{|\phi|}{y+1} dydx\bigg|+C(C\varepsilon_n+|c_n+1|)  \bigg|\int_{D_T} \frac{|\tilde{v}_n|}{|u_n-c_n|}\frac{|\phi|}{y+1} dydx \bigg|\nonumber\\
\leq& C\varepsilon_n\|\tilde{v}_n\|_{L^2(D_T)}\|\phi\|_{H^1(D_T)}+C(C\varepsilon_n+|c_n+1|)\|\tilde{v}_n\|_{H^1(D_T)}\|\phi\|_{H^1(D_T)}\to 0.
\end{align}
By \eqref{un-xy-bound-nd},  we have $|u_n(x,y)-c_n|\geq |u_n(x,y)-u_n(x,-1)|\geq {1\over2}(y+1)$ for $x\in P_n$ and $y\in(-1,1)$, and thus,
\begin{align*}
\bigg|(y+1)\frac{y+1-(u_n(x,y)-c_n)}{u_n(x,y)-c_n}\bigg|\leq &2|y+1-(u_n(x,y)-c_n)|\\
\leq& 2\|u_n-y\|_{C^0(D_T)}+|c_n+1|\to0.
\end{align*}
Then by  the fact that $\tilde v_n(x,\pm1)=0$, $\phi(x,\pm1)=0$ for $x\in[0,T]$,
Lemma \ref{Hardy type inequality2} and \eqref{tilde vn-H2 bound-nd}, we have
\ben\label{I2-bound-nd}
I_2\leq C(2\|u_n-y\|_{C^0(D_T)}+|c_n+1|)\|\tilde{v}_n\|_{H^1(D_T)}\|\phi\|_{H^1(D_T)}\to 0.
\een
Taking (\ref{I+II-nd})-(\ref{I2-bound-nd}) into account, we have
\ben\label{last-term-inf}
\bigg|\beta\int_{D_T} \bigg(\frac{\tilde{v}_n}{u_n-c_n}-\frac{\tilde{v}_0}{y+1}\bigg)\phi dxdy \bigg|\to 0.
\een
By  \eqref{sec-term-n-inf}, \eqref{last-term-inf} and the fact that $\tilde{v}_n\to \tilde{v}_0$ in $H^1(D_T)$, we  obtain \eqref{tildev_0-eq} by sending $n\to \infty$ in \eqref{virtical-vel-weak-n}.

Since  $\tilde v_n=-\partial_x\psi_n/\|v_n\|_{L^2(D_T)}$ is $T$-periodic in $x$, we have
 $\tilde v_0(x,y)=\sum_{k\neq0}\widehat{\tilde v}_{{0},k}(y)e^{ik\alpha x}$. It follows from $\|\tilde v_0\|_{L^2(D_T)}=1$ that  there exists $k_0\neq 0$ such that $\widehat{\tilde v}_{{0},k_0}\not\equiv 0$, where $\alpha={2\pi\over T}$.
 Notice that for $j=0,1,2$,
\beno
\|\widehat{\tilde v}_{0,k_0}^{(j)}\|^2_{L^2(-1,1)}=\int_{-1}^1 \bigg|{1\over T}\int_0^T  \pa_y^j \tilde v_{0}(x,y) e^{-ik_0\alpha x} dx\bigg|^2dy
\leq C\int_{-1}^1\int_0^T|\pa_y^j \tilde v_{0}(x,y)|^2dxdy <\infty
\eeno
and $\widehat{\tilde v}_{{0},k_0}(\pm 1)={1\over T}\int_0^T\tilde v_{0}(x,\pm 1) e^{-ik_0\alpha x} dx= 0$, Thus, we can take the test function $\phi(x,y)=\widehat{\tilde v}_{{0},k_0}(y) e^{ik_0\alpha x}$ in (\ref{tildev_0-eq}), and use the fact that $\widehat{\tilde v}_{{0},k_0}=\overline{\widehat{\tilde v}_{{0},-k_0}}$ and $\int_0^T e^{i(k-k_0)\alpha x} dx=0$ for $k\neq k_0$ to obtain
 \beno
\int_{-1}^{1}\left( -|\pa_y\widehat{\tilde v}_{{0},k_0}(y)|^2-(k_0\alpha)^2|\widehat{\tilde v}_{{0},k_0}(y)|^2+\frac{\beta}{y+1}|\widehat{\tilde v}_{{0},k_0}(y)|^2\right) dy=0.
\eeno
Thus,
\ben\label{til-v-k0}
\int_{-1}^{1} \left(|\pa_y\widehat{\tilde v}_{{0},k_0}(y)|^2-\frac{\beta}{y+1}|\widehat{\tilde v}_{{0},k_0}(y)|^2 \right) dy=-\int_{-1}^{1}(k_0\alpha)^2|\widehat{\tilde v}_{{0},k_0}(y)|^2 dy<0.
\een

For (1),
 since $0<\beta<\beta_*=\beta(-1)$, by Corollary \ref{lambda-decreasing} we have $\lambda_1(\beta,-1)>\lambda_1(\beta(-1),-1)=0$, and thus,
\beno
\inf_{\phi\in H_0^1(-1,1),\|\phi\|_{L^2(-1,1)}=1} \int_{-1}^1\left(|\phi'|^2-\frac{\beta}{y+1}|\phi|^2\right)dy>0,
\eeno
 which is a contradiction.

 For (2ii), by \eqref{til-v-k0} we have $\lambda_1(\beta,-1)\leq -(k_0\alpha)^2 <-\alpha_\beta^2=\lambda_1(\beta,-1)$, which is a contradiction.
 \\
{\bf Case 4.}  $|c_0|<1$.

Let $\delta_0=\frac{1-|c_0|}{2}>0$. Then $[c_0-\delta_0,c_0+\delta_0]\subset [-1+\delta_0,1-\delta_0]$.   By Theorem \ref{thm-non-shear},  there exists $\varepsilon_{\delta_0}>0$ such that
any traveling wave solution $(u(x-ct,y),v(x-ct,y))$ to the $\beta$-plane equation with $c\in[-1+\delta_0,1-\delta_0]$, $x$-period $T$ and satisfying that
$
\|(u,v)-(y,0)\|_{H^{s}{(D_T)}}<\varepsilon_{\delta_0},
$
must have $v(x,y)\equiv 0$.
Since $c_n\to c_0$ and $\varepsilon_n\to0$, we have  $c_n\in [c_0-\delta_0,c_0+\delta_0]$
and
$\|(u_n,v_n)-(y,0)\|_{H^{s}{(D_T)}}\leq\varepsilon_n<\varepsilon_{\delta_0}$ for $n$ large enough. Then  $v_n\equiv 0$ on $D_T$ for $n$ large enough, which contradicts  $\|v_n\|_{L^2(D_T)}\neq0$.

Finally, we prove (2i). We only prove it for $\beta>\beta_*$, and the proof for $\beta<-\beta_*$ is similar.
We divide the discussion in terms of $\alpha$.
\\{\bf Case I.} $\alpha={2\pi\over T}=\sqrt{-\lambda_1(\beta,-1)}$.

We first modify Couette flow to the nearby shear flow $(ay,0)$ with $a\in(0,1)$, and then  construct traveling waves  by bifurcation at $(ay,0)$.

For the shear flow $(ay,0)$, let $\tilde \lambda_1(\beta,c)$ be the principal
eigenvalue of the Rayleigh-Kuo BVP
\begin{align}\label{lambda1betac-1-R-K-ay}
-\phi''-{\beta\over ay-c}\phi=\lambda\phi,\quad\phi(\pm1)=0,
\end{align}
where $\phi\in H_0^1\cap H^2(-1,1)$ and $c\leq -1<-a$. Recall that $ \lambda_1(\beta,c)$ is  the principal
eigenvalue of the Rayleigh-Kuo BVP \eqref{lambda1betac-1-R-K} for Couette  flow. Since
$
-\phi''-{\beta\over ay-c}\phi=-\phi''-{\beta/a\over y-c/a}\phi,
$
we have
\begin{align}\label{ay-prin eigen-y}
\tilde \lambda_1(\beta,c)= \lambda_1\left({\beta\over a},{c\over a}\right)
\end{align}
for $c\leq -a$. Since $a\in(0,1)$, we infer from Corollary \ref{lambda-decreasing} that $0>-\alpha^2=\lambda_1(\beta,-1)>\lambda_1\left({\beta\over a},-1\right)$. It follows from Proposition \ref{larger than less than beta*+} (1) and Lemma \ref{mono-beta} (1) that there exists a unique $c_a\in(-\infty,-1)$ such that
$-\alpha^2=\lambda_1\left({\beta\over a},c_a\right)$. This, along with Lemma \ref{c less than -1 lim}, implies that $c_a\to-1$ as $a\to1^-$. By \eqref{ay-prin eigen-y}, we have $-\alpha^2=\lambda_1\left({\beta\over a},c_a\right)=\tilde \lambda_1\left({\beta},ac_a\right)$. Note that  $ac_a<-a$, and thus, $ac_a\notin \text{Ran} \, (ay)= [-a,a]$. This implies that  $ac_a\in
\sigma_d(\mathcal{R}_{\alpha,\beta})$ with $u(y)=ay$ in \eqref{linearized Euler operator}.

Let $s\geq3$. For any $\varepsilon>0$, there exists $a=a_\varepsilon\in(0,1)$ such that
\begin{align}\label{shear flow perturbation u+tau-varepsilon-u-1-case 2a-r}
\|(a_\varepsilon y,0)-(y,0)\|_{H^s(D_T)}\leq C|a_\varepsilon-1|< {\varepsilon\over2}\quad \text{and} \quad |a_\varepsilon c_{a_\varepsilon}+1|<{\varepsilon\over2}.
\end{align}
By Corollary 2.4 in \cite{LWZZ} and $a_\varepsilon c_{a_\varepsilon}\in
\sigma_d(\mathcal{R}_{\alpha,\beta})$ with $u(y)=a_\varepsilon y$ in \eqref{linearized Euler operator},  there exists a traveling wave solution $(u_\varepsilon(x-c_\varepsilon t,y),v_\varepsilon(x-c_\varepsilon t,y))$ to
\eqref{eq}-\eqref{bc} which has
 period $T={2\pi}/{\alpha}$ in $x$,
\begin{align}\label{traveling wave to u+tau-varepsilon-u-1-case 2a-r}
 \|(u_\varepsilon,v_\varepsilon)-(a_\varepsilon y,0)\|_{H^s(D_T)} \leq{\varepsilon\over2} \quad \text{and} \quad | c_{\varepsilon}-a_\varepsilon c_{a_\varepsilon}|\leq{\varepsilon\over2},
\end{align}
$u_{\varepsilon}\left(  x,y\right)-c_\varepsilon  \neq0$ and $\|v_\varepsilon\|_{L^2\left(  D_T\right)}\neq0$. Then by \eqref{shear flow perturbation u+tau-varepsilon-u-1-case 2a-r}-\eqref{traveling wave to u+tau-varepsilon-u-1-case 2a-r}, we have
$
 \|(u_\varepsilon,v_\varepsilon)-(y,0)\|_{H^s(D_T)}$ $<{\varepsilon}$ {and}  $|c_{\varepsilon}+1|<{\varepsilon}.
$
\\
{\bf Case II.} $\alpha={2\pi\over T}\in(0,\sqrt{-\lambda_1(\beta,-1)})$.

By Proposition \ref{larger than less than beta*+} (1),  there exists $c_{\alpha,\beta}<-1$ such that $\lambda_1(\beta,c_{\alpha,\beta})=-\alpha^2$.
Thus, $c_{\alpha,\beta}\in\sigma_d(\mathcal{R}_{\alpha,\beta})$. Let $s\geq3$.
Then the conclusion  follows from Corollary 2.4 in \cite{LWZZ}.
 \end{proof}


\section{ Existence of non-shear steady state in $H^{<{5\over2}}$ }

In this section, we prove the existence of non-shear steady states near Couette flow in (velocity) $H^{<{5\over2}}$ space.
The method is to construct non-shear steady states by bifurcation at modified shear flows near Couette. First, we give a bifurcation lemma for $\beta\neq0$. Note that the bifurcation lemma for $\beta=0$ in \cite{LZ} can not be applied or extended  to the case $\beta\neq0$, since the extension cause a singularity at the middle point $y=0$ for the Rayleigh-Kuo BVP, which is difficult to deal with. In the following bifurcation lemma, we require that $u''-\beta=0$ near $0$, and the price is that a similar  construction of modified shear flow in \cite{LZ} cannot be applied here, since the Gauss error function is an odd function. We introduce the new modified shear flow latter.  Another difference is to deal with the degeneracy of the Rayleigh-Kuo BVP.

\begin{lemma}
\label{lem-bifurcation} Consider a shear flow $u\in C^{4}\left( [ -1,1]\right)
$,  $u(0)=0$, $u'>0$ on $[-1,1]$,
 $\beta\in\mathbb{R}$ and $-u'(y)+\beta y\equiv K$ on $[-\delta_0,\delta_0]$ for some $\delta_0\in(0,1)$ and $K\in\mathbb{R}$. Define the Rayleigh-Kuo  operator
 \begin{align*}
 \mathcal{L}\phi:=-\phi''+{u''-\beta\over u}\phi,\;H^2\cap H_0^1(-1,1)\to L^2(-1,1).
 \end{align*}
If the principal eigenvalue of  $\mathcal{L}$ satisfies that  $\mu_1=-\alpha_0^2<0$, then there exists $\varepsilon_0>0$ such that for each $0<\varepsilon<\varepsilon_0$, there exists a steady solution $(u_\varepsilon(x,y),v_\varepsilon(x,y))$ to the $\beta$-plane equation with minimal period $T_\varepsilon$ in $x$ such that
\begin{align*}
\|(u_\varepsilon,v_\varepsilon)-(u,0)\|_{H^{3}{(D_{T_\varepsilon})}}<\varepsilon,
 \end{align*}
 and $v_{\varepsilon}\not\equiv0$, where $D_{T_\varepsilon}=[0,T_\varepsilon]\times[-1,1]$.
 Furthermore,  $T_\varepsilon\to {2\pi\over \alpha_0}$ as $\varepsilon\to0^+$.
\end{lemma}

\begin{proof}
$(u,v) $ is a solution of \eqref{eq}-\eqref{bc} if
and only if
$\partial_x\omega\partial_y\psi-\partial_x\psi(\partial_y\omega+\beta)=0$
and $\psi$ takes constant values on $\left\{  y=\pm1\right\}$, where
$\omega=\partial_xv-\partial_yu$ and $(u,v)=(\partial_y\psi,-\partial_x\psi)$. Let $\mu_n$ be the $n$-th eigenvalue of $\mathcal{L}$ for $n\geq1$.
The proof is divided into two cases.\\
{\bf Case 1.} $\mu_n\neq0$ for $n\geq2$.

 Let $\psi_{0}$  be a stream function associated
with the shear flow $\left(  u,0\right)  $, i.e., $\psi_{0}^{\prime
}\left(  y\right)  =u\left(  y\right)$. Since $u<0$ on $[-1,0)$ and $u>0$ on $(0,1]$, $\psi_{0}$
is decreasing on $[-1,0]$ and increasing on $[0,1]$.
 Let $\psi_{0+}^{-1}$ and $\psi_{0-}^{-1}$ be the inverse maps of  $\psi_{0}|_{[0,1]}$ and $\psi_{0}|_{[-1,0]}$, respectively. Then we can define two functions $\tilde f_{-}=(-\psi_0''
+\beta (\cdot))\circ\psi_{0-}^{-1}$ on
$[\psi_0(0),\psi_{0}(-1)]$ and  $\tilde f_{+}=(-\psi_0''
+\beta (\cdot))\circ\psi_{0+}^{-1}$ on
$[\psi_0(0),\psi_{0}(1)]$. Then $\tilde f_{-}\equiv K$ on
$[\psi_0(0),\psi_{0}(-\delta_0)]$ and  $\tilde f_{+}\equiv K$ on
$[\psi_0(0),\psi_{0}(\delta_0)]$. Clearly, $\tilde f_{-}\in C^3([\psi_0(-\delta_0/2),\psi_{0}(-1)])$ and $\tilde f_{+}\in C^3([\psi_0(\delta_0/2),\psi_{0}(1)])$. Thus, $\tilde f_{-}\in C^3([\psi_0(0),\psi_{0}(-1)])$   and  $\tilde f_{+}\in C^3([\psi_0(0),\psi_{0}(1)])$.
Then we extend $\tilde f_{+}, \tilde f_{-}$ to $f_+, f_-$ such
that $f_+, f_-\in C_{0}^{3}\left(  \mathbb{R}\right)$, $f_{-}=\tilde f_{-}$ on $[\psi_0(0),\psi_{0}(-1)]$, $f_{-}\equiv K$ on $[\psi_0(0)-\delta_1,\psi_{0}(0)]$, $f_{+}=\tilde f_{+}$ on $[\psi_0(0),\psi_{0}(1)]$, and $f_{+}\equiv K$ on $[\psi_0(0)-\delta_1,\psi_{0}(0)]$ for a given $\delta_1\in(0,{1\over2}\left(\min\{\psi_0(-\delta_0),\psi_0(\delta_0)\}-\psi_0(0)\right))$.
 Then  we construct steady solutions near
$\left(  u,0\right)  $ by solving the elliptic equation
\[
-\Delta\psi+\beta y=\mathbf{1}_{[-1,0)}(y)f_-(\psi(x,y)) +\mathbf{1}_{[0,1]}(y)f_+(\psi(x,y))=\left\{
\begin{array}
[c]{cc}%
f_-(\psi(x,y)) & \text{if }y<0,\\
f_+(\psi(x,y)) & \text{if }y\geq0.
\end{array}
\right.
\]
 Let
$\xi=\alpha x$ and $\psi\left(  x,y\right)  =\tilde{\psi}\left(  \xi,y\right)  ,$
where $\tilde{\psi}\left(  \xi,y\right)  $ is $2\pi$-periodic in $\xi.$ We use
$\alpha^{2}$ as the bifurcation parameter. The equation for $\tilde{\psi
}\left(  \xi,y\right)  $ becomes
\begin{equation}
-\alpha^{2}\frac{\partial^{2}\tilde{\psi}}{\partial\xi^{2}}-\frac{\partial
^{2}\tilde{\psi}}{\partial y^{2}}+\beta y-\left(\mathbf{1}_{[-1,0)}(y)f_-(\tilde\psi(\xi,y)) +\mathbf{1}_{[0,1]}(y)f_+(\tilde\psi(\xi,y))\right)=0
\label{eqn-psi-tilde}%
\end{equation}
with the boundary conditions that $\tilde{\psi}$ takes constant values on
$\left\{  y=\pm1\right\}  $.
Define a map
 \[\mathscr {F}:
  H^{4}(\mathbb{T}_{{2\pi}}\times\lbrack-1,1])\rightarrow H^{2}(\mathbb{T}_{{2\pi}}\times\lbrack-1,1])\]
  by
  \[\mathscr {F}(\psi)(x,y)=\mathbf{1}_{[-1,0)}(y)f_-(\psi(x,y)) +\mathbf{1}_{[0,1]}(y)f_+(\psi(x,y))=\left\{
\begin{array}
[c]{cc}%
f_-(\psi(x,y)) & \text{if }y<0,\\
f_+(\psi(x,y)) & \text{if }y\geq0.
\end{array}
\right.\]
Then $\mathscr {F}\in C^2(V)$, where $V=\{\psi\in H^4(\mathbb{T}_{{2\pi}}\times\lbrack-1,1])|\|\psi-\psi_0\|_{H^4}\leq \delta_2\}$ for $\delta_2>0$ sufficiently small. In fact, since
$|\psi(x,0)-\psi_0(0)|\leq \|\psi-\psi_0\|_{C^0(\mathbb{T}_{2\pi}\times[-1,1])}\leq C\|\psi-\psi_0\|_{H^4(\mathbb{T}_{2\pi}\times[-1,1])}\leq C\delta_2$ for $x\in \mathbb{T}_{2\pi}$ and $\psi\in V$, we have $\psi(x,0)\in\left(\psi_0(0)-{\delta_1\over2},\psi_0(0)+{\delta_1\over2}\right)$ by taking $\delta_2\in\left(0,{\delta_1\over 2C}\right)$. Thus,
 $f_-(s)=f_+(s)=K$ for $s\in\left(\psi(x,0)-{\delta_1\over2},\psi(x,0)+{\delta_1\over2}\right)$. Since $f_+, f_-\in C_{0}^{3}\left(  \mathbb{R}\right)$, we have $\mathscr {F}(\psi)\in C^2(\mathbb{T}_{{2\pi}}\times[-1,1])$ and $\mathscr {F}\in C^2(V)$. Note that
\begin{align}\label{steady-equ}
-\psi_0''(y)+\beta y=&\mathscr {F}(\psi_0)(y),\\\nonumber
\mathscr {F}'(\psi_0)(y)=& \mathbf{1}_{[-1,0)}(y)f'_{-}(\psi_0(y))+
\mathbf{1}_{[0,1]}(y)f'_{+}(\psi_0(y))=-{u''(y)-\beta\over u(y)}.
\end{align}
Define the perturbation of the stream function
by
\[
\phi\left(  \xi,y\right)  =\tilde{\psi}\left(  \xi,y\right)  -\psi_{0}\left(
y\right)  .
\]
Then we reduce the equation \eqref{eqn-psi-tilde}-\eqref{steady-equ} to
\begin{equation*}
-\alpha^{2}\frac{\partial^{2}\phi}{\partial\xi^{2}}-\frac{\partial^{2}\phi
}{\partial y^{2}}-(\mathscr {F}(\phi+\psi_0)-\mathscr {F}(\psi_0))
=0.
\end{equation*}
Define the spaces%
\begin{align*}
B= \{\phi(\xi,y)\in H^{4}(\mathbb{T}_{2\pi}\times\lbrack-1,1]):\text{ }\phi
(\xi,\pm1)=0\,  \text{ and even in }\xi\}
\end{align*}
and
\[
C=\left\{  \phi(\xi,y)\in H^{2}(\mathbb{T}_{2\pi}\times\lbrack-1,1]):\text{ }%
\text{ even in }\xi\right\}  .
\]
Consider the mapping
\[
F\ :B\times\mathbb{R}^{+}\rightarrow C
\]
defined by
\[
F(\phi,\alpha^{2})=-\alpha^{2}\frac{\partial^{2}\phi}{\partial\xi^{2}}%
-\frac{\partial^{2}\phi}{\partial y^{2}}-(\mathscr {F}(\phi+\psi_0)-\mathscr {F}(\psi_0)).
\]
We study the bifurcation near the trivial solution $\phi=0$ of the equation
$F(\phi,\alpha^{2})=0$ in $B$. First,  $F\in C^2(V\times\mathbb{R}^{+})$. By \eqref{steady-equ}, the linearized operator of $F$
around$\ \left(  0,\alpha_{0}^{2}\right)  $ has the form
\begin{align*}
\mathcal{G}:=F_{\phi}(0,\alpha_{0}^{2})=-\alpha_{0}^{2}\frac{\partial^{2}}{\partial\xi^{2}}-\frac{\partial^{2}%
}{\partial y^{2}}+{u''-\beta\over u}.\end{align*}
Since  $\mu_1=-\alpha_{0}^{2}<0$ is  the principal eigenvalue of  the operator $\mathcal{L}$ and $\mu_n\neq0$ for $n\geq2$,   the kernel of $\mathcal{G}:$
$\ B\rightarrow C\ $is given by
\[
\ker(\mathcal{G})=\left\{  \phi_{0}(y)\cos\xi\right\},
\]
where $\phi_0$ is an eigenfunction corresponding to  $\mu_1$.
Thus, $\dim(\ker({\mathcal{G}}))=1$. Since
{$\mathcal{G}$} is self-adjoint, $\phi_{0}(y)\cos\xi\not \in \text{Ran}\,($%
{$\mathcal{G}$}$)$. Note that $\partial_{\alpha^{2}}\partial_{\phi}%
F(\phi,\alpha^{2})$ is continuous and
\[
\partial_{\alpha^{2}}\partial_{\phi}F(0,\alpha_{0}^{2})\left(  \phi_{0}%
(y)\cos\xi\right)  =-\frac{\partial^{2}}{\partial\xi^{2}}\left[  \phi
_{0}(y)\cos\xi\right]  =\phi_{0}(y)\cos\xi\not \in \text{Ran}\,({\mathcal{G}}).
\]
By the Crandall-Rabinowitz local bifurcation theorem, there exists
a local bifurcating curve $\left(  \phi_{\kappa},\alpha(\kappa)^{2}\right),\kappa\in[-\kappa_0,\kappa_0], $
of the equation $F(\phi,\alpha^{2})=0$, which intersects the trivial curve $\left(
0,\alpha^{2}\right)  $ at $\alpha^{2}=\alpha_{0}^{2}$, such that
\[
\phi_{\kappa}(\xi,y)=\kappa\phi_{0}(y)\cos\xi+o(|\kappa|),
\]
$\alpha(\kappa)^{2}$ is a continuous function of $\kappa$, and $\alpha(0)^{2}=\alpha_{0}^{2}$. So the stream functions of the perturbed steady flows
in $(\xi,y)\ $coordinates take the form
\begin{equation*}
\psi_{(\kappa)}(\xi,y)=\psi_{0}\left(  y\right)  +\kappa\phi_{0}(y)\cos
\xi+o(|\kappa|).
\end{equation*}
Let the velocity $\vec{u}_{(\kappa)}=\left(  u_{(\kappa)},v_{(\kappa)}\right)
=\left(  \partial_{y}\psi_{(\kappa)},-\partial_{x}\psi_{(\kappa)}\right)  $. Then
\begin{align*}
u_{(\kappa)}(x,y)&=u\left(  y\right)+\kappa\phi_{0}^{\prime}(y)\cos(\alpha x)+o(|\kappa|),\\
v_{(\kappa)}(x,y)&=-\alpha\kappa\phi_{0}(y)\sin
(\alpha x)+o(|\kappa|)\not\equiv0
\end{align*}
for $\kappa\in[-\kappa_0,\kappa_0]$.
Then $ \|(u_{(\kappa)},v_{(\kappa)})-(u,0)\|_{H^3([0,{2\pi\over\alpha(\kappa)}]\times[-1,1])} \leq
C_0\kappa$ for some constant $C_0>0$. Then we can choose $\kappa_0$ smaller and $\varepsilon_0=C_0\kappa_0>0$ such that for $\varepsilon\in(0,\varepsilon_0)$, $(u_{\varepsilon},v_{\varepsilon}) :=(u_{(\kappa)},v_{(\kappa)})|_{\kappa=\varepsilon/C_0}$ satisfies that $\|(u_{\varepsilon},v_{\varepsilon})-(u,0)\|_{H^3(D_{T_\varepsilon})} \leq
\varepsilon$ and $\|v_{\varepsilon}\|_{L^2(D_{T_\varepsilon})}\neq0$, where $D_{T_{\varepsilon}}=[0,{2\pi\over\alpha(\varepsilon/C_0)}]\times[-1,1]$ and $T_\varepsilon={2\pi\over\alpha(\varepsilon/C_0)}\to{2\pi\over\alpha_0}$ as $\varepsilon\to0^+$.\\
{\bf Case 2.}
There exists $n_0\geq2$ such that
$\mu_{n_0}=0$.

Choose $\zeta\in C^4([-1,1])$ such that $\zeta\equiv0$ on $[-\delta_0,\delta_0]$ and $\zeta<0$ on $[-1,-\delta_0)\cup(\delta_0,1]$. Let $u_1$ solve the regular ODE
$u_1'' u-(u''-\beta)u_1=\zeta$ on $[-1,1]$. Then $u_1\in C^4([-1,1])$ and $u_1''\equiv0$ on $[-\delta_0,\delta_0]$. Thus, $u''+\nu u_1''-\beta\equiv0$ on $[-\delta_0,\delta_0]$ for $\nu\in\mathbb{R}$. Let $\mu_n(\nu)$ be the $n$-th eigenvalue of the regular BVP
\begin{align*}
\mathcal{L}_{(\nu)}\phi=-{d^2\over dy^2}\phi+{u''+\nu u_1''-\beta\over u+\nu u_1}\phi=\mu\phi,\quad \phi\in H^2\cap H_0^1(-1,1)
\end{align*}
for  $\nu\in[-\nu_0,\nu_0]$, where $n\geq1$, and $\nu_0>0$ satisfies that there exists $\delta_3>0$ such that   $|u(y)+\nu u_1(y)|\geq\delta_3$   for  $y\in[-1,-\delta_0]\cup[\delta_0,1]$ and $\nu\in[-\nu_0,\nu_0]$.  Then \begin{align*}
\mu_{n}'(0)
=\int_{-1}^{1}{u_1''u-(u''-\beta)u_1\over u^2}\phi_{n}^2dy=\int_{-1}^{1}{\zeta\over u^2}\phi_{n}^2dy<0,
\end{align*}
where $\phi_{n}$ is a real $L^2$ normalized eigenfunction of $\mu_{n}(0)\in\sigma(\mathcal{L}_{(0)})$. Then $\mu_1(\cdot)$ and $\mu_{n_0}(\cdot)$ are decreasing on $[-\nu_0,\nu_0]$ if we take $\nu_0>0$ smaller.

Note that $\phi_1$ is linearly independent of $\phi_{n_0}$. Choose $\xi_1,\xi_2\in C^4([-1,1])$ such that $\xi_1,\xi_2$ are supported on $[-1,-\delta_0]\cup[\delta_0,1]$ and
\begin{align*}
\int_{-1}^1{\xi_1\over u^2}\phi_1^2dy\int_{-1}^1{\xi_2\over u^2}\phi_{n_0}^2dy-\int_{-1}^1{\xi_2\over u^2}\phi_1^2dy\int_{-1}^1{\xi_1\over u^2}\phi_{n_0}^2dy\neq0.
\end{align*}
 Then there exist $k_1,k_2\in\mathbb{R}$, which are not both zero, such that $\xi=k_1\xi_1+k_2\xi_2$ satisfies
\begin{align*}
\int_{-1}^1{\xi\over u^2}\phi_1^2dy=1\quad\text{and}\quad\int_{-1}^1{\xi\over u^2}\phi_{n_0}^2dy=-1.
\end{align*}
Let $u_2$ solve the regular ODE
$u_2'' u-(u''-\beta)u_2=\xi$ on $[-1,1]$. Then $u_2\in C^4([-1,1])$ and $u_2''\equiv0$ on $[-\delta_0,\delta_0]$. Thus, $u''+\nu u_1''+\tau u_2''-\beta\equiv0$ on $[-\delta_0,\delta_0]$. Let $\mu_n(\nu,\tau)$ be the $n$-th eigenvalue of the regular BVP
\begin{align*}
\mathcal{L}_{(\nu,\tau)}\phi=-{d^2\over dy^2}\phi+{u''+\nu u_1''+\tau u_2''-\beta\over u+\nu u_1+\tau u_2}\phi=\mu\phi, \quad \phi\in H^2\cap H_0^1(-1,1)
\end{align*}
for  $\nu\in[-\nu_0,\nu_0]$ and $\tau\in[-\tau_0,\tau_0]$, where $n\geq1$, $\tau_0>0$ and $\nu_0>0$ can be taken smaller such that there exists $\delta_4>0$ such that  $|u(y)+\nu u_1(y)+\tau u_2(y)|>\delta_4$ for  $y\in[-1,-\delta_0]\cup[\delta_0,1]$, $\nu\in[-\nu_0,\nu_0]$ and $\tau\in[-\tau_0,\tau_0]$. Then \begin{align*}
\partial_\tau\mu_{1}(0,0)
=&\int_{-1}^{1}{u_2''u-(u''-\beta)u_2\over u^2}\phi_{1}^2dy=\int_{-1}^{1}{\xi\over u^2}\phi_{1}^2dy=1>0,\\
\partial_\tau\mu_{n_0}(0,0)
=&\int_{-1}^{1}{\xi\over u^2}\phi_{n_0}^2dy=-1<0.
\end{align*}
Take $\nu_0,\tau_0>0$ smaller such that
$\partial_\tau\mu_{1}(\nu,\tau)>0$
 for $(\nu,\tau)\in[0,\nu_0]\times[-\tau_0,\tau_0]$.
Note that $\mu_{n_0-1}(0,0)<\mu_{n_0}(0,0)=0<\mu_{n_0+1}(0,0)$. By the continuity of $\partial_\tau\mu_{n_0}$, $\mu_{n_0-1}$ and $\mu_{n_0+1}$, we can choose $\nu_0>0$ and $\tau_0>0$ smaller such that
$
 \partial_{\tau}\mu_{n_0}(\nu,\tau)<0\quad\text{and}\quad \mu_{n_0-1}(\nu,\tau)<0<\mu_{n_0+1}(\nu,\tau)
$
for $(\nu,\tau)\in[0,\nu_0]\times [-\tau_0,\tau_0]$.
Since $ \partial_\tau\mu_{1}(0,0)>0$ and $\mu_{1}(\nu_0,0)<\mu_{1}(0,0)$,
   we can choose $\tilde \tau\in(0,\tau_0]$ such that
 $\mu_{1}(\nu_0,\tilde\tau)<\mu_{1}(0,0)=-\alpha_0^2<\mu_{1}(0,\tilde\tau)$.
 Then there exists $\nu_{\tilde\tau}\in(0,\nu_0)$ such that $\mu_{1}(\nu_{\tilde\tau},\tilde\tau)=-\alpha_0^2. $
On the other hand, $\partial_{\tau}\mu_{n_0}(\nu,\tau)<0$ and $\mu_{n_0}(\cdot,0)$ is decreasing on $[0,\nu_0]$, we have
$
\mu_{n_0}(\nu,\tau)<\mu_{n_0}(\nu,0)<\mu_{n_0}(0,0)=0,
$
which implies that
$
 \mu_{n_0+1}(\nu,\tau)>0>\mu_{n_0}(\nu,\tau)$ for $ (\nu,\tau)\in(0,\nu_0]\times(0,\tau_0].
$
 Since $(\nu_{\tilde\tau},\tilde\tau)\in(0,\nu_0]\times(0,\tau_0]$, we have $\mu_n(\nu_{\tilde\tau},\tilde\tau)\neq0$ for $n\geq2$.
 Fix $\varepsilon\in(0,\varepsilon_0)$. Then we can choose $\nu_0,\tau_0>0$ smaller such that for $ \nu_{\tilde\tau}\in(0,\nu_0]$ and $\tilde \tau\in(0,\tau_0]$,
\begin{align}\label{shear flow perturbation u+tau-varepsilon-u-1-case 2a}
\|(u+\nu_{\tilde\tau} u_1+\tilde\tau u_2,0)-(u,0)\|_{H^3(-1,1)}\leq\nu_{\tilde\tau}\|u_1\|_{H^3(-1,1)} +\tilde\tau\|u_2\|_{H^3(-1,1)}< {\varepsilon\over2}.
\end{align}
Since $\mu_{1}(\nu_{\tilde\tau},\tilde\tau)=-\alpha_0^2$ and $\mu_n(\nu_{\tilde\tau},\tilde\tau)\neq0$ for $n\geq2$, we can apply Case 1 to the shear flow $(u+\nu_{\tilde\tau} u_1+\tilde\tau u_2,0)$ to obtain that there exists a nonparallel steady  solution $(u_\varepsilon,v_\varepsilon)$ which has minimal
 period $T_\varepsilon$ in $x$,
\begin{align}\label{traveling wave to u+tau-varepsilon-u-1-case 2a}
 \|(u_\varepsilon,v_\varepsilon)-(u+\nu_{\tilde\tau} u_1+\tilde\tau u_2,0)\|_{H^3(D_{T_\varepsilon})} \leq{\varepsilon\over2}
\end{align}
 and $\|v_\varepsilon\|_{L^2\left(  D_{T_\varepsilon}\right)}\neq0$. Then by \eqref{shear flow perturbation u+tau-varepsilon-u-1-case 2a}--\eqref{traveling wave to u+tau-varepsilon-u-1-case 2a}, we have
$
 \|(u_\varepsilon,v_\varepsilon)-(u,0)\|_{H^3(D_{T_\varepsilon})}$ $<{\varepsilon}$.
\end{proof}

\noindent{\bf{Construction of the modified shear flow}}\medskip

Fix $\beta\neq0$. Let $\gamma<\min\{\frac{1}{2},\frac{1}{10|\beta|}\}$. 
Recall that
\beno
erf(x)=\frac{2}{\sqrt{\pi}}\int_0^x e^{-s^2} ds,        \quad x\in\mathbb{R}
\eeno
is the Gauss error function introduced in \cite{LZ}.
Let
\beno
\alpha(y)=\left\{
\begin{aligned}
&0,&&x\leq 0,\\
&e^{-\frac{1}{x}},&&x>0,
\end{aligned}
\right.
\eeno
 $$\tilde\eta(x)=\alpha(x-1)\cdot\alpha(2-x),$$
 and
$$\tilde{I}(x)= \frac{\int_{|x|}^{2}\tilde\eta(t)dt}{\int_1^{2}\tilde\eta(t)dt}.$$
Then $\tilde{I}(x)\in C^\infty(\mathbb{R})$, $\tilde{I}(x)=1$ for $|x|\leq 1$ and $\tilde{I}(x)=0$ for $|x|\geq 2$. $\tilde{I}'$ and $\tilde{I}''$ are smooth functions with compact support. Thus, there exist $M_1, M_2>0$ such that  $|\tilde{I}'(x)|\leq M_1$ and $|\tilde{I}''(x)|\leq M_2$ for $x\in\mathbb{R}$. Define the cut-off function by
$$I_{\gamma}(y)=\tilde{I}\left(\frac{y}{\gamma}\right)=\frac{\int_{|y|/\gamma}^{2}\tilde\eta(t)dt}{\int_1^{2}\tilde\eta(t)dt}.$$
Then
\beno
I_{\gamma}(y)=\left\{
\begin{aligned}
&1,&&y\in [-\gamma,\gamma],\\
&\text{smooth},&&y\in[-2\gamma,-\gamma]\cup[\gamma,2\gamma],\\
&0,&&y\in[-1,-2\gamma]\cup[2\gamma,1].
\end{aligned}
\right.
\eeno
We define the modified shear profile
\ben\label{modified shear flow}
U_{\gamma,a}(y)&=&y+\frac{1}{2}\beta y^2 I_{\gamma}(y)+a\gamma^2 erf\left(\frac{y-5\gamma}{\gamma}\right)I_{\gamma}(y-5\gamma), \; y\in[-1,1].
\een
Here, the second term is supported on $[-2\gamma,2\gamma]$, and the last term is supported on $[3\gamma,7\gamma]$.
The second term is added such  that $U_{\gamma,a}''-\beta\equiv0$ on $[-2\gamma,2\gamma]$, which  cancels the singularity near $y=0$. To apply the bifurcation lemma for $\beta\neq0$, we have to add an error function to produce negative eigenvalues of the Rayleigh-Kuo  operator. A direct addition of the Gauss error function in \cite{LZ} induces new singularity near $y=0$. Our novelty is to translate a cut-off Gauss error function away from $0$ but not very far, and to choose  the size of the cut-off function  suitably small.

Denote
\beno
Q_{\gamma,a}(y)&=&\frac{U_{\gamma,a}''(y)-\beta}{U_{\gamma,a}(y)},
\eeno
and define the Rayleigh-Kuo operator  by
\beno
\mathcal{L}_{[\gamma,a]}=-\frac{d^2}{dy^2}+Q_{\gamma,a}(y):  H^2\cap H_0^1(-1,1)\to L^2(-1,1).
\eeno
Then $\mathcal{L}_{[\gamma,a]}\phi=\lambda\phi$ is a regular BVP for $\gamma>0$ small enough and fixed $a\geq0$.
Let $\lambda_{n,\gamma,a}$ be the $n$-th eigenvalue of $\mathcal{L}_{[\lambda,a]}$, and consider it as a function of $(a,\gamma)$ on $[0,+\infty)\times (0,+\infty)$. Then we have the following estimates of the principal eigenvalue $\lambda_{1,\gamma,a}$ of  $\mathcal{L}_{[\gamma,a]}$.

\begin{lemma}\label{prin-neg} Fix $\beta\neq0$ and $a>0$. Then
\begin{align}\label{limsup-lambda1}
\limsup_{\gamma\to0^+}\lambda_{1,\gamma,a}\leq 3+{3\over2}b_0a,
\end{align}
where $b_0=2\int_{0}^{2}\left(\frac{1}{(x+5)^3}-\frac{1}{(-x+5)^3}\right)erf(x)\tilde{I}(x)  dx$ is a negative constant.
In particular, we have the following conclusions.

$(1)$
Fix $a>-\frac{2}{b_0}$. Then $\lambda_{1,\gamma,a}<0$ for $\gamma>0$ sufficiently small.

$(2)$
Fix $d<0$ and $a_d>\frac{2d-6}{3b_0}$. Then $\lambda_{1,\gamma,a_d}<d$ for $\gamma>0$ sufficiently small.
\end{lemma}
\begin{proof} The principal eigenvalue $\lambda_{1,\gamma,a}$ of $\mathcal{L}_{[\gamma,a]}$ satisfies
\begin{align}\nonumber
\lambda_{1,\gamma,a}=&\inf_{\phi\in H_0^1(-1,1),\|\phi\|_{L^2(-1,1)}=1}\langle\mathcal{L}_{[\gamma,a]}\phi,\phi\rangle\\\label{lambda1variation}
=&\inf_{\phi\in H_0^1(-1,1),\|\phi\|_{L^2(-1,1)}=1}\left(\|\phi'\|^2_{L^2(-1,1)}+\int_{-1}^1Q_{\gamma,a}(y)|\phi(y)|^2 dy\right).
\end{align}
Let $\phi_1(y)=1-|y|$ for $y\in[-1,1]$. Then $\|\phi_1\|_{L^2(-1,1)}^2={2\over3}$ and
\begin{align*}
&\langle\mathcal{L}_{[\gamma,a]}\phi,\phi\rangle=\|\phi'_1\|^2_{L^2(-1,1)}+\int_{-1}^1Q_{\gamma,a}(y)\phi_1(y)^2 dy\\
=&2+\int_{-1}^1\frac{\left(a\gamma^2 erf\left(\frac{y-5\gamma}{\gamma}\right)I_{\gamma}(y-5\gamma)\right)''}{U_{\gamma,a}(y)}(1-|y|)^2dy+\int_{-1}^1\frac{\left(\frac{1}{2}\beta y^2I_\gamma(y)\right)''-\beta}{U_{\gamma,a}(y)}(1-|y|)^2dy\\
=&2+B(\gamma,a)+C(\gamma,a).
\end{align*}
Note that
\beno
\frac{1}{U_{\gamma,a}}=\frac{1}{y}\frac{1}{1+\frac{1}{2}\beta\gamma \frac{y}{\gamma}\tilde{I}\left(\frac{y}{\gamma}\right)+\gamma a \frac{\gamma}{y}\Lambda\left(\frac{y-5\gamma}{\gamma}\right)\frac{y-5\gamma}{\gamma}\tilde{I}\left(\frac{y-5\gamma}{\gamma}\right)},
\eeno
where $\Lambda(y)=\frac{erf(y)}{y}$. Since $|x\tilde{I}(x)|\leq C$ and $|\Lambda(x)|\leq C$ for $x\in \mathbb{R}$, $\frac{\gamma}{y}\in[\frac{1}{7},\frac{1}{3}]$ for $y\in [3\gamma,7\gamma]$, we have
\beno
D_\gamma(y):=\frac{1}{1+\frac{1}{2}\beta\gamma \frac{y}{\gamma}\tilde{I}\left(\frac{y}{\gamma}\right)+\gamma a \frac{\gamma}{y}\Lambda\left(\frac{y-5\gamma}{\gamma}\right)\frac{y-5\gamma}{\gamma}\tilde{I}\left(\frac{y-5\gamma}{\gamma}\right)} \to 1
\eeno
as $\gamma \to0^+$.
Since $\tilde{I}(\pm 2)=\tilde{I}'(\pm 2)=0$ and $erf(\cdot)$ is an odd function, we have
\begin{align*}
&\int_{-1}^{1} \frac{1}{y}\left(a\gamma^2 erf\left(\frac{y-5\gamma}{\gamma}\right){I}_{\gamma}({y-5\gamma})\right)'' dy\\
=&\int_{3\gamma}^{7\gamma} \frac{1}{y}\left(a\gamma^2 erf\left(\frac{y-5\gamma}{\gamma}\right)\tilde{I}\left(\frac{y-5\gamma}{\gamma}\right)\right)'' dy\\
=&a\int_{-2}^{2}\frac{1}{x+5}\left(erf(x)\tilde{I}(x)\right)''  dx\\
=&a\int_{-2}^{2}\left(\frac{1}{x+5}\right)''\left(erf(x)\tilde{I}(x)\right)  dx\\
=&2a\int_{-2}^{2}\frac{1}{(x+5)^3}\left(erf(x)\tilde{I}(x)\right)  dx\\
=&2a\int_{0}^{2}\left(\frac{1}{(x+5)^3}-\frac{1}{(-x+5)^3}\right)\left(erf(x)\tilde{I}(x)\right)  dx
=b_0 a,
\end{align*}
where $b_0=2\int_{0}^{2}\left(\frac{1}{(x+5)^3}-\frac{1}{(-x+5)^3}\right)erf(x)\tilde{I}(x)  dx$ is a negative constant. We claim that
 \begin{align} \label{B-lim} B(\gamma,a)\to b_0a \quad\text{as} \quad\gamma\to 0^+.\end{align}
Note that
\begin{align*}
&|B(\gamma,a)-(b_0a)|\\
\leq&\left|\int_{3\gamma}^{7\gamma}\frac{\left(a\gamma^2 erf\left(\frac{y-5\gamma}{\gamma}\right)I_{\gamma}(y-5\gamma)\right)''}{y}
\left(D_\gamma(y)-1\right) |\phi_1(y)|^2dy\right|\\
&+\left|\int_{3\gamma}^{7\gamma}\frac{\left(a\gamma^2 erf\left(\frac{y-5\gamma}{\gamma}\right)I_{\gamma}(y-5\gamma)\right)''}{y}(\phi_1(y)^2-1) dy\right| = B_1(\gamma,a)+B_2(\gamma,a).
\end{align*}
For any $\varepsilon>0$, there exists $\delta_1>0$ such that  $|D_\gamma(y)-1|<\varepsilon$ for $y\in[3\gamma,7\gamma]$ and $0<\gamma<\delta_1$.
Let $\delta_2=\frac{\varepsilon}{14}$. Then $|\phi_1(y)^2-1|<14\gamma<\varepsilon$ for $3\gamma<y<7\gamma$ and $0<\gamma<\delta_2$.
Since $\left|{y-5\gamma\over \gamma}\right|\leq C,$ $|\gamma I_{\gamma}'(y-5\gamma)|\leq C$ and $|\gamma^2 I_{\gamma}''(y-5\gamma)|\leq C$ for $y\in[3\gamma,7\gamma]$, we have
\ben\label{2-derivative}&&\left(\gamma^2 erf\left(\frac{y-5\gamma}{\gamma}\right)I_{\gamma}(y-5\gamma)\right)''\\\nonumber
&=&
\left|{4\over\sqrt{\pi}}e^{-\left({y-5\gamma\over \gamma}\right)^2}\left({-(y-5\gamma)\over \gamma}I_{\gamma}(y-5\gamma)+\gamma I'_\gamma(y-5\gamma)\right)+\gamma^2erf\left({y-5\gamma\over \gamma}\right)I''_{\gamma} (y-5\gamma)\right|\\\nonumber
&\leq& C,\een and thus,
\beno
\int_{3\gamma}^{7\gamma}\left|\frac{\left(\gamma^2 erf\left(\frac{y-5\gamma}{\gamma}\right)I_{\gamma}(y-5\gamma)\right)''}{y}\right|dy
\leq C\int_{3\gamma}^{7\gamma}{1\over y}dy=C\ln(7/3)\leq C.
\eeno
Then
\begin{align*}
|B(\gamma,a)-(b_0a)|\leq |B_1(\gamma,a)|+|B_2(\gamma,a)|\leq Ca\varepsilon
\end{align*}
for $0<\gamma<\delta=\min\{\delta_1,\delta_2\}$.
This proves \eqref{B-lim}.

For $C(\gamma,a)$, let $g_{\gamma}(y)=\left(\frac{1}{2}\beta y^2I_\gamma(y)\right)''-\beta$ for $y\in[-1,1]$, which  is an even function. Note that $g_{\gamma}(y)=0$ for $y\in[0,\gamma]$ and $g_{\gamma}(y)=-\beta$ for $y\in[2\gamma,1]$. For $y\in[\gamma,2\gamma]$, we get
\begin{align*}
g_{\gamma}(y)=&\left(\frac{1}{2}\beta y^2I_\gamma(y)\right)''-\beta\\
=&\beta I_\gamma(y)+2\beta y I_\gamma'(y)+\frac{1}{2}\beta y^2 I_\gamma''(y)-\beta\\
=&\beta (I_\gamma(y)-1)+2\beta {y\over \gamma} \tilde{I}'\left(\frac{y}{\gamma}\right)+\frac{1}{2}\beta \left({y\over\gamma}\right)^{2} \tilde{I}''\left(\frac{y}{\gamma}\right).
\end{align*}
Thus,
\ben\label{g-estimate}
|g_{\gamma}(y)|\leq (1+4M_1+2M_2)|\beta|
\een
for $y\in[\gamma,2\gamma]$. Then
\begin{align*}
&C(\gamma,a)=\int_{-1}^1\frac{g_{\gamma}(y)}{U_{\gamma,a}(y)}(1-|y|)^2dy\\
=&\int_0^1\left(\frac{1}{U_{\gamma,a}(y)}+\frac{1}{U_{\gamma,a}(-y)}\right)g_{\gamma}(y)(1-|y|)^2dy\\
=&\int_{\gamma}^{2\gamma}\frac{-g_{\gamma}(y)(1-|y|)^2 \beta y^2 I_\gamma(y)}{y^2-\left(\frac{1}{2}\beta y^2I_\gamma(y)\right)^2} dy\\
&+\int_{3\gamma}^{7\gamma}
\left(\frac{1}{y+a\gamma^2 \Lambda\left(\frac{y-5\gamma}{\gamma}\right)\frac{y-5\gamma}{\gamma}\tilde{I}\left(\frac{y-5\gamma}{\gamma}\right)}-\frac{1}{y}\right) g_{\gamma}(y)(1-|y|)^2dy\\
=&\int_{\gamma}^{2\gamma}\frac{-g_{\gamma}(y)(1-|y|)^2 \beta \left(\frac{y}{\gamma}\right)^2 I_\gamma(y)}{\left(\frac{y}{\gamma}\right)^2-\left(\frac{1}{2}\beta \gamma  \left(\frac{y}{\gamma}\right)^2I_\gamma(y)\right)^2} dy\\
&+\int_{3\gamma}^{7\gamma} \frac{1}{y}
\left(\frac{1}{1+\gamma a \frac{\gamma}{y}\Lambda\left(\frac{y-5\gamma}{\gamma}\right)\frac{y-5\gamma}{\gamma}\tilde{I}\left(\frac{y-5\gamma}{\gamma}\right)}-1\right) g_{\gamma}(y)(1-|y|)^2dy\\
=&C_1(\gamma,a)+C_2(\gamma,a).
\end{align*}
Since $|g_{\gamma}(y)|\leq C|\beta|$, $|(1-|y|)^2I_\gamma(y)|\leq 1$,  $ \beta \left(\frac{y}{\gamma}\right)^2\leq 4|\beta|$  and $\left(\frac{y}{\gamma}\right)^2-\left(\frac{1}{2}\beta \gamma  \left(\frac{y}{\gamma}\right)^2I_\gamma(y)\right)^2\geq \frac{24}{25}$ for $y\in[\gamma,2\gamma]$  and  $\gamma<\min\{\frac{1}{2},\frac{1}{10|\beta|}\}$, we have
\beno
|C_1(\gamma,a)|\leq C\beta^2 \gamma\to 0
\eeno
as $\gamma\to0^+$.
For any $\varepsilon>0$, there exists $\delta_3>0$ such that
\begin{align*}
\left|\frac{1}{1+\gamma a \frac{\gamma}{y}\Lambda\left(\frac{y-5\gamma}{\gamma}\right)\frac{y-5\gamma}{\gamma}\tilde{I}\left(\frac{y-5\gamma}{\gamma}\right)}-1\right|<\varepsilon
\end{align*}
for $y\in [3\gamma,7\gamma]$ and $0<\gamma<\delta_3$,
 and
thus,
\beno
|C_2(\gamma,a)|\leq C|\beta|\ln(7/3)\varepsilon\to 0
\eeno
as $\gamma\to0^+$.
In summary, for fixed $a>0$ we have
\beno
\lambda_{1,\gamma,a}\leq {3\over2}\langle\mathcal{L}_{[\gamma,a]}\phi_1,\phi_1\rangle={3\over2}(2+B(\gamma,a)+C(\gamma,a))\to 3+{3\over2}b_0a
\eeno
as $\gamma\to 0^+$. This proves \eqref{limsup-lambda1}. The conclusions (1)-(2) are  direct applications   of \eqref{limsup-lambda1}.
\end{proof}

For  $\beta\neq 0$, we study the range of the principal eigenvalue $\lambda_{1,\gamma,a}$ with respect to $a$ for $\gamma>0$ small enough.
Roughly speaking, the range could cover any interval $[a,b]$ for some $b<0$ and any $a<b$.
Note that there exist $M, M_0>0$ such that  $\left|\left(x^2\tilde{I}(x)\right)'\right|\leq M$ and $\left|\left(erf(x)\tilde{I}(x)\right)'\right|\leq M_0$ for $x\in \mathbb{R}$.

\begin{lemma}\label{lem-range-anybeta}
Let $ \beta\neq 0$ and  $C_\beta:=\inf_{0<\gamma\leq \delta_*} \lambda_{1,\gamma,0}$, where $\delta_*\in\left(0,{2\over {|\beta|}M}\right)$. Then $C_\beta>-\infty$, and
for any $d<C_\beta$ and $a_d>{2d-6\over 3b_0}$, there exists $\delta=\delta(a_d)>0$ such that
$[d,C_\beta]\subset  \{\lambda_{1,\gamma,a}:a\in[0,a_d]\}$ for any fixed  $\gamma\in(0,\delta)$.
\end{lemma}
\begin{proof}
Since $\delta_*\in\left(0,{2\over {|\beta|}M}\right)$, we have $|U_{\gamma,0}'(y)|\geq 1-\left|{1\over2}\beta\gamma\left((\cdot)^2\tilde I(\cdot)\right)'\circ\left({y\over\gamma}\right)\right|>0$ for $0<\gamma\leq \delta_*$ and $y\in[-1,1]$. Thus, $\lambda_{1,\gamma,0}$ is continuous with respect to $\gamma\in(0,\delta_*]$. It suffices to show that $\lambda_{1,\gamma,0}$ has lower bound as $\gamma\to 0^+$. For any real function $\phi\in H_0^1(-1,1)$ with $\|\phi\|_{L^2(-1,1)}=1$, we have
\begin{align}\nonumber
&\langle\mathcal{L}_{[\gamma,0]}\phi,\phi\rangle
=\|\phi'\|^2_{L^2(-1,1)}+\int_{-1}^1Q_{\gamma,0}(y)\phi(y)^2 dy\\\nonumber
=&\|\phi'\|^2_{L^2(-1,1)}+\left(\int_{-2\gamma}^{-\gamma}+\int_{\gamma}^{2\gamma}\right)\frac{g_\gamma(y)}{y+\frac{1}{2}\beta y^2I_\gamma(y)}\phi(y)^2 dy+\left(\int_{-1}^{-2\gamma}+\int_{2\gamma}^{1}\right){-\beta\over y}\phi(y)^2 dy\\\label{L-phi-phi-derivative-B-C}
=&\|\phi'\|^2_{L^2(-1,1)}+B(\gamma)+C(\gamma).
\end{align}
For $B(\gamma)$, we have
\begin{align}\nonumber
B(\gamma)=&\int_{\gamma}^{2\gamma}g_\gamma(y)\left({\phi(y)^2\over y+{1\over2}\beta y^2I_\gamma (y)}-{\phi(-y)^2\over y-{1\over2}\beta y^2I_\gamma (y)}\right)dy\\\label{B-gamma-estimates}
=&\int_{\gamma}^{2\gamma}{1\over y}g_\gamma(y){(\phi(y)^2-\phi(-y)^2)-{1\over2}\beta y I_\gamma(y)(\phi(y)^2+\phi(-y)^2)\over 1-\left({1\over2}\beta yI_\gamma (y)\right)^2}dy.
\end{align}
Note that $|g_\gamma(y)|<C|\beta|< C$, $|\phi(y)-\phi(-y)|\leq \|\phi'\|_{L^2(-1,1)}\sqrt{2y}$, $|\phi(y)+\phi(-y)|\leq |\phi(y)|+|\phi(-y)|\leq \|\phi'\|_{L^2(-1,1)}\sqrt{y+1}+\|\phi'\|_{L^2(-1,1)}\sqrt{-y+1}\leq 2\sqrt{2}\|\phi'\|_{L^2(-1,1)}$ and $\phi(y)^2+\phi(-y)^2\leq \|\phi'\|_{L^2(-1,1)}^2(y+1)+\|\phi'\|_{L^2(-1,1)}^2(-y+1)\leq 4\|\phi'\|_{L^2(-1,1)}^2$ for $y\in[\gamma,2\gamma]$.
For any $\varepsilon_0>0$, there exists $\delta_0>0$ such that $\left|1-\left({1\over2}\beta yI_\gamma (y)\right)^2\right|\geq{1\over2}$ and $|{1\over2}\beta yI_\gamma (y)|+\sqrt{y}\leq {\varepsilon_0\over 8C \ln2}$ for $y\in[\gamma,2\gamma]$, where $\gamma\in(0,\delta_0)$. Then
\ben\label{B-gamma}
|B(\gamma)|&\leq&{\varepsilon_0\over 8C \ln2}8C\|\phi'\|_{L^2(-1,1)}^2\int_{\gamma}^{2\gamma}{1\over y}dy\leq\|\phi'\|_{L^2(-1,1)}^2\varepsilon_0
\een
for $\gamma\in(0,\delta_0)$. For $C(\gamma)$, we have
\begin{align*}
&C(\gamma)=-\beta\left(\int_{-1}^{-2\gamma}+\int_{2\gamma}^{1}\right){1\over y}\phi(y)^2 dy\\
=&-\beta\left(\phi(y)^2 \ln y \bigg|_{2\gamma}^1-\int_{2\gamma}^1 2\phi(y)\phi'(y)\ln y dy+\phi(y)^2\ln |y| \bigg|_{-1}^{-2\gamma}-\int_{-1}^{-2\gamma} 2\phi(y)\phi'(y)\ln |y| dy\right)\\
=&-\beta\left(\ln(2\gamma)\left(\phi(-2\gamma)^2-\phi(2\gamma)^2\right)-\left(\int_{-1}^{-2\gamma}+\int_{2\gamma}^1\right) 2\phi(y)\phi'(y)\ln |y| dy\right)\\
=&-\beta(I_\gamma+II_\gamma).
\end{align*}
Choose $\delta_0>0$ smaller such that
\beno
|I_\gamma|&=&|\ln(2\gamma)| \cdot|\phi(-2\gamma)-\phi(-1)+\phi(2\gamma)-\phi(1)| \cdot|\phi(-2\gamma)-\phi(2\gamma)|\\
&\leq&|\ln(2\gamma)| (4\gamma)^{1\over2} \|\phi'\|^2_{L^2(-1,1)}<\varepsilon_0\|\phi'\|^2_{L^2(-1,1)}.
\eeno
For any $\varepsilon_1, \varepsilon_2>0$, we have
\beno
|II_\gamma|&\leq& \left(\int_{-1}^{-2\gamma}+\int_{2\gamma}^1\right)\left(\frac{1}{\varepsilon_1} (\phi(y)\ln |y|)^2+\varepsilon_1|\phi(y)'|^2 \right)dy\\
&\leq&\left(\int_{-1}^{-2\gamma}+\int_{2\gamma}^1\right)\left(\frac{1}{2\varepsilon_1\varepsilon_2}|\ln |y||^4+\frac{\varepsilon_2}{2\varepsilon_1}  |\phi(y)|^4+\varepsilon_1|\phi'(y)|^2 \right)dy\\
&\leq&\frac{1}{2\varepsilon_1\varepsilon_2}\int_{-1}^1 |\ln |y||^4 dy +\frac{\varepsilon_2}{2\varepsilon_1} \int_{-1}^1 |\phi(y)|^4dy +\varepsilon_1\int_{-1}^1|\phi'(y)|^2 dy\\
&\leq&\frac{C_0}{2\varepsilon_1\varepsilon_2} +\frac{\varepsilon_2}{2\varepsilon_1} \left(C_1\int_{-1}^1 |\phi'(y)|^2dy+C_2\right) +\varepsilon_1\int_{-1}^1|\phi'(y)|^2 dy,
\eeno
where we used   Gagliardo-Nirenberg interpolation inequality (\ref{Gagliardo-Nirenberg inequality}) in the last inequality. Take $\varepsilon_0$, $\varepsilon_1$ and $\varepsilon_2$ smaller such that $\varepsilon_0+|\beta|\left(\varepsilon_0+\varepsilon_1+\frac{C_1\varepsilon_2}{2\varepsilon_1}\right)<1$. Then we have
\begin{align*}\langle\mathcal{L}_{[\gamma,0]}\phi,\phi\rangle\geq&\left(1-\varepsilon_0-|\beta|\left(\varepsilon_0+\varepsilon_1+
\frac{C_1\varepsilon_2}{2\varepsilon_1}\right)\right)\|\phi'\|^2_{L^2(-1,1)} -\left(\frac{C_0}{2\varepsilon_1\varepsilon_2}+\frac{C_2\varepsilon_2}{2\varepsilon_1}\right)|\beta|\\
\geq&-\left(\frac{C_0}{2\varepsilon_1\varepsilon_2}+\frac{C_2\varepsilon_2}{2\varepsilon_1}\right)|\beta|>-\infty
\end{align*}
for $\gamma\in(0,\delta_0)$. This proves that  $\lambda_{1,\gamma,0}$ has lower bound as $\gamma\to 0^+$, and thus, $C_\beta>-\infty$.

For any $d<C_\beta$ and $a_d>{2d-6\over 3b_0}$, we infer from  Lemma \ref{prin-neg} (2) that there exists $\delta_{a_d}>0$ such that $\lambda_{1,\gamma,a_d}<d$ for  $0<\gamma<\delta_{a_d}$. By the definition of $C_\beta$, we have $\lambda_{1,\gamma,0}\geq C_\beta$ for $0<\gamma<\delta_*$.
 For  fixed $0<\gamma<\delta=\delta(a_d)=\min\left\{\delta_*,\delta_{a_d},\frac{1}{\frac{|\beta|}{2}M+a_dM_0}\right\}$, we claim that $\lambda_{1,\gamma,a}$ is  continuous with respect to $a\in[0,a_d]$.
In fact, since $0<\gamma<\frac{1}{\frac{|\beta|}{2}M+a_dM_0}$ and
\beno
U_{\gamma,a}(y)=y+\frac{1}{2}\beta\gamma^2\left((\cdot)^2\tilde{I}(\cdot)\right)\circ\left(\frac{y}{\gamma}\right) +a\gamma^2 \left(erf(\cdot)\tilde{I}(\cdot)\right)\circ\left(\frac{y-5\gamma}{\gamma}\right),\eeno we have
\ben\label{U-gamma-a-der}
U_{\gamma,a}'(y)=1+\frac{1}{2}\beta\gamma\left((\cdot)^2\tilde{I}(\cdot)\right)'\circ\left(\frac{y}{\gamma}\right) +a\gamma \left(erf(\cdot)\tilde{I}(\cdot)\right)'\circ\left(\frac{y-5\gamma}{\gamma}\right)>0
\een
for $y\in[-1,1]$, which implies that  there is no singularity for $Q_{\gamma,a}$ and thus, $\lambda_{1,\gamma,a}$ is  continuous on $a\in[0,a_d]$.
This, along with the fact that
 $\lambda_{1,\gamma,a_d}<d$ and $\lambda_{1,\gamma,0}\geq C_\beta$, implies that $[d,C_\beta]\subset  \{\lambda_{1,\gamma,a}:a\in[0,a_d]\}$ for any fixed  $\gamma\in(0,\delta)$.
\end{proof}
\begin{remark}\label{beta-gg-1}
$(1)$ For $|\beta|\gg1$, $\lambda_{1,\gamma, 0}<0$  for   $\gamma>0$ small enough, and thus, we  can not expect $C_\beta>0$ in general. In fact, for any fixed $n\geq1$, if $|\beta|$ is large enough, then
 $\lambda_{n,\gamma, a}(\beta)<0$  uniformly for $a\in[0,1]$ and $\gamma\in\left(0,\frac{1}{\frac{|\beta|}{2}M+M_0}\right)$, where we write $\lambda_{n,\gamma, a}=\lambda_{n,\gamma, a}(\beta)$ to indicate its dependence on $\beta$.

 We only give the proof for $\beta\gg1$.
Let $\varphi_i$, $1\leq i\leq n$,  be $n$ smooth functions such that $ \text{supp}\,(\varphi_i)=\left({1\over2}+{i-1\over 2n},{1\over2}+{i\over 2n}\right)$ and
 $\|\varphi_i\|_{L^2(-1,1)}=1$. Thus,
$\varphi_i\bot\varphi_j$ in the  sense of $L^2$ for $i\neq j$.
Let
$V_{n}=\text{span}\{\varphi_1,\cdots,\varphi_{n}\}.$ Then $V_{n}\subset H_0^1(-1,1)$.
Moreover,
$C_n:=\max\limits_{1\leq i\leq n}\int_{{1\over2}+{i-1\over 2n}}^{{1\over2}+{i\over 2n}}\left(|\varphi_i'|^2-{\beta\over y}|\varphi_i|^2\right)dy<0$ for $\beta>0$ sufficiently large.
Since  $\gamma\in\left(0,\frac{1}{\frac{|\beta|}{2}M+M_0}\right)$,  there is no singularity on $Q_{\gamma,a}$ by \eqref{U-gamma-a-der}, and thus,
 \begin{align}\label{n-th-eigenvalue-gamma-a}
\lambda_{n,\gamma, a}(\beta)
=&\inf_{\dim V_n=n}\;\;\sup_{\phi\in H_0^1(-1,1), \phi\in V_n, \|\phi\|_{L^2(-1,1)}=1}\int_{-1}^{1}\left(|\phi'|^2+{U_{\gamma,a}''-\beta\over U_{\gamma,a}}|\phi|^2\right)dy.
\end{align}
 Then there exist $b_{i,\beta}\in\mathbf{R}$, $i=1,\cdots,n$, with $\sum_{i=1}^{n}|b_{i,\beta}|^2=1 $ such that $\varphi_\beta=\sum_{i=1}^{n}b_{i,\beta}\varphi_i\in V_{n}$ with $\|\varphi_\beta\|^2_{L^2}=1$, and
\begin{align}\nonumber
\lambda_{n,\gamma, a}(\beta)\leq &\sup_{\|\phi\|_{L^2}=1,\phi\in V_{n}}\int_{-1}^{1}\left(|\phi'|^2+{U_{\gamma,a}''-\beta\over U_{\gamma,a}}|\phi|^2\right)dy=
\int_{-1}^{1}\left(|\varphi_\beta'|^2+{U_{\gamma,a}''-\beta\over U_{\gamma,a}}|\varphi_\beta|^2\right)dy\\ \nonumber
=&\int_{{1\over2}}^{1}\left(|\varphi_\beta'|^2-{\beta\over y}|\varphi_\beta|^2\right)dy=\sum_{i=1}^{n}|b_{i,\beta}|^2\int_{{1\over2}+{i-1\over 2n}}^{{1\over2}+{i\over 2n}}\left(|\varphi_i'|^2-{\beta\over y}|\varphi_i|^2\right)dy\leq C_n<0
\end{align}
uniformly for $a\in[0,1]$ and $\gamma\in\left(0,\frac{1}{\frac{|\beta|}{2}M+M_0}\right)$.

$(2)$  As is indicated in Remark $\ref{beta-gg-1}$ $(1)$, the number of negative eigenvalues of  the Rayleigh-Kuo operator
$
\mathcal{L}_{[\gamma,a]}$ could take any  positive  integer  as long as  $|\beta|$ is taken appropriately large. Then we prove that  the number is exactly $1$ for some $(\gamma,a)$ when $|\beta|$ is small.
Precisely,
 let $|\beta|<1$ and $a_0>0$. Then there exists $\delta=\delta(a_0)>0$ such that   $\lambda_{2,\gamma, a}>0$ for $a\in[0,a_0]$ and $\gamma\in(0,\delta)$.

For any 2-dimensional space $V_2=\text{span}\{ \phi_1,\phi_2\}\subset H_0^1(-1,1)$, there exists $0\neq(\xi_{1}, \xi_{2})\in \mathbb{R}^2$  such that  $\xi_{1}\phi_1(0)+\xi_{2}\phi_2(0)=0$. Define $\phi_{*}=\xi_{1}\phi_1+\xi_{2}\phi_2$.  We normalize   $\phi_{*}$ such that $\|\phi_{*}\|_{L^2(-1,1)}=1$. Then $ \phi_{*}\in V$, $\phi_{*}(0)=0$
 and
\begin{align*}
&\langle\mathcal{L}_{[\gamma,a]}\phi_*,\phi_*\rangle
=\|\phi'_{*}\|^2_{L^2(-1,1)}+\int_{-1}^1\frac{\left(a\gamma^2 erf\left(\frac{y-5\gamma}{\gamma}\right)I_{\gamma}(y-5\gamma)\right)''}{U_{\gamma,a}(y)}\phi_{*}(y)^2 dy\\
&+\int_{-1}^1\frac{\left(\frac{1}{2}\beta y^2I_\gamma(y)\right)''-\beta}{U_{\gamma,a}(y)}\phi_{*}(y)^2 dy
=\|\phi'_{*}\|^2_{L^2(-1,1)}+B_*(\gamma,a)+C_*(\gamma,a).
\end{align*}
For $B_*(\gamma,a)$, we have
\begin{align}\label{B*}
|B_*(\gamma,a)|
=\left|\int_{3\gamma}^{7\gamma}{1\over y}{\left(a\gamma^2 erf\left(\frac{y-5\gamma}{\gamma}\right)I_{\gamma}(y-5\gamma)\right)''\over 1+{a\gamma^2\over y}erf\left({y-5\gamma\over \gamma}\right)I_{\gamma} (y-5\gamma)}\phi_{*}(y)^2dy\right|.
\end{align}
Choose $\delta_1({a_0})>0$ such that $\left|1+{a\gamma^2\over y}erf\left({y-5\gamma\over \gamma}\right)I_{\gamma} (y-5\gamma)\right|>{1\over2}$ for  $y\in [{3\gamma},{7\gamma}]$, where $\gamma\in(0,\delta_1({a_0}))$. By \eqref{2-derivative},
we have
\begin{align}\label{B*2}
\left|{\left(\gamma^2 erf\left(\frac{y-5\gamma}{\gamma}\right)I_{\gamma}(y-5\gamma)\right)''\over 1+{a\gamma^2\over y}erf\left({y-5\gamma\over \gamma}\right)I_{\gamma} (y-5\gamma)}\right|\leq C_1
\end{align}
for $y\in[3\gamma,7\gamma]$.
For $C_*(\gamma,a)$, we have
\begin{align}\nonumber
|C_*(\gamma,a)|\leq&\left(\int_{-1}^{-7\gamma}+\int_{7\gamma}^{1}+\int_{-7\gamma}^{-2\gamma}+\int_{2\gamma}^{3\gamma}\right){|\beta|\over |y|} \phi_{*}(y)^2dy\\\nonumber
&+\left(\int_{-2\gamma}^{-\gamma}+\int_{\gamma}^{2\gamma}\right){1\over |y|} \frac{|g_{\gamma}(y)|}{|1+\frac{1}{2}\beta y I_{\gamma}(y)|}\phi_{*}(y)^2dy\\\label{C*}
&+\int_{3\gamma}^{7\gamma}{1\over y}{|\beta|\over \left|1+a\gamma\left({\gamma\over y}\right)erf\left({y-5\gamma\over \gamma}\right)I_{\gamma} (y-5\gamma)\right|} \phi_{*}(y)^2dy.
\end{align}
Choose $\delta_2({a_0})>0$ such that  $\left|1+a\gamma\left({\gamma\over y}\right)erf\left({y-5\gamma\over \gamma}\right)I_{\gamma} (y-5\gamma)\right|>{1\over2} $ for $y\in[{3\gamma},{7\gamma}]$ and ${\left|1+\frac{1}{2}\beta y I_{\gamma}(y)\right|}>{1\over2}$ for $y\in [{-2\gamma},{-\gamma}]\cup[{\gamma},{2\gamma}]$, where $\gamma\in(0,\delta_2(a_0))$. This, along with \eqref{g-estimate}, gives
\begin{align}\label{C*2}
{|\beta|\over \left|1+a\gamma\left({\gamma\over y_1}\right)erf\left({y_1-5\gamma\over \gamma}\right)I_{\gamma} (y_1-5\gamma)\right|}\leq 2|\beta|,\;\;
\frac{|g_{\gamma}(y)|}{|1+\frac{1}{2}\beta y I_{\gamma}(y)|}\leq C_2|\beta|
\end{align}
for $y_1\in[{3\gamma},{7\gamma}]$ and $y\in[{-2\gamma},{-\gamma}]\cup[{\gamma},{2\gamma}]$.
Let $\varepsilon_0>0$ and  $\varepsilon_1>0$ such that $\varepsilon_0+ \varepsilon_1+{|\beta|}={1}$ and set $\delta(a_0)=\min\left\{{1\over 4C_1+2C_2|\beta|+14|\beta|}\varepsilon_0,\delta_1(a_0),\delta_2(a_0)\right\}>0.$
   Since $\phi_{*}(0)=0$, we have
$|\phi_{*}(y)|^2\leq \|\phi_{*}'\|_{L^2(-1,1)}^2|y|$ for $y\in[-1,1]$.
Thus, by \eqref{B*}, \eqref{B*2}, \eqref{C*} and \eqref{C*2} we have
\begin{align}\nonumber
&|B_*(\gamma,a)|+|C_*(\gamma,a)|\\\nonumber
\leq& \|\phi_{*}'\|_{L^2(-1,1)}^2(4 C_1\gamma+6|\beta|\gamma+2C_2|\beta|\gamma+8|\beta|\gamma)+\left(\int_{-1}^{-7\gamma}+\int_{7\gamma}^{1}\right){|\beta|\over |y|} \phi_{*}(y)^2dy\\\label{BC*}
\leq&\|\phi_{*}'\|_{L^2(-1,1)}^2\varepsilon_0+\left(\int_{-1}^{-7\gamma}+\int_{7\gamma}^{1}\right){|\beta|\over |y|} \phi_{*}(y)^2dy
\end{align}
for $\gamma\in(0,\delta(a_0))$.
Since $\phi_{*}(0)=0$,  we have $\phi_{*}(y)^2\leq \|\phi_{*}'\|_{L^2(-1,0)}^2|y|$ for $y\in[-1,0]$, and $\phi_{*}(y)^2\leq \|\phi_{*}'\|_{L^2(0,1)}^2|y|$ for $y\in[0,1]$. Thus,
\begin{align}\label{C*-gamma}
\left(\int_{-1}^{-7\gamma}+\int_{7\gamma}^{1}\right){|\beta|\over |y|} \phi_{*}(y)^2dy\leq |\beta|(\|\phi_*'\|_{L^2(-1,0)}^2+\|\phi_*'\|_{L^2(0,1)}^2)=|\beta|\|\phi_*'\|_{L^2(-1,1)}^2.
\end{align}
By \eqref{eigen-mu}, we have
$
{\pi^2\over4}={\pi^2\over4}\|\phi_{*}\|_{L^2(-1,1)}^2\leq \|\phi_{*}'\|_{L^2(-1,1)}^2.
$
Then by \eqref{BC*} and \eqref{C*-gamma}, we have
\begin{align*}
&\langle\mathcal{L}_{[\gamma,a]}\phi_*,\phi_*\rangle
=\|\phi'_{*}\|^2_{L^2(-1,1)}+B_*(\gamma,a)+C_*(\gamma,a)\\
\geq&\|\phi'_{*}\|^2_{L^2(-1,1)}-\varepsilon_0\|\phi_{*}'\|_{L^2(-1,1)}^2-{|\beta|}\|\phi_*'\|_{L^2(-1,1)}^2
=\varepsilon_1{\|\phi'_{*}\|^2_{L^2(-1,1)}}\geq {\pi^2\over4}\varepsilon_1
\end{align*}
for $\gamma\in(0,\delta(a_0))$.
This, along with \eqref{n-th-eigenvalue-gamma-a}, implies that $\lambda_{2,\gamma,a}\geq{\pi^2\over4}\varepsilon_1>0$ for $0<\gamma<{\delta(a_0)}$.
\end{remark}

In contrast with the case $|\beta|\gg1$ in Remark \ref{beta-gg-1}, the principal eigenvalue of the Rayleigh-Kuo operator $\mathcal{L}_{[\gamma,0]}$ is positive for $|\beta|<{4\sqrt{2}\over3\pi}$ and $\gamma>0$ small enough.

\begin{lemma}\label{prin-0-pos}
Let $|\beta|<{4\sqrt{2}\over3\pi}$. If $\gamma>0$ is small enough, then $\lambda_{1,\gamma,0}>0$.
\end{lemma}
\begin{proof}
To estimate the lower bound of $\lambda_{1,\gamma,0}$, we take any real function $\phi\in H_0^1(-1,1)$ with $\|\phi\|_{L^2(-1,1)}=1$, and  decompose $\langle\mathcal{L}_{[\gamma,0]}\phi,\phi\rangle$ as follows.
\beno
\langle\mathcal{L}_{[\gamma,0]}\phi,\phi\rangle
=\|\phi'\|^2_{L^2(-1,1)}+B(\gamma)+C(\gamma),
\eeno
where $B(\gamma)$ and $C(\gamma)$ are the same terms in \eqref{L-phi-phi-derivative-B-C}.
Choose $\varepsilon_0>0$ and $\varepsilon_1>0$ such that ${3\sqrt{2}\pi|\beta|\over8}+\varepsilon_0+\varepsilon_1=1$. Choose $\delta_0>0$ such that $\left|1-\left({1\over2}\beta yI_\gamma (y)\right)^2\right|\geq{1\over2}$, $|{1\over2}\beta yI_\gamma (y)|+\sqrt{y}\leq {\varepsilon_0\over 8C \ln2}$ for $y\in[\gamma,2\gamma]$, where $\gamma\in(0,\delta_0)$. By a similar argument as in \eqref{B-gamma-estimates}-\eqref{B-gamma}, we have
\begin{align}\label{B-gamma2}
|B(\gamma)|\leq\|\phi'\|_{L^2(-1,1)}^2\varepsilon_0
\end{align} for $\gamma\in(0,\delta_0)$.
Since $|\phi(y)-\phi(-y)|\leq \|\phi'\|_{L^2(-y,y)}\sqrt{2y},$ and
 \ben\nonumber
 &&|\phi(y)+\phi(-y)|=|\phi(y)-\phi(1)+\phi(-y)-\phi(-1)|\\\nonumber
 &\leq& \sqrt{1-y}\left(\|\phi'\|_{L^2(y,1)}+\|\phi'\|_{L^2(-1,-y)}\right),
 \een we have
\ben\nonumber
|C(\gamma)|&\leq&|\beta|\int_{2\gamma}^1{1\over y}|\phi(y)-\phi(-y)||\phi(y)+\phi(-y)|dy\\\nonumber
&\leq &|\beta|\sqrt{2}\int_{2\gamma}^1{1\over\sqrt{ y}}\sqrt{1-y}\|\phi'\|_{L^2(-y,y)}\left(\|\phi'\|_{L^2(y,1)}+\|\phi'\|_{L^2(-1,-y)}\right)dy\\\nonumber
&\leq &|\beta|\sqrt{2}\int_{2\gamma}^1{1\over\sqrt{ y}}\sqrt{1-y}\left({1\over4}\left(\|\phi'\|_{L^2(-y,y)}+\|\phi'\|_{L^2(y,1)}+\|\phi'\|_{L^2(-1,-y)}\right)^2\right)dy\\\label{C-gamma}
&\leq &|\beta|{3\sqrt{2}\over4}\|\phi'\|_{L^2(-1,1)}^2\int_{0}^1{1\over\sqrt{ y}}\sqrt{1-y}dy={3\sqrt{2}\pi|\beta|\over8}\|\phi'\|_{L^2(-1,1)}^2.
\een
 Let $\tilde\lambda_1$
be the principal eigenvalue of $-\phi''=\tilde\lambda\phi$, $\phi(\pm1)=0$. Then $\tilde\lambda_1={\pi^2\over4}$ and
\begin{align}\label{eigen-mu}
{\pi^2\over4}={\pi^2\over4}\|\phi\|_{L^2(-1,1)}^2\leq \|\phi'\|_{L^2(-1,1)}^2.
\end{align}
 Then by \eqref{B-gamma2}, \eqref{C-gamma} and \eqref{eigen-mu}, we have
\beno
\langle\mathcal{L}_{[\gamma,0]}\phi,\phi\rangle
&=&\|\phi'\|^2_{L^2(-1,1)}+B(\gamma)+C(\gamma)
\geq\left(1-\varepsilon_0-{3\sqrt{2}\pi|\beta|\over8}\right)\|\phi'\|_{L^2(-1,1)}^2\\
&=&\varepsilon_1{\|\phi'\|^2_{L^2(-1,1)}}\geq {\pi^2\over4}\varepsilon_1.
\eeno
This, along with \eqref{lambda1variation}, implies that $\lambda_{1,\gamma,0}\geq{\pi^2\over4}\varepsilon_1>0$ for $0<\gamma<\delta_0$.
\end{proof}
By means of the above lemma, we get better conclusion about
the range of the principal eigenvalue $\lambda_{1,\gamma,a}$ with respect to $a$ for $|\beta|<{4\sqrt{2}\over3\pi}$.

\begin{lemma}\label{lem-range}
Let $|\beta|<{4\sqrt{2}\over3\pi}$.
For any $d<0$ and $a_d>{2d-6\over 3b_0}$, there exists $\delta=\delta(a_d)>0$ such that
$[d,0]\subset  \{\lambda_{1,\gamma,a}:a\in[0,a_d]\}$ for any fixed  $\gamma\in(0,\delta)$.
 \end{lemma}
\begin{proof}
By Lemma \ref{prin-neg} (2), there exists $\delta_{a_d}>0$ such that $\lambda_{1,\gamma,a_d}<d$ for  $0<\gamma<\delta_{a_d}$.
By Lemma \ref{prin-0-pos},  there exists $\delta_{0}>0$ such that $\lambda_{1,\gamma,0}>0$ for  $0<\gamma<\delta_{0}$. Recall that $\left|\left(x^2\tilde{I}(x)\right)'\right|\leq M$ and $\left|\left(erf(x)\tilde{I}(x)\right)'\right|\leq M_0$ for $x\in \mathbb{R}$.
For  fixed $0<\gamma<\delta=\delta(a_d)=\min\left\{\delta_0,\delta_{a_d},\frac{1}{\frac{|\beta|}{2}M+a_dM_0}\right\}$,
$\lambda_{1,\gamma,a}$ is  continuous on $a\in[0,a_d]$ by a similar argument with \eqref{U-gamma-a-der}.
Since $\lambda_{1,\gamma,0}>0$, $\lambda_{1,\gamma,a_d}<d$ and $\lambda_{1,\gamma,a}$ is  continuous on $a\in[0,a_d]$, we have $[d,0]\subset  \{\lambda_{1,\gamma,a}:a\in[0,a_d]\}$ for  fixed  $\gamma\in(0,\delta)$.
\end{proof}

Now, we are in the position to prove  Theorem \ref{thm1}.

\begin{proof} First, we prove (1).
Let $T>0$ and $\alpha={2\pi\over T}$. Choose  $k\in\mathbb{Z}^+$ such that  $d_k:=-(k\alpha)^2<-|C_\beta|$. Set $a_{2d_k}=\frac{4d_k-6}{3b_0}+1$. Taking $d=2d_k$ in Lemmas \ref{prin-neg} (2) and \ref{lem-range-anybeta}, there exists $\delta(a_{2d_k})>0$ small enough such that
 $\lambda_{1,\gamma,a_{2d_k}}<2d_k$
and $[2d_k,C_\beta]\subset  \{\lambda_{1,\gamma,a}:a\in[0,a_{2d_k}]\}$ for any $\gamma\in(0,\delta(a_{2d_k}))$.

 For any $\tilde{d}\in[2d_k,C_\beta]$ and $\gamma\in(0,\delta(a_{2d_k}))$, we can define
$$f_{\tilde {d}}(\gamma)=\inf_{0<a<a_{2d_k}}\{a|\lambda_{1,\gamma,a}=\tilde{d}\}.$$
Thus,
\ben\label{lin18}
\sqrt{-\lambda_{1,\gamma,f_{\tilde d_k}(\gamma)}}=\sqrt{-\tilde d_k}<k\alpha<\sqrt{-{3\over2}d_k}=\sqrt{-\lambda_{1,\gamma,f_{3d_k/2}(\gamma)}},
\een
where $\tilde d_k\in(d_k,-|C_\beta|)$.
Now let us check that  for $a\in[0,a_{2d_k}]$ and  $s\in[0,\frac{5}{2})$,
\ben\label{0.4}
\|(U_{\gamma,a},0)-(y,0)\|_{H^s(D_T)} \to 0
\een
as $\gamma\to 0^+$.
 In fact, using the Fourier transform, we have
\beno
&&\left\|f \left(\frac{y}{\gamma}\right) \right\|^2_{\dot{H}^s(\mathbb{R})}=\int_\mathbb{R}|\xi|^{2s}\gamma^2|\hat{f}(\gamma \xi)|^2 d\xi=\int_\mathbb{R}|\eta|^{2s}\gamma^{-2s+1}|\hat{f}(\eta)|^2 d\eta=\gamma^{-2s+1}\|f(x)\|^2_{\dot{H}^s(\mathbb{R})},\\
&&\left\|f\left(\frac{y}{\gamma} \right) \right\|^2_{L^2(\mathbb{R})}=\int_\mathbb{R}\gamma^2|\hat{f}(\gamma \xi)|^2 d\xi=\int_\mathbb{R}\gamma |\hat{f}(\eta)|^2 d\eta=\gamma \|f(x)\|^2_{L^2(\mathbb{R})},\\
&&\left\|f \left(\frac{y}{\gamma}\right) \right\|^2_{H^s(\mathbb{R})}\leq C\left(\left\|f \left(\frac{y}{\gamma}\right) \right\|^2_{\dot{H}^s(\mathbb{R})}+\left\|f\left(\frac{y}{\gamma} \right) \right\|^2_{L^2(\mathbb{R})}\right)
\leq C\gamma^{-2s+1}\|f(x)\|^2_{H^s(\mathbb{R})}
\eeno
for any $f\in H^s(\mathbb{R})$ and $\gamma<1$.
Then
\begin{align*}
\|U_{\gamma,a}-y\|_{H^s(-1,1)}\leq& a_{2d_k}\left\|\gamma^2 erf\left(\frac{y-5\gamma}{\gamma}\right){I}_\gamma\left({y-5\gamma}\right)\right\|_{H^{s}(\mathbb{R})}
+{1\over2}|\beta|\left\|y^2{I}_\gamma({y})\right\|_{{H}^{s}(\mathbb{R})}\\
=&\gamma^2 \left(a_{2d_k}\left\|erf\left(\frac{y-5\gamma}{\gamma}\right)\tilde{I}\left(\frac{y-5\gamma}{\gamma}\right)\right\|_{H^{s}(\mathbb{R})}
+{1\over2}|\beta|\left\|\left(\frac{y}{\gamma}\right)^2\tilde{I}\left(\frac{y}{\gamma}\right)\right\|_{H^{s}(\mathbb{R})}\right)\\
=&\gamma^{2-\frac{2s-1}{2}} \left(Ca_{2d_k}\|erf(x-5)\tilde{I}(x-5)\|_{H^{s}(\mathbb{R})}
+C|\beta|\|x^2\tilde{I}(x)\|_{H^{s}(\mathbb{R})}\right)\\
=&C\gamma^{\frac{5}{2}-s},
\end{align*}
which implies (\ref{0.4}) by the assumption that $s< \frac{5}{2}$.
Here, $erf(\cdot-5)\tilde{I}(\cdot-5)$ and $(\cdot)^2\tilde{I}(\cdot)$ are smooth functions with compact support and thus, belong to $H^{s}(\mathbb{R})$.

For any $\varepsilon>0$, by taking $\delta(a_{2d_k})>0$ smaller, we have
\ben\label{lin20}
\|(U_{\gamma,a},0)-(y,0)\|_{H^s(D_T)}\leq \frac{\varepsilon}{2} \quad\textrm{for}\quad (\gamma,a)\in\Omega_{d_k},
\een
where $\Omega_{d_k} =\{(\gamma,a)|0< \gamma\leq \delta(a_{2d_k}),  f_{\tilde d_k}(\gamma)\leq a\leq f_{3d_k/2}(\gamma)\}$.

By Lemma \ref{lem-bifurcation}, for any $(\gamma,a)\in\Omega_{d_k}$, there exists non-shear steady flows  of the $\beta$-plane equation near the shear flow $(U_{\gamma,a}(y),0)$. For  fixed $0<\gamma<\delta(a_{2d_k})$, there exists $r_0>0$ (by compactness, independent of $a\in(f_{\tilde d_k}(\gamma), f_{3d_k/2}(\gamma))$ ) such that for any $0<r<r_0$, there exists a non-shear steady solution
\beno
(u_{\gamma,a;r}(x,y), v_{\gamma,a;r}(x,y)),
\eeno
 which has $x$-period $T(\gamma,a;r)$ and
\beno
\|(u_{\gamma,a;r}, v_{\gamma,a;r})-(U_{\gamma,a},0)\|_{H^3(D_{T(\gamma,a;r)})}\leq r,
\eeno
where $D_{T(\gamma,a;r)}= [0,T(\gamma,a;r)]\times [-1,1]$.
Moreover, for $a\in(f_{\tilde d_k}(\gamma), f_{3d_k/2}(\gamma))$, we have
\beno
\frac{2\pi}{T(\gamma,a;r)}\to \sqrt{-\lambda_{1,\gamma,a}} \quad \textrm{as} \quad r\to 0^+.
\eeno
By (\ref{lin18}), when $r_0$ is small enough,
\beno
T(\gamma,f_{3d_k/2}(\gamma);r)<{T\over k}<T(\gamma,f_{\tilde d_k}(\gamma);r) \quad \textrm{for}\quad 0<r<r_0.
\eeno
Since $T(\gamma,a;r)$ is continuous with respect to $a$ (by Crandall-Rabinowitz Theorem) for each $\gamma\in(0,\gamma_0)$ and $r>0$ small enough, there exists $a_T=a_T(\gamma,r)\in(f_{\tilde d_k}(\gamma), f_{3d_k/2}(\gamma))$ such that $T(\gamma,a_T;r)={T\over k}$. Then
\beno
(u_{\gamma;r}(x,y),v_{\gamma;r}(x,y)):=(u_{\gamma,a_T;r}(x,y),v_{\gamma,a_T;r}(x,y))
\eeno
is a non-shear steady solution to \eqref{eq}-\eqref{bc} with minimal $x$-period ${T\over k}$ (and thus $x$-period $T$), and
\beno
\|(u_{\gamma;r},v_{\gamma;r})-(U_{\gamma,a_T},0)\|_{H^3(D_T)}=\sqrt{k}\|(u_{\gamma;r},v_{\gamma;r})-(U_{\gamma,a_T},0)\|_{H^3(D_{{T\over k}})}\leq \sqrt{k}r.
\eeno
Thus, for any $0<r<{1\over \sqrt{k}}\min \{r_0,\frac{\varepsilon}{2}\}$, combining with (\ref{lin20}) we have
\beno
\|(u_{\gamma;r},v_{\gamma;r})-(y,0)\|_{H^s(D_T)}\leq \varepsilon.
\eeno

Next, we prove (2). We replace $C_\beta$ by $0$, and choose $k=1$    in (1). Let $d_1:=-\alpha^2<0$.
Set $a_{2d_1}=\frac{4d_1-6}{3b_0}+1$. Taking $d=2d_1$ in Lemmas \ref{prin-neg} (2) and \ref{lem-range}, there exists $\delta(a_{2d_1})>0$ small enough such that
 $\lambda_{1,\gamma,a_{2d_1}}<2d_1$
and $[2d_1,0]\subset  \{\lambda_{1,\gamma,a}:a\in[0,a_{2d_1}]\}$ for any $\gamma\in(0,\delta(a_{2d_1}))$.
The rest of the proof is a repeated process of  (1).
\end{proof}

\section{asymptotic stability of shear flows near Couette for $\beta$-plane equation on $\mathbb{T} \times \mathbb{R}$}
This is a generalization of Theorem 1 in \cite{BM}. We use the same notation, conventions, coordinate transform and time-dependent norm. So we omit the same details and only point out the differences. To avoid confusion with former sections, we use $w,\vartheta$ {to denote the vorticity perturbation and changed vertical coordinate} here, instead of $\omega,v$ in \cite{BM}, respectively.

The change of coordinates is $(t,x,y)\rightarrow(t,z,\vartheta)$, where
\beno
z(t,x,y)&=&x-t\vartheta,\\
\vartheta(t,y)&=&y+\frac{1}{t}\int_0^t \langle U^x\rangle(\tau,y)d\tau.
\eeno
Here, $\langle U^x\rangle=\frac{1}{2\pi}\int_{\mathbb{T}_{2\pi}} U^x dx$.
Define $f(t,z,\vartheta)=w(t,x,y)$ and $\phi(t,z,\vartheta)=\tilde\psi(t,x,y)$. We {denote} $[\pa_t \vartheta](t,\vartheta)=\pa_t \vartheta(t,y)$, $\vartheta'(t,\vartheta)=\pa_y \vartheta(t,y)$ and $\vartheta''(t,\vartheta)=\pa_{yy}\vartheta(t,y)$.
Then we get the evolution equation for $f$,
\beno
\pa_t f+[\pa_t \vartheta]\pa_\vartheta f+\pa_t z\pa_z f=-y\pa_z f+\vartheta'\pa_\vartheta \phi\pa_z f-\vartheta'\pa_z \phi\pa_\vartheta f-\beta\pa_z\phi.
\eeno
Notice that $\pa_t z=-y-\langle U^x \rangle (t,y)$. Using the Biot-Savart law, we can transform $\langle U^x\rangle$ to $-\vartheta'\pa_\vartheta\langle\phi\rangle$ in the new variables. Then the equation becomes
\beno
\pa_t f-(\vartheta'\pa_\vartheta(\phi-\langle\phi\rangle))\pa_z f+([\pa_t \vartheta]+\vartheta'\pa_z\phi)\pa_\vartheta f+\beta\pa_z\phi=0.
\eeno
The Biot-Savart law in the new variables reads
\beno
f=\pa_{zz}\phi+(\vartheta')^2(\pa_\vartheta-t\pa_z)^2\phi+\vartheta''(\pa_\vartheta-t\pa_z)\phi:=\Delta_t \phi.
\eeno
The $\beta$-plane equation (\ref{vor-eq}) becomes
\begin{eqnarray}
\label{eq1}
\left\{
\begin{aligned}
&\pa_t f+u\cdot\nabla_{z,\vartheta} f+\beta\pa_z \phi=0,\\
&u=(0,[\pa_t \vartheta])+\vartheta'\nabla^\perp_{z,\vartheta}P_{\neq 0}\phi,\\
&\phi=\Delta_t^{-1}{f}.
\end{aligned}
\right .
\end{eqnarray}
Without confusion we write $\nabla_{z,\vartheta}=\nabla$ in the following. Let $\tilde{u}(t,z,\vartheta)=U^x(t,x,y)$ and $p(t,z,\vartheta)=P(t,x,y)$. Then we have the equation for $\tilde{u}$, 
\beno
\pa_t\tilde{u}+[\pa_t \vartheta]\pa_\vartheta\tilde{u}+\pa_z P_{\neq 0}\phi+\vartheta'\nabla^\perp P_{\neq 0}\phi\cdot\nabla\tilde{u}=-\pa_z p+\beta y U^y.
\eeno
Isolating the zero mode of the velocity field by taking average in $z$, we have
\beno
\pa_t \tilde{u}_0+[\pa_t \vartheta]\pa_\vartheta\tilde{u}_0+\vartheta'\langle\nabla^\perp P_{\neq 0}\phi\cdot\nabla\tilde{u} \rangle=0.
\eeno
Here, the term $\beta y U^y$ brings nothing new since its average in $z$ is zero.
Finally, $\vartheta'$ and $[\pa_t \vartheta]$ are solutions to (2.13) in \cite{BM}
coupled to (\ref{eq1}).

Define the same time-dependent norm and main energy as in Subsection 2.3 of \cite{BM} by
\beno
\left\|A(t)f(t) \right\|_{L^2(\Omega)}^2&=&\sum_{k\in \mathbb{Z}}\int_\mathbb{R} |A_k(t,\eta)\hat{f}_k(t,\eta) |^2 d\eta,\\
A_k(t,\eta)&=&e^{\lambda(t)|k,\eta|^s}\langle k,\eta\rangle^\sigma J_k(t,\eta);\\
\tilde{A}_k(t,\eta)&=&e^{\lambda(t)|k,\eta|^s}\langle k,\eta\rangle^\sigma \tilde{J}_k(t,\eta);\\
E(t)&=&\frac{1}{2}\left\|A(t)f(t) \right\|_{L^2(\Omega)}^2+E_\vartheta(t).
\eeno
See (2.15)-(2.18) and (2.22) in \cite{BM} for more details of $\lambda(t)$, $J_k(t,\eta)$, $\tilde{J}_k(t,\eta)$ and $E_\vartheta(t)$.
The goal is to prove that the energy $E(t)$ is uniformly bounded for all time, as long as $\epsilon$ is small enough.
The local well-posedness theory for the $\beta$-plane equation in Gevrey spaces is similar to that for 2D Euler equation, so we have the same result as Lemma 2.1 in \cite{BM}, and thus we can focus on times $t\geq 1$.
Then we will prove the same bootstrap proposition as Proposition 2.1 in \cite{BM}, under the same bootstrap hypotheses for $t\geq 1$,\\
(B1) $E(t)\leq 4\epsilon^2$;\\
(B2) $\|\vartheta'-1\|_{L^\infty} \leq \frac{3}{4}$;\\
(B3) `CK' integral estimates (for `Cauchy-Kovalevskaya'):
\beno
\int_1^t \bigg[CK_\lambda(\tau)+CK_w(\tau)+CK_w^{\vartheta,2}(\tau)+CK_\lambda^{\vartheta,2}(\tau)
+K_\vartheta^{-1}\left(CK_w^{\vartheta,1}(\tau)+CK_\lambda^{\vartheta,1}(\tau)\right)\\
+K_\vartheta^{-1} \sum_{i=1}^2 \left(CCK_w^{i}(\tau)+CK_\lambda^{i}(\tau)\right) \bigg] d\tau \leq 8\epsilon^2.
\eeno
The definitions of $CK$ terms are same as (2.21), (2.29) and (2.31) in \cite{BM}.
Let $I_E$ be the connect set of times $t\geq 1$, on which the bootstrap hypotheses \textup{(B1-B3)} hold. To be rigorous, we only consider solutions with regularized initial data while calculating and finally perform a passage to the limit. So $E(t)$ is a continuous function of $t$, and thus $I_E$ is a closed interval $[1,T^*]$ with $T^*>1$. If we can prove that $I_E$ is also open (under the subspace topology of $[1,+\infty)$), then $I_E=[1,+\infty)$.

\begin{proposition}[Bootstrap]
There exists $\epsilon_0\in(0,\frac{1}{2})$ depending only on $\lambda_0,\lambda',s$ and $\sigma$ such that if $\epsilon<\epsilon_0$, and the bootstrap hypotheses \textup{(B1-B3)} hold on $[1,T^*]$, then for any $t\in [1,T^*]$,\\
$(1)$ $E(t)< 2\epsilon^2$;\\
$(2)$ $\|\vartheta'-1\|_{L^\infty} < \frac{5}{8}$;\\
$(3)$ `CK' controls satisfy:
\beno
\int_1^t \bigg[CK_\lambda(\tau)+CK_w(\tau)+CK_w^{\vartheta,2}(\tau)+CK_\lambda^{\vartheta,2}(\tau)
+K_\vartheta^{-1}\left(CK_w^{\vartheta,1}(\tau)+CK_\lambda^{\vartheta,1}(\tau)\right)\\
+K_\vartheta^{-1} \sum_{i=1}^2 \left(CCK_w^{i}(\tau)+CK_\lambda^{i}(\tau)\right) \bigg] d\tau \leq 6\epsilon^2.
\eeno
\end{proposition}
We prove the bootstrap proposition in the same way as \cite{BM}. The first step is to prove (2.19) in \cite{BM}. So it is natural to compute the derivative of $E(t)$. The first difference comes from
\beno
\frac{1}{2}\frac{d}{dt}\int_{\mathbb{T}_{2\pi}} |Af|^2 dx=-CK_\lambda-CK_w-\int_{\mathbb{T}_{2\pi}} AfA(u\cdot\nabla f)dx-\int_{\mathbb{T}_{2\pi}} AfA(\beta\pa_z\phi)dx.
\eeno
Compared with (2.20) in \cite{BM}, there is an additional term $-\int AfA(\beta\pa_z\phi)dx$ on the right hand side, so we need to control this term besides three old contributions: Transport, Reaction and Remainder (see (2.26) in \cite{BM} for details). Controls of old contributions have no difference with Propositions 2.2, 2.3 and 2.6 in \cite{BM}. For Proposition 2.4 in \cite{BM}, in fact their proof can get a more precise elliptic control (see the treatments of $T^1$ and $T^2$ in Subsection 4.2 of \cite{BM} for details), which will be useful in dealing with our new term.
\begin{proposition}[Precision elliptic control]
\label{Precision}  Under the bootstrap hypotheses \textup{(B1-B3)},
\beno
\begin{split}
\left\|\langle \frac{\pa_\vartheta}{t\pa_z}\rangle^{-1}\left(\pa_z^2+(\pa_\vartheta-t\pa_z)^2\right)\left(\frac{|\nabla|^\frac{s}{2}}{\langle t\rangle^s}A+
\sqrt{\frac{\pa_t w}{w}}\tilde{A}\right)\phi_1 \right\|_{L^2(\Omega)}^2\\
\lesssim \epsilon^2 CK_\lambda+\epsilon^2 CK_w+\epsilon^2\sum_{i=1}^2 (CCK_\lambda^i+CCK_w^i),
\end{split}
\eeno
where $\phi_1=P_{\neq 0}\phi-\Delta_L^{-1}P_{\neq 0}f$, $\Delta_L=\pa_z^2+(\pa_\vartheta-t\pa_z)^2$ and $P_{\neq 0}f=f-\langle f\rangle=f-\int_{\mathbb{T}_{2\pi}}f dx$.
\end{proposition}

Proposition 2.5 in \cite{BM} can be proved in the same way with a little change since equation (\ref{eq1}) is slightly different with {(2.11)} in \cite{BM}. Denote $h(t,\vartheta)=\vartheta'(t,\vartheta)-1$ and write
\ben\label{8.7}
\pa_t h+[\pa_t \vartheta]\pa_\vartheta h=\frac{1}{t}(-f_0-h)=\vartheta'\pa_\vartheta[\pa_t \vartheta]{:=\bar{h}(t,\vartheta)}.
\een
From (\ref{8.7}) and (\ref{eq1}), we derive
\ben\label{8.9}
\quad \pa_t \bar{h}=-\frac{\bar{h}}{t}-\frac{1}{t}(\pa_t f_0+\pa_t h)=-\frac{2}{t}\bar{h}-[\pa_t \vartheta]\pa_\vartheta \bar{h}+\frac{1}{t}\langle \vartheta'\nabla^\perp P_{\neq 0}\phi\cdot\nabla f\rangle {+\frac{1}{t}\langle \beta\pa_z\phi \rangle},
\een
which has nothing different with {(8.9)} in \cite{BM} since $\langle \beta\pa_z\phi\rangle=0$. So the same argument is valid.\\
The only thing left is to treat the new term $\int_{\mathbb{T}_{2\pi}} AfA(\beta\pa_z\phi)dx$. 
Since $\int_{\mathbb{T}_{2\pi}} AfA\pa_z\Delta_L^{-1}f dx=0$, we have
\beno
\int_{\mathbb{T}_{2\pi}} AfA(\beta\pa_z\phi)dx=\beta\int_{\mathbb{T}_{2\pi}} AfA(\pa_z\phi_1)dx.
\eeno
We get the following estimate to control this term.
\begin{proposition}[New term]\label{newterm}
 Under the bootstrap hypotheses \textup{(B1-B3)}, we have
\beno
\label{2.27}\begin{split}
\left| \int_{\mathbb{T}_{2\pi}} AfA\pa_z\phi_1 dx \right|\lesssim&\epsilon CK_\lambda+\epsilon CK_w+\frac{\epsilon^3}{\langle t\rangle^{2-K_D \epsilon/2}}+\epsilon CK_\lambda^{\vartheta,1}+\epsilon CK_w^{\vartheta,1}\\
&+\frac{1}{\epsilon} \left\|\langle \frac{\pa_\vartheta}{t\pa_z}\rangle^{-1}(\pa_z^2+(\pa_\vartheta-t\pa_z)^2) \left(\frac{|\nabla|^\frac{s}{2}}{\langle t\rangle^s}A+\sqrt{\frac{\pa_t w}{w}}\tilde{A}\right)\phi_1 \right\|_{L^2(\Omega)}^2.
\end{split}
\eeno
\end{proposition}
\begin{proof}
Divide the new term into two parts:
\beno
\int_{\mathbb{T}_{2\pi}} AfA\pa_z\phi_1 dx=\frac{1}{2\pi}\sum_{k\neq 0}\int_\mathbb{R}  A_k^2(t,\eta)\hat{f}_k(\eta)\overline{ik\widehat{\phi_1}_k(\eta)}(\chi^R+\chi^{NR})d\eta=R_1+R_2,
\eeno
where $\chi^R=I_{t\in I_{k,\eta}}$ and $\chi^{NR}=I_{t\notin I_{k,\eta}}$. For $R_1$ and $R_2$, we have
\beno
\begin{split}
|R_1|&\lesssim\sum_{k\neq 0} \int_\mathbb{R} A_k^2(t,\eta)\hat{f}_k(\eta) k \left|(\widehat{\phi_1})_k(\eta)\right| I_{t\in I_{k,\eta}}d\eta\\
&\lesssim\left\|\sqrt{\frac{\pa_t w}{w}}\tilde{A}P_{\neq 0}f\right\|_{L^2(\Omega)} \left\|\sqrt{\frac{w}{\pa_t w}}|\pa_z|\chi^R\tilde{A}\phi_1\right\|_{L^2(\Omega)},
\end{split}
\eeno
\beno
\begin{split}
|R_2|&\lesssim\sum_{k\neq 0} \int_\eta A_k^2(t,\eta)\hat{f}_k(\eta) k \left|(\widehat{\phi_1})_k(\eta)\right|I_{t\notin I_{k,\eta}} d\eta\\
&\lesssim \left\||\nabla|^{\frac{s}{2}}AP_{\neq 0}f \right\|_{L^2(\Omega)} \left\||\nabla|^{1-\frac{s}{2}}\chi^{NR}A\phi_1\right\|_{L^2(\Omega)}.
\end{split}
\eeno
By Cauchy-Schwartz inequality, we have
\beno
\begin{split}
\left|\int_{\mathbb{T}_{2\pi}} AfA\pa_z\phi_1 dx \right|
\lesssim&\frac{\epsilon}{\langle t\rangle^{2s}} \left\|\sqrt{\frac{\pa_t w}{w}}\tilde{A}P_{\neq 0}f \right\|_{L^2(\Omega)}^2 +
\frac{\langle t\rangle^{2s}}{\epsilon} \left\|\sqrt{\frac{w}{\pa_t w}}|\pa_z|\chi^R\tilde{A}\phi_1 \right\|_{L^2(\Omega)}^2\\
&+\epsilon \left\||\nabla|^{\frac{s}{2}}AP_{\neq 0}f \right\|_{L^2(\Omega)}^2+ \frac{1}{\epsilon}\left\| |\nabla|^{1-\frac{s}{2}}\chi^{NR}A\phi_1 \right\|_{L^2(\Omega)}^2.
\end{split}
\eeno
The conclusion then follows from the estimates
\ben\label{estimates-nonlinear damping1}
\langle t\rangle^{2s} \left\||\nabla|^{1-\frac{s}{2}}\chi^{NR}A\phi_1 \right\|_{L^2(\Omega)}^2&\lesssim&
\left\|\langle\frac{\pa_\vartheta}{t\pa_z}\rangle^{-1}\frac{|\nabla|^\frac{s}{2}}{\langle t\rangle^s}\Delta_L A\phi_1 \right\|_{L^2(\Omega)}^2,\\\label{estimates-nonlinear damping2}
\left\|\sqrt{\frac{w}{\pa_t w}}|\pa_z|\chi^R \tilde{A}\phi_1\right\|_{L^2(\Omega)}^2&\lesssim&
\left\|\langle\frac{\pa_\vartheta}{t\pa_z}\rangle^{-1}\sqrt{\frac{\pa_t w}{w}}\Delta_L \tilde{A}\phi_1 \right\|_{L^2(\Omega)}^2.
\een
Here, the proof of \eqref{estimates-nonlinear damping1}-\eqref{estimates-nonlinear damping2} is similar to that of  Lemma 6.1 in \cite{BM}, and we thus omit it.
\end{proof}
Now, we give the proof of Theorem \ref{thm-damping}.
\begin{proof}[Proof of Theorem \ref{thm-damping}]
Controls of Transport, Reaction and Remainder contributions are the same as Propositions 2.2, 2.3 and 2.6 in \cite{BM}.
By Propositions \ref{Precision} and \ref{newterm}, we have
\beno
\label{2.27}\begin{split}
\left| \int_{\mathbb{T}_{2\pi}} AfA\pa_z\phi_1 dx \right|\lesssim&\epsilon CK_\lambda+\epsilon CK_w+\frac{\epsilon^3}{\langle t\rangle^{2-K_D \epsilon/2}}+\epsilon CK_\lambda^{\vartheta,1}+\epsilon CK_w^{\vartheta,1}\\
& +\epsilon\sum_{i=1}^2 (CCK_\lambda^i+CCK_w^i),
\end{split}
\eeno
which shows that the new term has similar estimates with Reaction, and thus can be absorbed into Reaction. The method in the  proof of Proposition 2.1 in \cite{BM} is still valid. Meanwhile, the change of coordinates cause no difference due to $\langle \beta\pa_z\phi\rangle=0$ in (\ref{8.9}). Therefore, Theorem 1 in \cite{BM} is also true for the $\beta$-plane equation on $\Omega=\mathbb{T}_{2\pi}\times \mathbb{R}$.
\end{proof}

\section*{Acknowledgement}
The authors express their gratitude to Profs. Zhiwu Lin and Dongyi Wei for helpful discussions.
Z. Zhang is partially supported by NSF
of China under Grant 12171010. H. Zhu is  partially supported by National Key Research and Development Program of China under Grant 2021YFA1002400, NSF
of China under Grant 12101306 and NSF of Jiangsu Province, China under Grant BK20210169.

\end{CJK*}
\end{document}